\documentclass[11pt,reqno,tbtags]{amsart}
\usepackage{amssymb}
\usepackage{url}
\usepackage{ifdraft}
\usepackage[square,numbers]{natbib}

\numberwithin{equation}{section}

\allowdisplaybreaks

\newtheorem{theorem}{Theorem}[section]
\newtheorem{lemma}[theorem]{Lemma}

\newtheorem{corollary}[theorem]{Corollary}

\theoremstyle{definition}
\newtheorem{example}[theorem]{Example}
\newtheorem{definition}[theorem]{Definition}
\newtheorem{problem}[theorem]{Problem}
\newtheorem{remark}[theorem]{Remark}
\newtheorem*{warning}{Warning}

\theoremstyle{remark}

\newenvironment{romenumerate}[1][0pt]{
\addtolength{\leftmargini}{#1}\begin{enumerate}
 }{\end{enumerate}}

\newcounter{oldenumi}
\newenvironment{romenumerateq}
{\setcounter{oldenumi}{\value{enumi}}
\begin{romenumerate} \setcounter{enumi}{\value{oldenumi}}}
{\end{romenumerate}}

\newcounter{thmenumerate}
\newenvironment{thmenumerate}
{\setcounter{thmenumerate}{0}%
 \def\item{\par
 \refstepcounter{thmenumerate}\textup{(\roman{thmenumerate})\enspace}}
}
{}

\newcounter{romxenumerate}   

\newcounter{xenumerate}   

\newcommand\pfitem[1]{\par(#1):}
\newcommand\pfitemref[1]{\par\ref{#1}:}

\newcommand{\refT}[1]{Theorem~\ref{#1}}
\newcommand{\refC}[1]{Corollary~\ref{#1}}
\newcommand{\refL}[1]{Lemma~\ref{#1}}
\newcommand{\refR}[1]{Remark~\ref{#1}}
\newcommand{\refS}[1]{Section~\ref{#1}}
\newcommand{\refSS}[1]{Subsection~\ref{#1}}

\newcommand{\refD}[1]{Definition~\ref{#1}}
\newcommand{\refE}[1]{Example~\ref{#1}}

\newcommand{\refApp}[1]{Appendix~\ref{#1}}

\newcommand{\refand}[2]{\ref{#1} and~\ref{#2}}

\begingroup
  \divide\count255 by 60
  \multiply\count255 by -60
  \advance\count255 by \time
  \ifnum \count255 < 10 \xdef\klockan{\the\count1.0\the\count255}
  \else\xdef\klockan{\the\count1.\the\count255}\fi
\endgroup

\newcommand\nopf{\qed}   

\newcommand{\sumi}{\sum_{i=1}^\infty}

\newcommand{\sumni}{\sum_{n=1}^\infty}
\newcommand{\sumin}{\sum_{i=1}^n}
\newcommand{\sumjn}{\sum_{j=1}^n}

\newcommand\set[1]{\ensuremath{\{#1\}}}
\newcommand\bigset[1]{\ensuremath{\bigl\{#1\bigr\}}}

\newcommand\xpar[1]{(#1)}
\newcommand\bigpar[1]{\bigl(#1\bigr)}
\newcommand\Bigpar[1]{\Bigl(#1\Bigr)}

\newcommand\lrpar[1]{\left(#1\right)}

\newcommand\Bigsqpar[1]{\Bigl[#1\Bigr]}

\newcommand\xcpar[1]{\{#1\}}
\newcommand\bigcpar[1]{\bigl\{#1\bigr\}}
\newcommand\Bigcpar[1]{\Bigl\{#1\Bigr\}}

\newcommand\abs[1]{|#1|}
\newcommand\bigabs[1]{\bigl|#1\bigr|}
\newcommand\Bigabs[1]{\Bigl|#1\Bigr|}
\newcommand\biggabs[1]{\biggl|#1\biggr|}
\newcommand\lrabs[1]{\left|#1\right|}
\def\rompar(#1){\textup(#1\textup)}    

\def\xexp(#1){e^{#1}}
\newcommand\ceil[1]{\lceil#1\rceil}

\newcommand\setn{\set{1,\dots,n}}
\newcommand\ntoo{\ensuremath{{n\to\infty}}}

\newcommand\norm[1]{\|#1\|}

\newcommand\Bignorm[1]{\Bigl\|#1\Bigr\|}
\newcommand\downto{\searrow}
\newcommand\upto{\nearrow}

\newcommand\punkt[1]{\if.#1\else.\spacefactor1000\fi{#1}}
\newcommand\iid{i.i.d\punkt}    
\newcommand\ie{i.e\punkt}
\newcommand\eg{e.g\punkt}
\newcommand\viz{viz\punkt}
\newcommand\cf{cf\punkt}
\newcommand{\as}{a.s\punkt}
\newcommand{\aex}{a.e\punkt}

\newcommand\ii{\mathrm{i}}

\newcommand{\tend}{\longrightarrow}
\newcommand\dto{\overset{\mathrm{d}}{\tend}}

\newcommand\eqd{\overset{\mathrm{d}}{=}}

\newcommand\bbR{\mathbb R}
\newcommand\bbC{\mathbb C}
\newcommand\bbN{\mathbb N}

\newcommand\bbQ{\mathbb Q}

\newcommand\bbF{\mathbb F}

\newcounter{CC}
\newcounter{cc}

\newcommand\E{\operatorname{\mathbb E{}}}
\renewcommand\P{\operatorname{\mathbb P{}}}

\newcommand\Bi{\operatorname{Bi}}

\newcommand\Be{\operatorname{Be}}

\newcommand\sgn{\operatorname{sgn}}
\newcommand\Tr{\operatorname{Tr}}
\newcommand\supp{\operatorname{supp}}

\newcommand\ga{\alpha}

\newcommand\gd{\delta}

\newcommand\gf{\varphi}
\newcommand\gam{\gamma}
\newcommand\gG{\Gamma}

\newcommand\gl{\lambda}

\newcommand\gO{\Omega}
\newcommand\gs{\sigma}
\newcommand\eps{\varepsilon}

\renewcommand\phi{\xxx}  

\newcommand\cA{\mathcal A}
\newcommand\cB{\mathcal B}
\newcommand\cC{\mathcal C}
\newcommand\cD{\mathcal D}
\newcommand\cE{\mathcal E}
\newcommand\cF{\mathcal F}

\newcommand\cH{\mathcal H}

\newcommand\cK{\mathcal K}
\newcommand\cL{{\mathcal L}}

\newcommand\cP{\mathcal P}

\newcommand\cS{{\mathcal S}}

\newcommand\cW{\mathcal W}
\newcommand\cX{{\mathcal X}}

\newcommand\ett[1]{\boldsymbol1\xcpar{#1}} 
\newcommand\bigett[1]{\boldsymbol1\bigcpar{#1}} 
\newcommand\Bigett[1]{\boldsymbol1\Bigcpar{#1}} 
\newcommand\etta{\boldsymbol1}

\newcommand\smatrixx[1]{\left(\begin{smallmatrix}#1\end{smallmatrix}\right)}

\newcommand\qw{^{-1}}
\newcommand\qww{^{-2}}
\newcommand\qq{^{1/2}}
\newcommand\qqw{^{-1/2}}
\newcommand\qqc{^{3/2}}

\newcommand\qqqq{^{1/4}}

\renewcommand{\=}{:=}

\newcommand\intoi{\int_0^1}

\newcommand\oi{[0,1]}
\newcommand\ooi{(0,1]}
\newcommand\oiq{[0,1]^2}
\newcommand\ooo{[0,\infty)}

\newcommand\setoi{\set{0,1}}

\newcommand\dd{\,\textup{d}}

\newcommand\lhs{left-hand side}
\newcommand\rhs{right-hand side}

\newcommand\gnw{\ensuremath{G(n,W)}}

\newcommand\norml[1]{\norm{#1}_{1}}
\newcommand\normlx[2]{\norm{#2}_{1,#1}}
\newcommand\normll[1]{\norm{#1}_{L^1}}
\newcommand\normlss[1]{\norm{#1}\qliss}

\newcommand\normoo[1]{\norm{#1}_{\infty}}

\newcommand\cn[1]{\norm{#1}\cut}
\newcommand\cnx[2]{\norm{#2}_{\square,#1}}
\newcommand\cntwo[1]{\norm{#1}_{\square,2}}
\newcommand\cnone[1]{\norm{#1}_{\square,1}}
\newcommand\cnc[1]{\cnx{\bbC}{#1}}
\newcommand\cnh[1]{\cnx{\mathsf H}{#1}}
\newcommand\cutnorm{\cn}

\newcommand\cutnormtwo{\cntwo}
\newcommand\cut{_{\square}}
\newcommand\dcut{{\delta_{\square}}}
\newcommand\dcutone{{\delta_{\square,1}}}
\newcommand\dcuttwo{{\delta_{\square,2}}}
\newcommand\dl{{\delta_{1}}}

\newcommand\sss{{\Omega}}
\newcommand\sssq{{\sss^2}}
\newcommand\sssxy{{\sss_1\times\sss_2}}
\newcommand\sssxyz{{\sss_1\times\sss_2\times\sss_3}}
\newcommand\sssxz{{\sss_1\times\sss_3}}

\newcommand\liss{{L^1(\sss^2)}}
\newcommand\lis{{L^1(\sss)}}

\newcommand\qliss{_{L^1(\sss^2)}}
\newcommand\qlis{_{L^1(\sss)}}

\newcommand\www{\cW}
\newcommand\ints{\int_\sss}
\newcommand\intsq{\int_\sssq}
\newcommand\mx[2]{\overline{#1}^{(#2)}}
\newcommand\mi[1]{\mx{#1}{1}}
\newcommand\mii[1]{\mx{#1}{2}}
\newcommand\mpp{measure-preserving}
\newcommand\mpb{\mpp{} bijection}
\newcommand\pmm{probability measure}
\newcommand{\ps}{probability space}
\newcommand\gff{\gf\otimes\gf}
\newcommand\cWx{\cW^*}
\newcommand\innprod[1]{\langle#1\rangle}
\newcommand\wgv[1]{W^{\mathsf V}_{#1}}
\newcommand\wgi[1]{W^{\mathsf I}_{#1}}
\newcommand\cwm{\cW_{\mathsf m}}
\newcommand\equ{\cong} 
\newcommand\tmu{\widetilde\mu}
\newcommand\tgs{\widetilde\gs}
\newcommand\tW{\widetilde W}
\newcommand\dref[1]{\drefx{\ref{#1}}}
\newcommand\drefx[1]{\gd_{{#1}}}
\newcommand\casl{\cA(\sss_\ell^2)}
\newcommand\casll{\cA(\sss_{\ell_1}\times\sss_{\ell_2})}
\newcommand\cfo{\cF_0}
\newcommand\oivalued{\setoi-valued}  
\newcommand\comp{^{\mathsf c}}
\newcommand\aaa[1]{a^{(#1)}}
\newcommand\bwx{\widehat {\cW}}
\newcommand{\tWx}[1]{W_{#1}^{\gam_{#1}}}
\newcommand\dt{\gd_{\mathsf t}}

\newcommand\hcF{\widehat{\cF}}
\newcommand\xiij{\xi_{ij}}
\newcommand\ent{\cE}
\newcommand\fyll{\;}
\newcommand\txi{\tilde\xi}
\newcommand\teta{\tilde\eta}
\newcommand\lsmu{L^1(\sss,\mu)}
\newcommand\lsfmu{L^1(\sss,\cF,\mu)}
\newcommand\psiw{\psi_W}
\newcommand\psihw{\psi_{\hW}}
\newcommand\muw{\mu_W}
\newcommand\sssw{\sss_W}
\newcommand\sssmuw{(\sssw,\muw)}
\newcommand\cfw{\cF_W}
\newcommand\hW{\widehat W}

\newcommand\psixqw{(\psiw^*)\qw}
\newcommand\bh{\overline H}
\newcommand\rww{r_{W\circ W}}
\newcommand\rw{r_{W}}
\newcommand\FF{\bar F}
\newcommand\bbNoo{\overline{\bbN}}
\newcommand\ssswb{\overline{\sssw}^\gs}
\newcommand\sssbx[1]{\overline{\sss_{#1}}^\gs}
\newcommand\bG{\overline G}
\newcommand\cXo{\widehat{\cX}}

\newcommand{\Holder}{H\"older}
\newcommand\CS{Cauchy--Schwarz}
\newcommand\CSineq{\CS{} inequality}

\newcommand{\Lovasz}{Lov\'asz}

\newcommand\REM[1]{{\raggedright\texttt{[#1]}\par\marginal{XXX}}}

\newenvironment{comment}{\setbox0=\vbox\bgroup}{\egroup} 

\newcommand\urladdrx[1]{{\urladdr{\def~{{\tiny$\sim$}}#1}}}

\begin{document}
\title[Graphons, cut norm  and distance]
{Graphons, cut norm and distance, couplings and rearrangements}

\date{10 September, 2010; 
revised 3 June 2011}

\author{Svante Janson}
\address{Department of Mathematics, Uppsala University, PO Box 480,
SE-751~06 Uppsala, Sweden}
\email{svante.janson@math.uu.se}
\urladdrx{http://www2.math.uu.se/~svante/}


\begin{comment}  
05 Combinatorics 
05C Graph theory [For applications of graphs, see 68R10, 90C35, 94C15]
05C05 Trees
05C07 Vertex degrees
05C35 Extremal problems [See also 90C35]
05C40 Connectivity
05C65 Hypergraphs
05C80 Random graphs
05C90 Applications
05C99 None of the above, but in this section

28   (1940-now) Measure and integration
28A   (1973-now) Classical measure theory
28A20 (1973-now) Measurable and nonmeasurable functions, sequences of
      measurable functions, modes of convergence 
\end{comment}

\begin{abstract} 
We give a survey of basic results on the cut norm and cut metric for 
graphons (and sometimes more general kernels), with emphasis on the
equivalence problem. The main results are not new, but we add various
technical complements,
and a new proof of the uniqueness theorem by Borgs, Chayes 
and Lov\'asz. 
We allow graphons on general probability spaces
whenever possible. 
We also give some new results for \{0,1\}-valued graphons and for pure
graphons.
\end{abstract}

\maketitle

\section{Introduction}\label{S:intro}

In the recent theory of \emph{graph limits}, introduced by
  \citet{LSz} and 
further developed by \eg{} \citet{BCLSV1,BCLSV2}, 
a prominent role is played by \emph{graphons}.
These are symmetric measurable functions 
$W:\sss^2\to\oi$, where, in general, $\sss$ is an arbitrary \ps.
The basic fact is that  every graph  limit can be represented by a
graphon (where we further may choose $\sss=\oi$ if we like);
however, such representations of graph limits 
are far from unique, see \eg, \cite{BRmetrics,BCL:unique,BCLSV1,SJ209,LSz}.
(This representation is essentially equivalent to the representation by
Aldous and Hoover of exchangeable arrays of random variables,
see \cite{Kallenberg:symmetries} for details of this representation
and \cite{Austin,SJ209} for the connection, which is summarized in
\refApp{Arg}.)  
See \refApp{Alimits} for a very brief summary.

It turns out that for studying both convergence and equivalence of gra\-ph\-ons,
a key tool is the \emph{cut metric} \cite{BCLSV1}.
The purpose of this paper is to give a survey over basic, and often
elementary, facts on the cut norm and cut metric.
Most results in this paper are not new,
even when we do not give a specific reference. (Most results are in at least
one of 
\cite{BRmetrics,BCL:unique,BCLSV1,SJ209,LSz}.)   
However, the results are sometimes difficult to find 
in the literature, since they are spread out over several papers, 
with somewhat different versions of the definitions and assumptions;
moreover, some elementary results have only been given implicitly and without
proof before.
Hence we try to collect the results and proofs here, and state them 
in as general forms as we find convenient. 
For example, we allow
general \ps{s} whenever possible. 
We thus add various technical complements to previous results.
We also give some new results, including some results on \oivalued{}
graphons in \refS{Srf}, and some results on pure graphons leading to a new
proof of 
the uniqueness theorem by \citet{BCL:unique} in \refS{Scanon}.

We include below for convenience some standard facts from measure theory,
sometimes repeating standard arguments.
Some general references (from different points of view)
are \cite{Billingsley,Cohn,Kallenberg,Parthasarathy}.

\begin{remark}
  The basic idea of graph limits has been generalized to limits of many other
finite combinatorial objects such as weighted graphs,
directed graphs, multigraphs, bipartite graphs,
hypergraphs, posets and permutations, see for example
\cite{Austin,BCLSV1,BCLSV2,SJ209,ElekSz,Hoppen,SJ224,KR,LSz:topology,LSz:compact}. Many
results below extend in a straightforward way to such extensions, but for
simplicity we leave such extensions to the reader and concentrate on the
standard case.  
\end{remark}

\section{The setting}

Let $(\sss,\cF,\mu)$ be a probability space.
(We will usually denote this space simply by $\sss$ or $(\sss,\mu)$, 
with $\cF$ and perhaps $\mu$ being clear from the context.)
Often we 
take $\sss$ to be $[0,1]$ (or $(0,1]$) with $\mu=\gl$, the Lebesgue measure;
this is 
sometimes convenient, and it is often possible to reduce to this case; in
fact, in several papers on graph limits only this case is considered for
convenience. (See \cite{SJ210} for a general representation theorem.)
However, it is also often convenient to consider other $\sss$, and we will here
be general and allow arbitrary probability spaces. 

Nevertheless, we will often consider $\oi$ or $\ooi$. 
Except when we explicitly say otherwise, we will always assume that these
spaces are equipped 
with the Borel $\gs$-field $\cB$ and the Lebesgue measure, which we denote by
$\gl$. 
(We denote the Lebesgue $\gs$-field by $\cL$; we will occasionally use it
instead of $\cB$, but not without saying so. Recall that $\cL$ is the
completion of $\cB$, see \eg{} \cite{Cohn}.)
\begin{remark}\label{RBL}
Our default use of $\cB$ is important when we consider mappings into $\oi$,
but for functions defined on $\oi$ or $\oiq$, it 
often does not matter whether we use $\cB$ or $\cL$, since every
$\cL$-measurable function is \aex{} equal to a $\cB$-measurable one.
In fact, it is sometimes more convenient to use $\cL$.
\end{remark}

In a few cases, we will need some technical assumptions on $\sss$. We refer
to \refApp{Aspaces} for the definitions of \emph{atomless}, \emph{Borel} and
\emph{Lebesgue} probability spaces.

We will study functions on $\sssq$, and various (semi)metrics on such
functions. 
Of course, $\sssq$ is itself a probability space, equipped with the product
measure $\mu^2\=\mu\times\mu$ and the product $\gs$-field (or its
completion; this 
makes no difference for our purposes).

\begin{remark}
  The definitions and many results can be extended to functions of $\sss^r$ for
  arbitrary $r\ge2$, which is the setting for hypergraph limits; see \eg{}
  \cite{SJ223} and \cite{ElekSz}.
\end{remark}

All subsets and all functions on $\sss$ or $\sssq$ 
that we consider will tacitly be assumed to be measurable. 
We will usually identify functions that are \aex{} equal.
This also means that functions
only have to be defined \aex.
(In particular, this means that it does not make any significant difference
if we replace $\cF$ by its completion; for example, on $\oi$ and $\oiq$,
with Lebesgue measure,
it does not
matter whether we consider Borel or Lebesgue measurable functions, \cf{}
\refR{RBL}. 
Moreover, in this case it does not matter whether we take 
$\oi$, $(0,1]$ or $(0,1)$.)

The natural domain of definition for the various metrics we consider is
$L^1(\sssq)$, but we are really mainly interested in some subclasses.
\begin{definition}
A \emph{kernel} on $\sss$ is
an integrable, symmetric function $W:\sss^2\to \ooo$. 

A \emph{standard kernel} or \emph{graphon} on $\sss$ is
a (measurable) symmetric function $W:\sss^2\to \oi$. 
 
We let $\www=\www(\sss)$ denote the set of all graphons on a given $\sss$.  
\end{definition}

We are mainly interested in the graphons (standard kernels), since they
correspond to graph limits. We use kernels when we find it more natural to
state results in this generality, but we will often consider just graphons for
convenience, leaving possible extensions to the reader.

\begin{warning}
The terminology varies between different authors and papers. \emph{Kernel}
and \emph{graphon} are used more or less interchangeably, with somewhat
different definitions in different papers.  
(This includes my own papers, where again there is no consistency.)
Apart from the two cases in the definition above, one sometimes considers
the intermediate case of
arbitrary bounded symmetric functions $\gO^2\to\ooo$. Moreover, sometimes
one considers $W$ with arbitrary values in $\bbR$, and not just $W\ge0$; 
for simplicity, we will not consider this case here. 
(Extensions to these cases are typically straight-forward when they are
possible.) 
\end{warning}

\begin{remark}
  For consistency we here require $W$ to be measurable for the
  product $\gs$-field $\cF\times\cF$, but it makes no essential difference if
  we only require $W$ to be measurable for the completion of $\cF\times\cF$,
  since every kernel of the latter type is \aex{} equal to an
  $\cF\times\cF$-measurable kernel. 
\end{remark}

\begin{remark}
  \label{RBL2}
A kernel is said to be \emph{Borel} if it is 
defined on a Borel space, and 
\emph{Lebesguian}  if it is defined on a Lebesgue space, see
\refApp{Aspaces} for definitions. We sometimes have to restrict to such
special kernels (which include all common examples). Note that the
difference between Borel and Lebesguian kernels is very minor: A Lebesgue
probability space is the same as the completion of a Borel probability
space. Hence, if $W$ is a Borel kernel
defined on some (Borel) space $(\sss,\cF,\mu)$, then
$W$ can also be regarded as a Lebesguian kernel defined on
$(\sss,\widehat\cF,\mu)$, where $\widehat\cF$ is the completion of $\cF$
(for $\mu$). Conversely, if $W$ is a Lesbeguian kernel defined on
$(\sss,\cF,\mu)$, then $\cF$ is the completion of a sub-$\gs$-field $\cF_0$
such that $(\sss,\cF_0,\mu)$ is a Borel space. Hence $W$ is \aex{} equal to
some $\cF_0\times\cF_0$-measurable function  $W_0$, which we may be assume
to be symmetric and with values in $\oi$; thus $W=W_0$ \aex{} where $W_0$ is a
Borel kernel. Consequently, up to \aex{} equivalence, the classes of Borel
and Lebesgue kernels are the same, and it is a matter of taste which
version we choose when we introduce one of these restrictions.
Cf.\ \refR{RBL}.
\end{remark}

\begin{remark}\label{Rbi}
  The definitions and results can be extended to the non-sym\-metric case,
  considering instead of $\cW(\sss)$ the set of arbitrary (measurable) functions
  $\sssq\to\oi$ or, more generally, $\sssxy\to\oi$. 
Such functions (\emph{bigraphons})
appear in the graph limit theory for
bipartite graphs, see \eg{} \cite{SJ209} and \cite{LSz:topology}.
\end{remark}

\begin{example}\label{Ewg1}
  Let $G$ be a (simple, undirected) graph. Then $G$ defines naturally a
  graphon $W_G$, which forms a link between graphs and graphons and is
  central in the graph limit theory, see \eg{} \cite{BCLSV1}. 
In fact, there are two natural
  versions, which we denote by $\wgv G$ and $\wgi G$.

For the first version, we regard the vertex set $V$ of $G$ as a probability
space with each vertex having equal probability $1/|G|$.
We define the graphon $\wgv G:V^2\to\oi$ on this probability space by
\begin{equation}\label{w1g}
\wgv G(u,v)=
\begin{cases}
1 &\text{if $u$ and $v$ are adjacent}, 
\\
0 & \text{otherwise}.  
\end{cases}
\end{equation}
In other words, $\wgv G$ equals (up to notation) the adjacency matrix of $G$.

For the second version we choose the probability space $\sss=(0,1]$.
Let $n\=|G|$ and partition $\ooi$ into $n$ intervals
  $I_{in}\=(\frac{i-1}n,\frac in]$. We assume that the vertices of $G$ are
	labelled $1,\dots,n$ (or, equivalently, that $V=\set{1,\dots,n}$),
and define
\begin{equation}\label{w2g}
\wgi G(x,y)\=
\wgv G(i,j)
\quad
\text{if $x\in I_{in}$, $y\in I_{jn}$}.
\end{equation}
  
The graphons $\wgv G$ and $\wgi G$ 
are equivalent in the sense defined below, see \refE{Ewg2}. Usually it does
not matter which version we choose, and we let $W_G$ denote any of them when
the choice is irrelevant.
\end{example}

\section{Step functions}
Recall that a function $f$ on $\sss$ is \emph{simple} or a \emph{step
  function} if there is a finite partition $\sss=\bigcup_{i=1}^n A_i$ of
  $\sss$ such that $f$ is constant on each $A_i$.
Similarly, we say that a function $W$ on $\sssq$ is a
  \emph{step   function} 
if there is a finite partition $\sss=\bigcup_{i=1}^n A_i$ of
  $\sss$ such that $W$ is constant on each $A_i\times A_j$.
Step functions are also said to be of \emph{finite type}.
If $W$ is a kernel or graphon that also is a step function, 
we call it a \emph{step kernel} or \emph{step graphon}.

When necessary, we may be more specific and say, for example, 
that $W$ is a \emph{$\cP$-step function}, 
where $\cP$ is the partition \set{A_i} above,
or an \emph{$n$-step function}, when the number of parts $A_i$ is (at most)
$n$.  

Step kernels (and graphons) are important mainly as a technical tool, see
several proofs below. 
However, they can also be studied for their own sake;
see \citet{LSos}, which can be seen as a study of
step graphons, although the results are stated in terms of
the corresponding graph limits and convergent sequences of graphs.

\begin{remark}\label{Rstep}
Note that being a step function on $\sssq$ is stronger than being a simple
function on that space, which means constant on the sets of some arbitrary
  partition of $\sssq$; it is important that we use product sets in the
  definition of a step function on $\sssq$. See also \refE{Estep} below.
\end{remark}

\begin{warning}
Some authors use different terminology. For example,
when studying functions on $\oi$, step functions are sometimes defined as
functions constant on some finite set of intervals partitioning $\oi$, \ie,
the parts $A_i$ are required to be intervals. We make no such assumption.  
\end{warning}

\section{The cut norm}

For functions in $L^1(\sssq)$ we have the usual $L^1$ norm 
\begin{equation}
\norml{W}\=\int_\sssq|W|\dd\mu^2   
\end{equation}
and the corresponding metric $\norml{W_1-W_2}$.

For the graph limit theory, it turns out that another norm is more important.
This is the \emph{cut norm} $\cn{W}$ of $W$, which was
introduced for a different purpose by \citet{FKquick}, and given a central
role in the graph limit theory by \citet{BCLSV1}.
(Its history actually goes back much further. 
For functions on $\oiq$,
the version in \eqref{cutnorm2}
is
the same as the \emph{Fr\'echet variation} 
of the corresponding distribution function $F(x,y)\=\int_0^x\int_0^y W$, see
\citet{Frechet}; more generally, 
$\cntwo{W}$ equals the Fr\'echet variation of the bimeasure on $\sssq$
corresponding to $W$.
See further \eg{} \citet{Littlewood} (where also the discrete version is
considered), \citet{ClA} and \citet{Morse}, and in particular
\citet{Blei} with further references.)  

There are several versions of the cut norm, equivalent
within constant factors. 
Following \cite{FKquick} and \cite{BCLSV1}, 
for $W\in L^1(\sss^2)$ we define
\begin{equation}\label{cutnorm1}
 \cnone W \= \sup_{S,T}
  \Bigabs{\int_{S \times T} W(x,y) \dd\mu(x)\dd\mu(y)},
\end{equation}
where the supremum is taken over all pairs of measurable subsets of $\sss$.
Alternatively, one can take
\begin{equation}\label{cutnorm2}
 \cntwo W \= \sup_{\normoo{f},\normoo{g}\le1}
  \Bigabs{\int_{\sss^2} W(x,y)f(x)g(y)\dd\mu(x)\dd\mu(y)},
\end{equation}
taking the supremum over all (real-valued) functions $f$ and $g$ with values
in $[-1,1]$. (We let $\normoo f$ denote the norm in $L^\infty$ of $f$, \ie, the
essential supremum of $|f|$.)
It is easily seen that in taking the supremum in \eqref{cutnorm2} one can
restrict to functions $f$ and $g$ taking only the values $\pm 1$.
Note that \eqref{cutnorm1} is equivalent to 
\eqref{cutnorm2} with the supremum taken over only $f$ and $g$ with values in
$\setoi$ (\ie, indicator functions);
it follows that
\begin{equation}\label{cn1=2}
 \cnone{W} \le \cntwo{W} \le 4\cnone{W}. 
\end{equation}
Thus the two norms $\cnone{\cdot}$ and $\cntwo{\cdot}$ are equivalent,
and it will almost never matter which one we use.
We shall write $\cn{\cdot}$ for either norm, when the choice of definition
does not matter.
For further, equivalent, versions of the cut norm, see
\refApp{Acut}.

We usually do not indicate $\sss$ or $\mu$ explicitly in the notation; when
necessary we may add them as subscripts and write, for example, 
$\cnx{\sss,\mu}\cdot$ or  $\cnx{\sss,\mu,1}\cdot$.

\begin{remark}\label{Rcn1}
Similarly, it is easily seen that
\eqref{cutnorm1} is equivalent to 
\eqref{cutnorm2} with the supremum taken over only $f$ and $g$ with values in
$\oi$.
\end{remark}

One advantage of the version $\cntwo{\cdot}$
is the simple ``Banach module'' property:
For any bounded functions $h$
and $k$ on $\sss$,
\begin{equation}
\cntwo{h(x)k(y)W(x,y)} \le \normoo{h}\normoo{k}\cntwo{W}.
\end{equation}
A similar advantage is seen in \refL{L0} below. (In both cases, using
$\cnone\cdot$ would introduce some constants.)
On the other hand, $\cnone{\cdot}$ is perhaps more natural, and probably
more familiar, in combinatorics.

Note that for either definition of the cut norm we have
\begin{equation}
 \Bigabs{\int_\sssq W} \le \cn{W} \le \norml{W}. 
\end{equation}

\begin{remark}\label{Rtensor}
The definition \eqref{cutnorm2} is natural for a functional analyst.
This norm is the dual of the projective tensor product norm in
$L^\infty(\sss)\hat\otimes L^\infty(\sss)$, and is thus the injective tensor
product norm in $L^1(\sss)\check\otimes L^1(\sss)$; equivalently, it is
equal to the  
operator norm of the corresponding integral operator $L^\infty(\sss)\to
L^1(\sss)$. 
This contrasts nicely to the $L^1$ norm on $\sssq$, which is the projective
tensor product norm in $L^1(\sss)\hat\otimes L^1(\sss)$. (See \eg{}
\cite{Treves}.) 
  \end{remark}

\begin{remark}
  We may similarly define the cut norm of functions defined on a product of
  two different spaces.
\end{remark}

\begin{remark}\label{Rcut1}
  The one-dimensional version of the cut norm coincides with the $L^1$
  norm. This is exact for $\cntwo\cdot$:
If $f$ is any integrable function of $\sss$, then
  \begin{equation}\label{rcut1}
\norml{f} = 
\sup_{\normoo{g}\le1}
  \Bigabs{\int_{\sss} f(x)g(x)\dd\mu(x)}.
  \end{equation}
\end{remark}
For the one-dimensional version of $\cnone\cdot$, we may in analogy with
\eqref{cn1=2} lose a factor 2; we omit the details.

We define the \emph{marginals} of a function $W\in L^1(\sssq)$ by
\begin{align}
 \mi W(x)&\=\ints W(x,y)\dd\mu(y),\\
 \mii W(y)&\=\ints W(x,y)\dd\mu(x). 
\end{align}

It a well-known consequence of Fubini's theorem that
$\norm{\mi W}\qlis\le\norm{W}\qliss$ for any $W\in\liss$.
This extends to the cut norm on $\sssq$, even though this norm is weaker.
This is stated in the next lemma, which can be seen as a consequence of
\refR{Rcut1} and the fact taking marginals (in any product, and in any
dimension) does not increase the cut norm.  

\begin{lemma}\label{L0}
If\/ $W\in\liss$, then 
$\norm{\mi W}\qlis, \norm{\mii W}\qlis
\le\cutnormtwo{W}$.
\end{lemma}

\begin{proof}
By symmetry, it suffices to consider $\mi W$.
If  $f\in L^\infty(\sss)$, then
\begin{equation*}
  \int_\sss\mi W(x)f(x)\dd\mu(x)
=\int_{\sss^2} W(x,y)f(x)\dd\mu(x)\dd\mu(y)
\end{equation*}
and the result follows from \eqref{cutnorm2}, 
letting $g(y)=1$ and
taking the supremum over all $f$ with $\normoo{f}\le 1$, using \eqref{rcut1}.
(Or simply taking $f(x)$ equal to the sign of $\mi W(x)$.)
\end{proof}

\begin{remark}\label{Rdense}
It is a standard fact that the step functions are dense in $\lis$ and
  $\liss$.
As a consequence, they are dense also in the cut norm in these spaces.  
\end{remark}

We finally note that the cut norm really is a norm if we, as usual, identify
functions that are equal a.e.

\begin{lemma}\label{Lcut0}
  If $W\in L^1(\sssq)$, then $\cn W=0\iff W=0$ a.e.
\end{lemma}
\begin{proof}
  Suppose that $\cn W=0$. Thus $\int_{S\times T} W(x,y)=0$ for all subsets
  $S,T\subseteq\sss$. It follows that $\int_\sssq W(x,y)f(x,y)=0$
for every step function $f$ on $\sssq$.

Let $g$ be any function on $\sssq$ with $\normoo g\le1$.
Since step functions are dense in $L^1(\sssq)$, there exists a sequence
$g_n$ of step functions such that $g_n\to g$ in $L^1(\sssq)$; by considering
a subsequence we may further assume that $g_n\to g$ \aex, and by truncating each
$g_n$ at $\pm1$ that $|g_n|\le1$. By dominated convergence, $\int_{\sssq}
Wg_n\to\int_\sssq Wg$, but each $\int_\sssq Wg_n=0$  since $g_n$ is a step
function; hence $\int_\sssq Wg=0$. If we choose $g\=\sgn(W)$, this shows
that $\int_\sssq|W|=0$, and thus $W=0$ a.e.
\end{proof}

\section{Pull-backs and rearrangements}\label{Spull}

Let $(\sss_1,\cF_1,\mu_1)$ and $(\sss_2,\cF_2,\mu_2)$ be two probability
spaces.

A mapping $\gf:\sss_1\to\sss_2$ is  \emph{measure-preserving} if it is
measurable and 
$\mu_1(\gf\qw(A))=\mu_2(A)$ for every $A\in\cF_2$ (\ie, for every measurable
$A\subseteq\sss_2$). 

A mapping $\gf:\sss_1\to\sss_2$  is a \emph{\mpb} if $\gf$ is a bijection of
$\sss_1$ onto $\sss_2$, and both $\gf$ and $\gf\qw$ are \mpp. (In other
words, $\gf$ is an isomorphism between the measure spaces
$(\sss_1,\cF_1,\mu_1)$ and $(\sss_2,\cF_2,\mu_2)$ in category theory sense.) 
Equivalently, $\gf$ is a \mpb{} if and only if it is a bijection that is
\mpp, and further $\gf\qw$ is measurable (and then automatically \mpp).
Note that if $\sss_1$ and $\sss_2$ are Borel spaces, then 
measurability of $\gf\qw$ is automatic by \refT{Tinjection}, so it suffices
to check that $\gf$ is a bijection and \mpp.

Note that if $\gf:\sss_1\to\sss_2$ is a measurable mapping, then 
$\gff:\sss_1^2\to\sss_2^2$ defined by
$\gff(x,y)=(\gf(x),\gf(y))$ is a measurable mapping, and if $\gf$ is
\mpp{} or a \mpb, then so is $\gff$.

We define, for any functions $f$ on $\sss_2$ and $W$ on $\sss_2^2$,
the \emph{pull-backs}
\begin{align}
  f^\gf(x)&\=f(\gf(x)), \\
 W^\gf(x,y)&\= W(\gf(x),\gf(y)); \label{Wpull}
\end{align}
these are functions on $\sss_1$ and $\sss_1^2$, respectively.

We will only consider measure-preserving $\gf$.
In the special case that $\gf$ is a measure-preserving bijection,
we say that $f^\gf$ and $W^\gf$
are \emph{rearrangements} of $f$  and $W$.
(However, we will not assume that $\gf$ is injective or bijective unless we
say so explicitly.) 
We further say that $W'$ is an 
\emph{\aex{} rearrangement} of $W$ if $W'=W^\gf$ \aex{} where $W^\gf$ is a
rearrangement of $W$.
Note that the relation
``$W_1$ is a rearrangement of $W_2$'' 
is symmetric and, moreover, an equivalence relation,
and similarly for \aex{} rearrangements.

\begin{remark}\label{Rbi2}
  Note that if $W$ is symmetric, then $W^\gf$ is too by \eqref{Wpull};
  recall that this is the case we really are interested in.

If we want to study general $W$, for example in connection with bipartite
graphs as mentioned in \refR{Rbi}, it is often more natural to allow different
maps $\gf_1$ and $\gf_2$ acting on the two coordinates.
\end{remark}

\begin{remark}
  Instead of \mpp{} bijections, it may be convenient to consider 
\emph{\mpp{} almost bijections}, which are mappings $\gf$ that are \mpp{}
bijections $\sss_1\setminus N_1\to\sss_2\setminus N_2$ for some null sets
$N_1$ and $N_2$. This makes essentially no difference below, and we leave
the details to the reader.
(See \refT{TU}\ref{TUtwin2} for a situation where almost bijections occur.)
\end{remark}

\begin{example}
  \label{Estep}
A kernel is a step kernel
if and only if it is a pull-back $W^\gf$ 
of some kernel
defined on a finite probability space. 
(The same holds for general functions $\sssq\to\bbR$.
Recall that step functions are the same as functions of finite type.)
\end{example}

\begin{remark}
  \label{Rpush}
We take here the point of view that $(\sss_1,\mu_1)$ and $(\sss_2,\mu_2)$
are given \ps{s}, and we consider suitable maps between them. A closely
related idea is to take a \ps{} $(\sss_1,\mu_1)$ and a measurable space
$\sss_2$ (without any particular measure). A measurable map
$\gf:\sss_1\to\sss_2$ then maps the measure $\mu_1$ to a measure $\mu_1^\gf$
on $\sss_2$ given by $\mu_1^\gf(A)\=\mu_1(\gf\qw(A))$ for all
$A\subseteq\sss_2$. 
Note that $\mu_1^\gf$ is the unique measure on $\sss_2$ that makes $\gf$ \mpp.
This well-known construction 
(called \emph{push-forward})
can be seen as a dual to
the pull-back above; note that measures map forward, from $\sss_1$ to
$\sss_2$, while functions map backward, from $\sss_2$ to $\sss_1$.

Note that, on the contrary, given a measurable map $\gf:\sss_1\to\sss_2$
between two measurable spaces, and a probability measure $\mu_2$ on $\sss_2$,
there is in general no measure $\mu_1$ on $\sss_1$ which makes $\gf$ \mpp.
This is a source of some of the technical difficulties in the theory.
\end{remark}

It is easy to see that the norms defined above are invariant under
rearrangements, and more generally under pull-backs by measure-preserving maps:

\begin{lemma}\label{L1}
If $\gf$ is measure-preserving, then, taking the norms in the
respective spaces, for any $f\in L^1(\sss)$ and $W\in L^1(\sssq)$,
\begin{align}
  \norml{f^\gf}&=\norml{f}, &   \norml{W^\gf}&=\norml{W} ,
\label{gf1}
\\
   \cutnorm{W^\gf}&=\cutnorm{W}. 
\label{gfcut}
\end{align}
\end{lemma}

\begin{proof}
The equalities \eqref{gf1} are standard. 

The cut norm equality
\eqref{gfcut} is obvious if $\gf$ is a measurable bijection.
In general, it seems simplest to 
first assume that $W$ is a step function, so that $W$ is constant on each
$A_i\times A_j$ for some partition $\sss_2=\bigcup_1^n A_i$, say $W=w_{ij}$ on
$A_i\times A_j$.
Then $A_i'\=\gf\qw(A_i)$ defines a partition of $\sss_1$, and
$W^\gf$ is a step function constant on each $A_i'\times A_j'$, and equal to
$w_{ij}$ there. 

Consider first $\cntwo W$. In the definition
\eqref{cutnorm2}, we may replace $f$ by its average on each $A_i$ (\ie, by
its conditional expectation given the partition) without changing the
integral, and similarly for $g$. This shows that it is enough to consider
$f$ and $g$ that are constant on each $A_i$, and we find
\begin{equation}\label{sjw}
\cntwo W=\sup\Bigabs{\sum_{i,j} w_{ij}a_ib_j\mu_2(A_i)\mu_2(A_j)},  
\end{equation}
taking the
supremum over all real numbers $a_i$ and $b_j$ with $|a_i|,|b_j|\le1$.
Since $\mu_1(A_i')=\mu_2(A_i)$, the same argument shows that $\cntwo{W^\gf}$
is given by the same quantity, and thus \eqref{gfcut} holds in this case.

For $\cnone\cdot$ we argue for step functions in exactly the same way, using
\refR{Rcn1} and taking $a_i,b_j\in\oi$ in \eqref{sjw}.

For a general $W$, let $\eps>0$ and let $W_1$ be a step function on
$\sss_2^2$ such that $\norml{W-W_1}<\eps$.
Then 
$$
\cn{W-W_1}\le\norml{W-W_1}<\eps.
$$ 
Further,
$\cn{W_1^\gf}=\cn{W_1}$ by what we just have shown, and
  \begin{equation*}
\cn{W^\gf-W_1^\gf}
\le\norml{W^\gf-W_1^\gf}	
=\norml{(W-W_1)^\gf}	
=\norml{W-W_1}
<\eps.
  \end{equation*}
The result $\cn{W^\gf}=\cn W$ follows by some applications of the triangle
inequality.   
\end{proof}

However, the distances $\norml{W_1-W_2}$ and $\cn{W_1-W_2}$ between two
kernels are, in general, not invariant under rearrangements of just one of
the kernels, since, in general, $\norm{W-W^\gf}\neq0$ for a kernel $W$ on a
space $\sss$ and a
measure-preserving bijection $\gf:\sss\to\sss$.
In the graph limit theory, we need a metric space where all
rearrangements are equivalent (and thus have distance 0 to each other); we
obtain this by taking the infimum over rearrangements.

Given two kernels $W_1$, $W_2$ on $[0,1]$,
the \emph{cut metric} of
\citet{BCLSV1} may be defined by
\begin{equation}\label{cutdefarr}
 \dcut(W_1,W_2) = \inf_{\gf} \cn{W_1-W_2^\gf},
\end{equation}
taking the infimum over all \mpb{} $\gf:\oi\to\oi$; in other words, over all
rearrangements $W_2^\gf$ of $W_2$.
(If we wish to specify which version of the cut norm is involved, we write
$\dcutone$ or $\dcuttwo$.)
\citet{BCLSV1} showed that for kernels on $\oi$, there are several
equivalent definitions of $\dcut$, see \refT{Tcut} below.
For general probability
spaces $\sss$, we have to 
use couplings between different kernels instead of rearrangements, 
see the following section; it then 
further is irrelevant whether the kernels
are defined on the same probability space or not.

On the other hand, if we restrict ourselves to $\oi$, we can do with
a special simple case of rearrangements.
Following \citet{BCLSV1}, we define an \emph{$n$-step interval permutation}
to be the map $\tgs$ defined
for a permutation $\gs$ of \set{1,\dots,n}
by taking 
the partition $(0,1]=\bigcup I_{in}$ with   $I_{in}\=((i-1)/n,i/n]$ 
	and mapping each $I_{in}$ by translation to $I_{\gs(i),n}$. (For
	completeness 	we also let $\tgs(0)=0$.) 
Evidently, $\tgs$ is a \mpp{} bijection $\oi\to\oi$.
We shall see in \refT{Tcut} below that it suffices to use such interval
permutations in 
\eqref{cutdefarr}. 

\begin{example}\label{Ebaddiscrete}
  To see one problem caused by using \eqref{cutdefarr} for kernels on
  a general \ps{}, let $\sss$ be the two-point space $\set{1,2}$, and let
  $\mu\set1=\frac12-\eps$, $\mu\set2=\frac12+\eps$, for some small $\eps>0$.
Let $W_1(x,y)\=\ett{x=y=1}$ and $W_2(x,y)\=\ett{x=y=2}$. 
On this probability space there
is no \mpp{} bijection except the identity, so \eqref{cutdefarr} yields
$\cn{W_1-W_2}=(\frac12+\eps)^2>\frac14$, 
while the coupling definition \eqref{dcut}
below yields  
$\dcut(W_1,W_2)=2\eps$.
\end{example}

\section{Couplings and the cut metric}\label{Scoupling}

Given two probability spaces $(\sss_1,\mu_1)$, $(\sss_2,\mu_2)$,
a \emph{coupling} of these spaces is  
a pair of measure preserving maps $\gf_i:\sss\to\sss_i$, $i=1,2$, defined on
a common (but arbitrary)  probability space $(\sss,\mu)$.

\begin{remark}\label{Rcoupling}
  Couplings are more common in the context of two random variables, say
  $X_1$ and $X_2$. 
These are often real-valued, but may more generally take values in
  any measurable spaces $\sss_1$ and $\sss_2$. A coupling of $X_1$ and $X_2$
  then is a pair $(X_1',X_2')$ of random variables defined on a common
  probability space such that $X_1'\eqd X_1$ and $X_2'\eqd X_2$. 
This is the same as a coupling of the two probability spaces
$(\sss_1,\mu_1)$ and $(\sss_2,\mu_2)$ according to our definition above,
where $\mu_1$ is the distribution of $X_1$ and $\mu_2$ the distribution of $X_2$.
\end{remark}

The general definition of the cut metric, for kernels defined on arbitrary
probability spaces (possibly different ones), is as follows.

Given kernels $W_i$ on $\sss_i$, $i=1,2$,
or more generally any functions $W_i\in L^1(\sss_i^2)$,
we define the \emph{cut metric} by 
\begin{equation}\label{dcut}
 \dcut(W_1,W_2) = \inf \cn{W_1^{\gf_1} - W_2^{\gf_2}}, 
\end{equation}
where the infimum is taken over all couplings
$(\gf_1,\gf_2):\sss\to(\sss_1,\sss_2)$ of $\sss_1$ 
and $\sss_2$ (with $\sss$ arbitrary), 
and $W_i^{\gf_i}$ is the pull-back defined in 
\eqref{Wpull}.

We similarly define
\begin{equation}\label{dl}
 \dl(W_1,W_2) = \inf \norm{W_1^{\gf_1} - W_2^{\gf_2}}\qliss, 
\end{equation}
again taking the infimum  over all couplings
$(\gf_1,\gf_2)$ of $\sss_1$ 
and $\sss_2$.

\begin{remark}
  \label{Rcutdef}
It is not obvious that the definition \eqref{dcut} 
agrees with \eqref{cutdefarr} for
kernels on $[0,1]$, but, as shown in~\cite{BCLSV1}, this is the case; see
\refT{Tcut} below.
Note, in somewhat greater generality, that if 
$W_1$ and $W_2$ are kernels of probability spaces $\sss_1$ and $\sss_2$,
and $\gf:\sss_1\to\sss_2$ is \mpp, then $(\iota,\gf)$ is a coupling defined
on $\sss_1$. (We let here and below $\iota$ denote the identity map in any
space.) Hence, we always have $\dcut(W_1,W_2)\le\cn{W_1-W_2^{\gf}}$.
\end{remark}

Note that $\dcut$ and $\dl$ really are
pseudometrics rather than metrics, since $\dcut(W_1,W_2)=0$ and
$\dl(W_1,W_2)=0$  
in many cases
  with $W_1\neq W_2$, for example if $W_1=W_2^{\gf}$ for a measure
  preserving $\gf$ (use the coupling $(\iota,\gf)$ in \eqref{dcut}, see
  \refR{Rcutdef}). 
Nevertheless, it is customary to call this pseudometric the \emph{cut metric}.
We will return to the important problem of when $\dcut(W_1,W_2)=0$ in
\refS{Sequiv}. 

It is obvious from the definition \eqref{dcut} that $\dcut$ and $\dl$ are
non-negative and symmetric, and $\dcut(W,W)=\dl(W,W)=0$ for every $W$.
It is less obvious that they really are subadditive, \ie, that the triangle
inequality holds, so we give a detailed proof in \refL{Ltriangle} below.

A coupling $(\gf_1,\gf_2)$ of two probability spaces $(\sss_1,\mu_1)$ and
$(\sss_2,\mu_2)$, with $\gf_1,\gf_2$ defined on $(\sss,\mu)$, defines a map
$\Phi\=(\gf_1,\gf_2):\sss\to\sss_1\times\sss_2$, which induces a unique measure
$\tmu$ on $\sss_1\times\sss_2$ such that
$\Phi:(\sss,\mu)\to(\sssxy,\tmu)$ is \mpp{} (see \refR{Rpush}).
Let $\pi_i:\sssxy\to\sss_i$ be the projection; then $\gf_i=\pi_i\circ\Phi$,
$i=1,2$. Note that if $A\subseteq \sss_i$, then
\begin{equation*}
  \tmu\lrpar{\pi_i\qw(A)}
=  \mu\lrpar{\Phi\qw\lrpar{\pi_i\qw(A)}}
=  \mu\lrpar{\gf_i\qw(A)}
=\mu_i(A),
\end{equation*}
since $\Phi$ and $\gf_i$ are \mpp; thus
$\pi_i:(\sssxy,\tmu)\to(\sss_i,\mu_i)$ is \mpp. 
Hence, $(\pi_1,\pi_2)$ is a coupling of $(\sss_1,\mu_1)$ and
$(\sss_2,\mu_2)$. 
If $W_i\in L^1(\sss_i^2)$, then $W_i^{\gf_i}=(W_i^{\pi_i})^{\Phi}$ and thus,
using \eqref{gfcut},
\begin{equation}\label{prod}
  \cn{W_1^{\gf_1}-W_2^{\gf_2}}
=
  \cn{\xpar{W_1^{\pi_1}-W_2^{\pi_2}}^\Phi}
=  \cn{W_1^{\pi_1}-W_2^{\pi_2}}.
\end{equation}
Consequently, in \eqref{dcut} it suffices to consider couplings of the type
$(\pi_1,\pi_2)$ defined on $(\sssxy,\tmu)$, where $\tmu$ is a probability
measure such that $\pi_1$ and $\pi_2$ are \mpp, \ie, such that $\tmu$ has
the correct marginals $\mu_1$ and $\mu_2$.

Before proving the triangle inequality, we prove a technical lemma
and a partial result.
\begin{lemma}\label{Lstep=}
    Let $\sss_1$ and $\sss_2$ be probability spaces and let 
$W_1\in L^1(\sss_1^2)$ and $W_2\in L^1(\sss_2^2)$ be step functions with 
corresponding partitions $\sss_1=\bigcup_{i=1}^I A_i$ and
$\sss_2=\bigcup_{j=1}^J B_j$. 
If $(\gf_1,\gf_2)$ and $(\gf_1',\gf_2')$ are two couplings of\/ $\sss_1$ and
$\sss_2$,  defined on $(\sss,\mu)$ and $(\sss',\mu')$ respectively,
such that 
$\mu\bigpar{\gf_1\qw(A_i)\cap\gf_2\qw(B_j)}
=
\mu'\bigpar{\gf'_1{}\qw(A_i)\cap\gf'_2{}\qw(B_j)}$
for every $i$ and $j$, then
$\cnx{\mu}{W_1^{\gf_1}-W_2^{\gf_2}}
=\cnx{\mu'}{W_1^{\gf'_1}-W_2^{\gf'_2}}$ and, 
similarly, 
$\normlx{\mu}{W_1^{\gf_1}-W_2^{\gf_2}}
=\normlx{\mu'}{W_1^{\gf'_1}-W_2^{\gf'_2}}$.
\end{lemma}
\begin{proof}
  Recall that $\cn{W_1^{\gf_1}-W_2^{\gf_2}}$ is given by
\eqref{cutnorm2},
in case of $\cnone\cdot$ further assuming $f,g\ge0$, see \refR{Rcn1}.
Since $W_1^{\gf_1}-W_2^{\gf_2}$ is constant on each set
$C_{ij}\=\gf_1\qw(A_i)\cap\gf_2\qw(B_j)$, we 
may as in the proof of \refL{L1} average $f$ and $g$ in \eqref{cutnorm2}
over each such set, so
it suffices to consider $f$ and $g$ that are constant on each set $C_{ij}$. 
Consequently, if $W_1=u_{ik}$ on $A_i\times A_k$ and $W_2=v_{jl}$ on
$B_j\times B_l$, then
\begin{equation}\label{erika}
\cnx{\mu}{W_1^{\gf_1}-W_2^{\gf_2}}
=\max_{(f_{ij}),(g_{kl})} 
 \biggabs{\sum_{i,j,k,l} \mu(C_{ij})\mu(C_{kl})(u_{ik}-v_{jl})f_{ij}g_{kl}},
\end{equation}
taking the maximum over all arrays $(f_{ij})$ and $(g_{kl})$ of numbers 
in $\oi$ for $\cnone\cdot$  and in $[-1,1]$ for $\cntwo\cdot$.
This depends on the coupling only through the numbers $\mu(C_{ij})$, and the
result 
follows.

For the $L^1$ norm we have immediately, with the same notation,
\begin{equation*}
\normlx{\mu}{W_1^{\gf_1}-W_2^{\gf_2}}
= \sum_{i,j,k,l} \mu(C_{ij})\mu(C_{kl})\bigabs{u_{ik}-v_{jl}},
\end{equation*}
and the result follows.
\end{proof}

\begin{lemma}\label{Ltri0}
  Let $\sss_1$ and $\sss_2$ be probability spaces and 
$W_1,W_1'\in  L^1(\sss_1^2)$ and $W_2\in  L^1(\sss_2^2)$.
Then $\dcut(W_1,W_2)\le\dcut(W_1',W_2)+\cn{W_1-W_1'}$ and,
similarly, $\dl(W_1,W_2)\le\dl(W_1',W_2)+\norml{W_1-W_1'}$.
\end{lemma}

\begin{proof}
  Let $(\gf_1,\gf_2)$ be a coupling of $\sss_1$ and $\sss_2$. Then, 
using \refL{L1},
  \begin{equation*}
	\begin{split}
\dcut(W_1,W_2)
&\le\cn{W_1^{\gf_1}-W_2^{\gf_2}}	  
\le\cn{(W_1')^{\gf_1}-W_2^{\gf_2}} +\cn{W_1^{\gf_1}-(W_1')^{\gf_1}}
\\
&=\cn{(W_1')^{\gf_1}-W_2^{\gf_2}} +\cn{(W_1-W_1')^{\gf_1}}	  
\\
&=\cn{(W_1')^{\gf_1}-W_2^{\gf_2}} +\cn{W_1-W_1'}.	  
	\end{split}
  \end{equation*}
The result for $\dcut$ follows by taking the infimum over all couplings.
The proof for $\dl$ is the same.
\end{proof}

\begin{lemma}\label{Ltriangle}
  Let, for $i=1,2,3$, $\sss_i$ be a probability space and $W_i\in
  L^1(\sss_i^2)$. 
Then $\dcut(W_1,W_3)\le\dcut(W_1,W_2)+\dcut(W_2,W_3)$ and,
similarly, $\dl(W_1,W_3)\le\dl(W_1,W_2)+\dl(W_2,W_3)$.
Hence $\dcut$ and $\dl$ are (pseudo)metrics.
\end{lemma}

\begin{proof}
Roughly speaking, given a coupling of $\sss_1$ and $\sss_2$ and another
coupling of $\sss_2$ and $\sss_3$, we want to couple the couplings so that
we can compare pull-backs of $W_1$ and $W_3$. This simple idea,
unfortunately, leads to technical difficulties in general, but it works
easily if, for example, the spaces are finite. We use therefore an
approximation argument with step functions
which essentially reduces to the finite case.

Thus, suppose first that $W_1,W_2,W_3$  are step functions with corresponding
partitions $\sss_1=\bigcup_{i=1}^I A_i$, $\sss_2=\bigcup_{j=1}^J B_j$,
$\sss_3=\bigcup_{k=1}^K C_k$, and assume for simplicity that $\mu_1(A_i)$,
$\mu_2(B_j)$ and $\mu_3(C_k)$ are non-zero for all $i,j,k$.
($\mu_\ell$ denotes the measure on $\sss_\ell$.)

We consider $\dcut$; the proof for $\dl$ is the same.
Let $\eps>0$. By the definition of $\dcut$ and the comments just made
(see \eqref{prod}), there
exist measures $\mu'$ on $\sss_1\times\sss_2$ and $\mu''$ on
$\sss_2\times\sss_3$, with marginals $\mu_\ell$ on $\sss_\ell$, such that 
\begin{align}\label{lt1}
 \cnx{\mu'}{W_1^{\pi_1}-W_2^{\pi_2}}&<\dcut(W_1,W_2)+\eps,
\\ 
\cnx{\mu''}{W_2^{\pi_2}-W_3^{\pi_3}}&<\dcut(W_2,W_3)+\eps.
\end{align}
(We abuse notation a little by letting $\pi_\ell$ denote the projection onto
$\sss_\ell$ from any product space.)

Define a measure $\mu$ on $\sss_1\times\sss_2\times\sss_3$ by,
for $E\subseteq \sss_1\times\sss_2\times\sss_3$,
\begin{equation*}
  \mu(E)\=
 \sum_{i,j,k} \frac{\mu'(A_i\times B_j)\mu''(B_j\times C_k)}{\mu_2(B_j)}
\cdot
\frac{\mu_1\times\mu_2\times\mu_3(E\cap(A_i\times B_j\times C_k))}
{\mu_1(A_i)\mu_2(B_j)\mu_3(C_k)}.
\end{equation*}
We have
$\mu_1(A_i)=\sum_j\mu'(A_i\times B_j)$, 
$\mu_2(B_j)=\sum_i\mu'(A_i\times B_j)=\sum_k\mu''(B_j\times C_k)$,
and $\mu_3(C_k)=\sum_j\mu''(B_j\times C_k)$.
It follows that the three mappings
$\pi_\ell:(\sssxyz,\mu)\to(\sss_\ell,\mu_\ell)$ 
are \mpp{} since, for example, if $F\subseteq\sss_1$, then
$\pi_1\qw(F)=F\times\sss_2\times\sss_3$ and
\begin{equation*}
  \begin{split}
\mu\bigpar{\pi_1\qw&(F)}
=\mu(F\times\sss_2\times\sss_3)
\\
&=\sum_{i,j,k} \frac{\mu'(A_i\times B_j)\mu''(B_j\times C_k)}{\mu_2(B_j)}
\cdot
\frac{\mu_1\times\mu_2\times\mu_3\bigpar{(F\cap A_i)\times B_j\times C_k}}
{\mu_1(A_i)\mu_2(B_j)\mu_3(C_k)}
\\
&=\sum_{i,j,k} \frac{\mu'(A_i\times B_j)\mu''(B_j\times C_k)}{\mu_2(B_j)}
\cdot
\frac{\mu_1(F\cap A_i)}{\mu_1(A_i)}
\\
&=\sum_{i,j} \mu'(A_i\times B_j)
\frac{\mu_1(F\cap A_i)}{\mu_1(A_i)}
=\sum_{i} \mu_1(F\cap A_i)
=\mu_1(F).
  \end{split}
\end{equation*}
In particular, $\mu$ is a \pmm.

The projections $\pi_{12}:\sssxyz\to\sss_1\times\sss_2$
and $\pi_{23}:\sssxyz\to\sss_2\times\sss_3$ map $\mu$ to measures $\tmu'$
on $\sss_1\times\sss_2$ and $\tmu''$ on $\sss_2\times\sss_3$.
We have, for any $i$ and $j$,
\begin{equation*}
  \begin{split}
\tmu'(A_i\times B_j)
&=\mu(\pi_{12}\qw(A_i\times B_j))
=\mu(A_i\times B_j\times\sss_3)
\\
&=\sum_{k} \frac{\mu'(A_i\times B_j)\mu''(B_j\times C_k)}{\mu_2(B_j)}
\cdot
\frac{\mu_1\times\mu_2\times\mu_3(A_i\times B_j\times C_k)}
{\mu_1(A_i)\mu_2(B_j)\mu_3(C_k)}
\\
&=\sum_{k} \frac{\mu'(A_i\times B_j)\mu''(B_j\times C_k)}{\mu_2(B_j)}
=\mu'(A_i\times B_j).
  \end{split}
\end{equation*}
Hence, by \refL{Lstep=},
\begin{equation}\label{lt2}
 \cnx{\sssxy,\tmu'}{W_1^{\pi_1}-W_2^{\pi_2}}
=\cnx{\sssxy,\mu'}{W_1^{\pi_1}-W_2^{\pi_2}}.  
\end{equation}
Further, since $\pi_{12}:(\sssxyz,\mu)\to(\sss_1\times\sss_2,\tmu')$ is
  \mpp, \refL{L1} implies that (recall our generic use of $\pi_\ell$)
\begin{equation}\label{lt3}
 \cnx{\sssxyz,\mu}{W_1^{\pi_1}-W_2^{\pi_2}}
=\cnx{\sssxy,\tmu'}{W_1^{\pi_1}-W_2^{\pi_2}}.  
\end{equation}
Combining \eqref{lt1}, \eqref{lt2} and \eqref{lt3}, we find
\begin{equation}
  \cnx{\mu}{W_1^{\pi_1}-W_2^{\pi_2}} < \dcut(W_1,W_2)+\eps.
\end{equation}
Similarly,
\begin{equation}
  \cnx{\mu}{W_2^{\pi_2}-W_3^{\pi_3}} < \dcut(W_2,W_3)+\eps.
\end{equation}
 We have reached our goal of finding suitable couplings on the same space,
\viz{} $(\sssxyz,\mu)$, and we can now use the triangle inequality for
$\cn\cdot$ 
and deduce
\begin{equation*}
  \begin{split}
\dcut(W_1,W_3)
&\le  \cnx{\mu}{W_1^{\pi_1}-W_3^{\pi_3}} 
\le  \cnx{\mu}{W_1^{\pi_1}-W_2^{\pi_2}} +\cnx{\mu}{W_2^{\pi_2}-W_3^{\pi_3}} 
\\
&< \dcut(W_1,W_2)+\dcut(W_2,W_3)+2\eps.	
  \end{split}
\end{equation*}
Since $\eps>0$ is arbitrary,
this implies the desired inequality
$\dcut(W_1,W_3)\le\dcut(W_1,W_2)+\dcut(W_2,W_3)$ in the case
of step functions.

In general, we approximate first each $W_\ell$ by a step function $W_\ell'$ 
such that $\cnx{\sss_\ell}{W_\ell-W'_\ell}<\eps$. (We may assume, as we did
above, that all sets in the partition have positive measures by removing any
null sets in them, redefining $W_\ell'$ on a null set.)
The result for step functions together with several applications of
\refL{Ltri0} yield
\begin{equation*}
  \begin{split}
\dcut(W_1,W_3)	
&\le \dcut(W'_1,W'_3)+2\eps
\le \dcut(W'_1,W'_2)+\dcut(W'_2,W'_3)+2\eps
\\&
\le \dcut(W_1,W_2)+\dcut(W_2,W_3)+6\eps.
  \end{split}
\end{equation*}
The result $\dcut(W_1,W_2)\le \dcut(W_1,W_2)+\dcut(W_2,W_3)$ follows.
\end{proof}

\begin{corollary}\label{C0}
  Let, for $i=1,2,3$, $\sss_i$ be a probability space and $W_i\in
  L^1(\sss_i^2)$. If $\dcut(W_1,W_2)=0$, then
  $\dcut(W_1,W_3)=\dcut(W_2,W_3)$. 
(The same result holds for $\dl$.)
\nopf
\end{corollary}

Consider the class $\cWx\=\bigcup_\sss \cW(\sss)$ of all graphons (on any
probability space).
We define a relation $\equ$ on this class 
(or on the even larger class $\bigcup_\sss L^1(\sssq)$)
by
\begin{equation}
  W_1\equ W_2 \text{\quad if\quad} \dcut(W_1,W_2)=0.
\end{equation}
\refC{C0} shows that this is an equivalence relation, and that $\dcut$ is a
true metric on the quotient space $\bwx\=\cWx/\equ$. We say that two graphons
$W_1,W_2$ are
\emph{equivalent} if $W_1\equ W_2$, \ie, if $\dcut(W_1,W_2)=0$.
(We will see in \refT{Teq1} below that $\dl(W_1,W_2)=0$ defines the same
equivalence relation.)

It is a central fact in the graph limit theory \cite{BCLSV1}
that this quotient space 
$\bwx\=\cWx/\equ$
is
homeomorphic to (and thus can be identified with) the set of graph limits;
moreover, the metric space $(\bwx,\dcut)$ is compact. (See also \cite{SJ209}.)
The compactness is closely related to 
Szemer\'edi's regularity lemma, see \cite{LSz:Sz}.

We will always regard $\bwx$ as a compact metric space equipped with the
metric $\dcut$, except a few times when we explicitly use $\dl$
instead. Note that $\dl$ is a larger metric and thus gives a stronger
topology. In particular, $(\bwx,\dl)$ is \emph{not} compact.

\begin{example}\label{Eequ1}
  If $W:\sssq\to\oi$ is any graphon (or kernel)
on a probability space $\sss$, and
  $\gf:\sss'\to\sss$ is a \mpp{} map, then, as remarked above,
$W$ is equivalent to its pull-back $W^\gf$.
\end{example}

\begin{example}
  \label{Ewg2}
Let $G$ be a graph with vertex set $V=\set{1,\dots,n}$,
and consider the graphons $\wgv G$ and $\wgi G$
defined in \refE{Ewg1}.
Let $\gf:\ooi\to V$ be the map $x\mapsto\ceil{nx}$. Then $\gf$ is \mpp{} and
\eqref{w2g} defines
$\wgi G$ as the pull-back $(\wgv G)^\gf$.
Hence $\wgv G\equ \wgi G$.
\end{example}

We can now prove, following \cite{BCLSV1},
that the definition \eqref{cutdefarr} agrees with our
definition \eqref{dcut} of the cut metric for $\oi$, and more generally for
any atomless Borel spaces. We include several related versions; note that
\ref{tcut1} is \eqref{dcut} and \ref{tcutrearr} is \eqref{cutdefarr}.

\begin{theorem}
  \label{Tcut}
Let\/ $W_1$ and $W_2$ be two kernels defined on \ps{s} $(\sss_1,\mu_1)$ and
$(\sss_2,\mu_2)$, respectively.
Then the following are the same, and thus all define $\dcut(W_1,W_2)$.
\begin{romenumerate}
\item \label{tcut1}
For any  $\sss_1$ and $\sss_2$,
  \begin{equation*}
  \inf_{\gf_1,\gf_2} \cnx{\sss,\mu}{W_1^{\gf_1}-W_2^{\gf_2}},  
\end{equation*}
where the infimum is over all couplings
(pairs of \mpp{} maps) $\gf_1:(\sss,\mu)\to(\sss_1,\mu_1)$
and $\gf_2:(\sss,\mu)\to(\sss_2,\mu_2)$.
\item \label{tcutxy}
For any  $\sss_1$ and $\sss_2$,
  \begin{equation*}
\inf_{\mu} \cnx{\sssxy,\mu}{W_1^{\pi_1}-W_2^{\pi_2}},	
  \end{equation*}
where $\pi_i:\sssxy\to\sss_i$ is the projection and
the infimum is over all measures $\mu$ on $\sssxy$
having marginals $\mu_1$ and $\mu_2$.
\item  \label{tcutxy2}
For any  $\sss_1$ and $\sss_2$,
for  $\dcuttwo$,
\begin{multline*}
\inf_{\mu}\sup_{\normoo{f},\normoo{g}\le1}
  \Bigabs{\int_{(\sssxy)^2} \bigpar{W_1(x_1,y_1)-W_2(x_2,y_2)}
\\\cdot
  f(x_1,x_2)g(y_1,y_2)\dd\mu(x_1,x_2)\dd\mu(y_1,y_2)},
\end{multline*}
taking 
the infimum over all measures $\mu$ on $\sssxy$ having marginals $\mu_1$ and
$\mu_2$; for $\dcutone$ we further restrict to $f,g\ge0$.
\item \label{tcutmpp}
Provided $\sss_1$ and $\sss_2$ are atomless Borel spaces,
\begin{equation*}
 \inf_{\gf} \cn{W_1-W_2^{\gf}},
\end{equation*}
where the infimum is over all 
\mpp{} $\gf:\sss_1\to\sss_2$.
\item \label{tcutrearr}
Provided $\sss_1$ and $\sss_2$ are atomless Borel spaces,
\begin{equation*}
 \inf_{\gf} \cn{W_1-W_2^{\gf}},
\end{equation*}
where the infimum is over all 
\mpb{s} $\gf:\sss_1\to\sss_2$,
\ie, over all
rearrangements of $W_2$ defined on $\sss_1$.
\item \label{tcutinterval}
Provided $\sss_1=\sss_2=\oi$,
\begin{equation*}
 \inf_{\tgs} \cn{W_1-W_2^{\tgs}},
\end{equation*}
where the infimum is over all 
interval permutations $\tgs:\oi\to\oi$, defined by permutations $\gs$ of
$\set{1,\dots,n}$ with $n$ arbitrary.
\label{tcutomega}
\end{romenumerate}
\end{theorem}

\begin{proof}
\ref{tcut1}$\iff$\ref{tcutxy}.
We have shown this in \eqref{prod} and the accompanying argument.

\ref{tcutxy}$\iff$\ref{tcutxy2}.
Directly from the definition \eqref{cutnorm2} (using \refR{Rcn1} for
$\dcutone$), writing $x=(x_1,x_2)$ and $y=(y_1,y_2)$. 
(The expression in \ref{tcutxy2} is just
writing the definition explicitly in this case.)

For \ref{tcutmpp} and \ref{tcutrearr}, we first note that by \refT{Tborelp},
$\sss_1$ and $\sss_2$ are isomorphic to $\oi$ (equipped with Lebesgue measure),
\ie, there are \mpp{} bijections $\psi_j:\oi\to\sss_j$. It is evident that
we may use these maps to transfer the problem to the pull-backs
$W_1^{\psi_1}$ and $W_2^{\psi_2}$ on $\oi$. In other words, we may 
in \ref{tcutmpp} and \ref{tcutrearr} assume that $\sss_1=\sss_2=\oi$.

In this case, denote the quantities in \ref{tcut1}--\ref{tcutomega} by
$\dref{tcut1},\dots,\dref{tcutomega}$.
We have
$\dref{tcut1}\le\dref{tcutmpp}\le\dref{tcutrearr}\le\dref{tcutinterval}$,
since we take infima over smaller and smaller sets of maps.
Further, we have shown that $\dref{tcut1}=\dref{tcutxy}=\dref{tcutxy2}$.
To complete the proof, it thus suffices to show that
$\dref{tcutinterval}\le\dref{tcutxy}$.

Let $\eps>0$ and
let $I_{iN}$ denote the interval $((i-1)/N,i/N]$, for $1\le i\le N$.
The set of step functions $W:\oiq\to\bbR$ that correspond to
partitions of $\oi$ (or rather $(0,1]$, but the difference does not matter
  here) into $m$ equally long intervals $I_{1m},\dots,I_{mm}$ for $m=1,2,\dots$,
  is a dense subset of $L^1(\oiq)$.
Hence we may choose $m>0$ and two such step functions $W_1'$ and $W_2'$ so
that $\norml{W_i-W_i'}<\eps$, $i=1,2$. (We may first obtain such $W_i'$ with
different $m_1$ and $m_2$, but we may then replace both by $m:=m_1m_2$.)
By \refL{Ltri0} and its proof, which applies to all the versions 
$\dref{tcut1},\dots,\dref{tcutomega}$, we have 
\begin{equation}
  \label{sofie}
\drefx{*}(W_1,W_2)-2\eps \le \drefx{*}(W'_1,W'_2)\le\drefx{*}(W_1,W_2)+2\eps
\end{equation}
for every $*=\ref{tcut1},\dots,\ref{tcutomega}$.

Choose a probability measure $\mu$ on $\sss_1\times\sss_2=\oiq$ such
that  
$\cn{W_1'{}^{\pi_1}-W_2'{}^{\pi_2}}<\dref{tcutxy}(W_1',W_2')+\eps$.
We may evaluate this cut norm by \eqref{erika} (replacing $W_i$ by $W_i'$) and
as asserted in \refL{Lstep=}, the cut norm depends only on the numbers
$\mu(C_{ij})$, where now
$C_{ij}\=\pi_1\qw(I_{im})\cap\pi_2\qw(I_{jm})=I_{im}\times I_{jm}$,
so we may assume that the coupling measure $\mu$ on $\oiq$
on each square $C_{ij}$ equals a constant factor $\gl_{ij}$ times the
Lebesgue measure. (Hence, $\mu(C_{ij})=\gl_{ij}/m^ 2$.)
We adjust these factors so that every $\mu(C_{ij})$ is rational;
we may do this so that the marginals still are correct, \ie,
for every $i$ and $j$, 
\begin{equation}\label{mumarg}
\sum_l\mu(C_{il})=\sum_l\mu(C_{lj})=\frac1m .
\end{equation}
The adjustment will change cut norm in
\eqref{erika} by an arbitrary small amount, so we can do this and still have
$\cn{W_1'{}^{\pi_1}-W_2'{}^{\pi_2}}<\dref{tcutxy}(W_1',W_2')+\eps$.

We now have $\mu(C_{ij})=a_{ij}/N$ for some integers $N$ and $a_{ij}$, $1\le
i,j\le m$.
Let $b\=N/m$. 
By \eqref{mumarg}, for every $i$ and $j$,
\begin{equation}\label{magnus}
\sum_j a_{ij}=\sum_ia_{ij}=\frac Nm=b. 
\end{equation}
Hence, $b$ is an integer, and thus every interval $I_{im}$ is a union
$\bigcup_{k=b(i-1)+1}^{bi}I_{kN}$ of $b$ intervals $I_{kN}$ of length
$1/N$.
By \eqref{magnus}, we may construct a permutation $\gs$ of
$\set{1,\dots,N}$ such that $\gs$ maps exactly $a_{ij}$ of the indices
$k\in[b(i-1)+1,bi]$ into $[b(j-1)+1,bj]$, for all $i,j$.
Hence, $\gl(I_{im}\cap\tgs\qw(I_{jm}))=a_{ij}/N=\mu(C_{ij})$.
Thus, \refL{Lstep=} applies to the couplings $(\pi_1,\pi_2)$ and
$(\iota,\tgs)$ (defined on $\oi$); hence,
\begin{equation*}
\dref{tcutinterval}(W_1',W_2')
\le
\cn{W_1'-W_2'{}^{\tgs}}
=\cn{W_1'{}^{\pi_1}-W_2'{}^{\pi_2}}
<\dref{tcutxy}(W_1',W_2')+\eps.
\end{equation*}
Finally, \eqref{sofie} yields
$\dref{tcutinterval}(W_1,W_2)
<
\dref{tcutxy}(W_1,W_2)+5\eps
$, and the result follows since $\eps$ is arbitrary.
\end{proof}

\begin{remark}\label{Rbaddiscrete}
On spaces with atoms, the quantities $\dref{tcutmpp}$ and $\dref{tcutrearr}$
defined in $\ref{tcutmpp}$ and $\ref{tcutrearr}$
are in general different from  $\dcut$, see \refE{Ebaddiscrete}.
(In this case, they are larger than $\dcut$, see \refR{Rcutdef}.)
Furthermore, for two general \ps{s} $\sss_1$ and $\sss_2$, it is possible
that there are no \mpp{} maps $\sss_1\to\sss_2$ at all, in which case 
the definitions in $\ref{tcutmpp}$ and $\ref{tcutrearr}$ are not appropriate;
and even if we may interpret $\dref{tcutmpp}$ as a
default value 1 (for graphons; for general kernels we would have to use
$\infty$), in such cases, 
$\dref{tcutmpp}$ is not even symmetric in general.
(For an example, modify \refE{Ebaddiscrete} by replacing $W_2$ by a
pull-back defined on 
$\oi$; then there are \mpp{} maps $\sss_2\to\sss_1$ but not conversely.
We have $\dref{tcutmpp}(W_2,W_1)=2\eps<\dref{tcutmpp}(W_1,W_2)=1$.)
\end{remark}

\begin{remark}
  In \ref{tcutmpp}, it suffices that $\sss_1$ and $\sss_2$
  are Borel spaces such that $\sss_1$ is atomless. To see this, replace $W_2$
  by its pull-back $W_2^\pi$ defined on the atomless Borel space
  $\widetilde\sss_2\=\sss_2\times\oi$, where $\pi$ is the projection onto
  $\sss_2$. 
\end{remark}

\begin{remark}\label{RcutLeb}
  \ref{tcutmpp} and \ref{tcutrearr} hold also for atomless Lebesgue spaces,
since then, for $\ell=1,2$, $\sss_\ell=(\sss_\ell,\cF_\ell,\mu_\ell)$ 
is the completion of some Borel
space $\sss_\ell^0=(\sss_\ell,\cF_\ell^0,\mu_\ell)$, 
and we may replace $W_\ell$ by a kernel $W_\ell^0$ on
$\sss_\ell^0$ 
with $W_\ell=W_\ell^0$ \aex, 
\cf{} \refR{RBL2}; note that every \mpp{} map
$\gf:\sss_1^0\to\sss_2^0$ also is \mpp{}
$\sss_1\to\sss_2$.
\end{remark}

\begin{remark}\label{Rcutl}
  An obvious analogue of \refT{Tcut} holds for $\dl$.
(In \ref{tcutxy2}, the integral is 
$
  \int_{(\sssxy)^2} \bigabs{W_1(x_1,y_1)-W_2(x_2,y_2)}
\dd\mu(x_1,x_2)\dd\mu(y_1,y_2),
$
and there are no $f$ and $g$.)
\end{remark}

\begin{remark}
  In probabilistic notation, see \refR{Rcoupling},
\ref{tcutxy2} can be written as
\begin{equation*}
\inf_{(X_1',X_2')}\sup_{\normoo{f},\normoo{g}\le1}
  \Bigabs{\E \Bigpar{\bigpar{W_1(X'_1,X''_1)-W_2(X'_2,X''_2)}
  f(X'_1,X'_2)g(X''_1,X''_2)}},
\end{equation*}
where the infimum is taken over all couplings $(X_1',X_2')$ of two random
variables $X_1$ and $X_2$ such that $X_\ell$ is $\sss_\ell$-valued and has
distribution $\mu_\ell$, and $(X_1'',X_2'')$ is an independent copy of
$(X_1',X_2')$. 
\end{remark}

\begin{corollary}
  Let $\sss$ be an atomless Borel spaces, \eg{} $\oi$, and let $W$ be a
  graphon on $\sss$.
Then the equivalence class
of all graphons on $\sss$
equivalent to $W$ 
equals the closure of the orbit of $W$ under \mpb{s} (or maps);
\ie,
\begin{equation*}
  \set{W'\in\cW(\sss):W'\equ W}
 =\overline{\set{W^{\gf}:\gf\in S_{\mathrm{mp}}}}
 =\overline{\set{W^{\gf}:\gf\in S_{\mathrm{mpb}}}},
\end{equation*}
where $S_{\mathrm{mp}}$ is the set of all \mpp{} $\gf:\sss\to\sss$,
and $S_{\mathrm{mpb}}$ is the subset of all \mpb{s}.
The closure may here be taken either for the cut norm or for the $L^1$ norm.
\end{corollary}
\begin{proof}
  For the closure in cut norm, this follows from 
\refT{Tcut}\ref{tcutmpp} and \ref{tcutrearr}.
By \refR{Rcutl}, the same holds for the closure in $L^1$ norm and the
equivalence class \set{W'\in\cW(\sss):\dl(W',W)=0}. However, by
\refT{Teq1} below, $\dl$ and $\dcut$ define the same equivalence classes.
\end{proof}

\refR{Rbaddiscrete} shows that the (equivalent) definitions in
\refT{Tcut}\ref{tcut1}--\ref{tcutxy2} are  
the only ones useful for general \ps{s}. Another advantage of them is that,
as shown by \citet{BRmetrics},
the infima are attained, at least for Borel spaces. 
(This is not true in general for the versions in
\ref{tcutmpp}--\ref{tcutinterval}, not even in the special case when the
infimum is 0, see \refE{Ebad0} below.)

\begin{theorem}\label{Tattained}
Let\/ $W_1$ and $W_2$ be two kernels defined on Borel
\ps{s} $(\sss_1,\mu_1)$ and
$(\sss_2,\mu_2)$, respectively.
Then the infima in \refT{Tcut}\ref{tcut1}--\ref{tcutxy2} are attained.
In other words, there exists a probability measure $\mu$ on $\sssxy$
with marginals $\mu_1$ and $\mu_2$,
and thus a corresponding coupling $(\pi_1,\pi_2)$,
such that $\dcut(W_1,W_2)= \cnx{\sssxy,\mu}{W_1^{\pi_1}-W_2^{\pi_2}}$.
\end{theorem}

\begin{proof}
  By \refT{Tborel} and \refR{Rcantor}, every Borel measurable space is
  either countable or 
  isomorphic to the Cantor cube $\cC\=\setoi^\infty$. 
Hence, we may without loss of generality assume that each of the two spaces
$\sss_\ell$ (where, as in the rest of the proof, $\ell=1,2$) is either a
finite set, 
the countable set $\set0\cup\set{1/n:n\in\bbN}$ or $\cC$, 
equipped with some \pmm{} $\mu_\ell$.
Note that in every case $\sss_\ell$ is a compact metric space.

For $\ell_1,\ell_2\in\set{1,2}$,
Let $\casll$ be the set of all step functions on
$\sss_{\ell_1}\times\sss_{\ell_2}$ 
corresponding to partitions $\sss_{\ell_m}=\bigcup_i A_{im}$ 
where every part $A_{im}$ is clopen (closed and open) in $\sss_{\ell_m}$,
$m=1,2$. (We extend here the definition of step functions on
$\sss\times\sss$
to products of two different spaces in the natural way.)
For the spaces we consider, $\casll$ is dense in
$L^1(\sss_{\ell_1}\times\sss_{\ell_2},\mu)$ for any \pmm{} $\mu$ on the product.
(This is the reason why we replaced $\oi$ by the totally disconnected space
$\cC$. It is possible to use $\oi$ instead, with minor modifications, see
\cite{BRmetrics}.) 

Denote the integral in \refT{Tcut}\ref{tcutxy2} by $\Phi(W_1,W_2,f,g,\mu)$.
By \refT{Tcut}, there exist \pmm{s} $\nu_n$ on $\sssxy$ such that
\begin{equation}\label{att15}
\sup_{\normoo{f},\normoo{g}\le1}|\Phi(W_1,W_2,f,g,\nu_n)|<\dcut(W_1,W_2)+1/n.
\end{equation}
(For $\dcutone$, we tacitly assume that $f,g\ge0$.)

Since $\sss_1$ and $\sss_2$ are compact metric spaces,  $\sssxy$ is
too. Hence, the set of \pmm{s} on $\sssxy$ is compact and metrizable (see
\cite{Billingsley}), so there exists a subsequence of $(\nu_n)$ that
converges (in the usual weak topology) to some \pmm{} $\nu$ on $\sssxy$.
We consider in the sequel this subsequence only.

Let $\eps>0$. By the remarks above, we may find $W_\ell'\in\casl$
with $\norm{W_\ell-W_\ell'}_{L^1(\sss_\ell\times \sss_\ell)}<\eps$, and
hence, assuming $\normoo{f},\normoo{g}\le1$,
\begin{equation}\label{att16}
  \begin{split}
  |\Phi(W_1,W_2,f,g,\nu)|
&\le   |\Phi(W_1',W'_2,f,g,\nu)|+\cn{W_1-W'_1}+\cn{W_2-W'_2}
\\&
\le  | \Phi(W_1',W'_2,f,g,\nu)|+2\eps	
  \end{split}
\raisetag{\baselineskip}
\end{equation}
and similarly, for every $n$ and every $f,g$ with $\normoo{f},\normoo{g}\le1$,
\begin{equation}\label{att17}
  |\Phi(W'_1,W'_2,f,g,\nu_n)|
\le   |\Phi(W_1,W_2,f,g,\nu_n)|+2\eps	.
\end{equation}
Since $W'_1$ and $W'_2$ are step functions, they are bounded, so there
exists some $M$ with $\normoo{W'_\ell}\le M$.
For any $f$ and $g$ with $\normoo{f},\normoo{g}\le1$, we may similarly find
$f'$ and $g'$ in $\cA(\sssxy)$, with $\normoo{f},\normoo{g}\le1$, such that
$\norm{f-f'}_{L^1(\nu)},\norm{g-g'}_{L^1(\nu)}\le\eps/M$.
It follows that
\begin{equation}\label{att18}
  |\Phi(W'_1,W'_2,f,g,\nu)|
\le   |\Phi(W'_1,W'_2,f',g',\nu)|+4\eps.	
\end{equation}
 
Since $W'_1,W'_2,f',g'$ all are step functions in the sets $\cA$, the
integral $\Phi(W'_1,W'_2,f',g',\mu)$ can be written as a linear combination of
integrals 
$$\int_{A_i\times B_j\times A_k\times B_m}\dd\mu(x_1,x_2)\dd\mu(y_1,y_2)
=\mu(A_i\times B_j)\mu(A_k\times B_m),
$$ 
where further the sets $A_i,B_j,A_k,B_m$ are clopen.
Hence each term, and thus $\Phi(W'_1,W'_2,f',g',\mu)$, is a continuous
functional of $\mu$; consequently,
\begin{equation}\label{atta}
\Phi(W'_1,W'_2,f',g',\nu_n)\to\Phi(W'_1,W'_2,f',g',\nu).  
\end{equation}
By \eqref{att17} and \eqref{att15}, 
\begin{equation}
|\Phi(W'_1,W'_2,f',g',\nu_n)|<\dcut(W_1,W_2)+1/n+2\eps,  
\end{equation}
and thus \eqref{atta} yields
\begin{equation}
|\Phi(W'_1,W'_2,f',g',\nu)|\le \dcut(W_1,W_2)+2\eps
\end{equation}
and, by \eqref{att16} and \eqref{att18},
\begin{equation}
|\Phi(W_1,W_2,f,g,\nu)|
\le|\Phi(W'_1,W'_2,f',g',\nu)|+6\eps
\le \dcut(W_1,W_2)+8\eps.
\end{equation}
Since $\eps$ is arbitrary, we thus obtain
$|\Phi(W_1,W_2,f,g,\nu)|\le \dcut(W_1,W_2)$, and 
\begin{equation}
\cnx\nu{W_1-W_2}=
\sup_{\normoo{f},\normoo{g}\le1}|\Phi(W_1,W_2,f,g,\nu)|\le\dcut(W_1,W_2),
\end{equation}
which shows that equality is attained in
\refT{Tcut}\ref{tcutxy}--\ref{tcutxy2}
by $\nu$, and in \refT{Tcut}\ref{tcut1} by the coupling $(\pi_1,\pi_2)$
defined on $(\sssxy,\nu)$.
\end{proof}

The assumption that the spaces are Borel (or Lebesgue, see \refR{RcutLeb})
really is essential here, even when
the infimum $\dcut(W_1,W_2)=0$;
in \refE{Ebadequal} we will see an example of two equivalent kernels such
that none of the infima in \refT{Tcut} is attained.

\section{Representation on $\oi$}\label{Srep}

As said in the introduction, many papers consider only kernels or graphons
on $\oi=(\oi,\gl)$. This is justified by the fact that every kernel
[graphon] is equivalent to such a kernel [graphon]. 
(See \cite{SJ210} for a generalization.)

\begin{theorem}\label{Trep}
  Every kernel [graphon] on a \ps{} $(\sss,\cF,\mu)$ is equivalent to a
  kernel [graphon] on $(\oi,\gl)$.
\end{theorem}

\begin{corollary}
The quotient space	$\bwx\=\bigcup_\sss\www(\sss)\big/\equ$,
which as said above can be identified with
the space of graph limits, can as well be defined
$\bwx\=\www(\oi)/\equ$.
\end{corollary}

Before proving \refT{Trep}, we
prove a partial result. 

\begin{lemma}\label{Lrep}
  Every kernel [graphon] on a \ps{} $(\sss,\cF,\mu)$ is a pull-back of a
  kernel [graphon] on  some Borel \ps. 
\end{lemma}

\begin{proof}
  Let $W:\sssq\to\ooo$ be a kernel. 
Since $W$ is measurable, each set $E_r\=\set{(x,y):W(x,y)<r}$, where $r\in\bbR$,
belongs to $\cF\times\cF$, and it follows that there exists a 
countable subset $\cA_r\subseteq\cF$ such that 
$E_r\in\cF(\cA_r)\times \cF(\cA_r)$, where
$\cF(\cA_r)$ is the $\gs$-field generated by $\cA_r$.
Hence, if
$\cfo$ is the $\gs$-field generated by the countable set
$\cA\=\bigcup_{r\in\bbQ}\cA_r$, then $\cfo\subseteq\cF$ and $W$ is
$\cfo\times\cfo$-measurable. 

List the elements of $\cA$ as \set{A_1,A_2,\dots}. (If $\cA$ is finite, we
for convenience repeat some element.) 
Let $\cC\=\setoi^\infty$ be the Cantor cube (see \refR{Rcantor})
and define a map
$\gf:\sss\to\cC\=\setoi^\infty$ 
by $\gf(x)=(\ett{x\in A_i})_{i=1}^\infty$.
Let $\nu$ be the \pmm{} on $\cC$  that makes $\gf:\sss\to\cC$ \mpp, see
\refR{Rpush}.

The $\gs$-field on $\sss$ generated by $\gf$ equals $\cfo$, and thus 
the $\gs$-field on $\sss\times\sss$ generated by $(\gf,\gf):\sss^2\to\cC^2$ 
equals $\cfo\times\cfo$. Since $W$ is measurable for this $\gs$-field, 
$W$ equals $V\circ(\gf,\gf)=V^\gf$ for some measurable $V:\cC^2\to\ooo$.
Since $W$ is symmetric, we may here replace $V(x,y)$ by
$\frac12\bigpar{V(x,y)+V(y,x)}$ and thus assume that also $V$ is symmetric. 
Hence, $V$ is a kernel on $\cC$ and $W=V^\gf$. 
If $W$ is a graphon, we may assume that $V:\sssq\to\oi$, and thus $V$ too is
a graphon.
This proves the result with the Borel \ps{} $(\cC,\nu)$.
\end{proof}

\begin{proof}[Proof of \refT{Trep}]
  Let $W$ be a kernel on some \ps{} $\sss$. By \refL{Lrep}, $W\equ V$ for
  some kernel $V$ on a Borel \ps{} $(\sss_1,\nu)$. 
(With $\sss_1=\cC$ in the proof above.)
If $\nu$ is atomless, the  result follows by \refT{Tborelp}. In general, 
let $\sss_2\=\sss\times\oi$, with product measure $\nu_2$.
The projection $\pi:\sss_2\to\sss_1$ is \mpp, so $V\equ V_2\=V^\pi$.
Moreover, $(\sss_2,\nu_2)$ is an atomless Borel \ps, so by \refT{Tborelp}
there exists a \mpp{} bijection $\psi:\oi\to\sss_2$. Hence 
$U:=V_2^\psi$ is a kernel [graphon] on $\oi$ and $U\equ V_2\equ V\equ  W$.
\end{proof}

\begin{remark}
If $\mu$ is atomless, we may by \refL{Lfinns} find an increasing family of sets
$B_r\subseteq\sss$, $r\in\oi$, such that $\mu(B_r)=r$.
In the construction in the proof of \refL{Lrep}, we may add each $B_r$ with
rational $r$ to the family $\cA$. Then the measure $\nu$ on $\cC$ is
atomless, because if $x$ were an atom, then $E:=\gf\qw\set x$ would be a
subset of 
$\sss$ with $\mu(E)>0$ such that for each rational $r$, either $E\subseteq
B_r$ or $E\cap B_r=\emptyset$, but this leads to a contradiction as in the
proof of \refL{Latomless}.
Consequently, we then can use \refT{Tborelp} directly to find a \mpp{}
bijection $\psi:\oi\to\cC$, and a kernel $U\=V^\psi$ on $\oi$ such that 
$W=V^\gf=U^{\psi\qw\circ\gf}$.
Consequently, every kernel on an atomless \ps{} $(\sss,\mu)$ is a pull-back
of a kernel on $\oi$, which combines and improves \refL{Lrep} and
\refT{Trep} in this case.
(Conversely, by \refL{Latomless}, no kernel on a space $(\sss,\mu)$ with
atoms is a pull-back of a kernel on $\oi$.)
\end{remark}

\section{Equivalence}\label{Sequiv}

We have seen that if $W_1$ and $W_2$ are two kernels on some \ps{s} $\sss_1$
and $\sss_2$, and $W_1=W_2^\gf$ (or just
$W_1=W_2^\gf$ \aex) for some \mpp{} $\gf:\sss_1\to\sss_2$, 
then $W_1\equ W_2$.
The converse does not hold, as shown by the following standard examples
\cite{BCL:unique}.

\begin{example}
  \label{Ebad0}
 Let $\gf:\oi\to\oi$ be given by $\gf(x)=2x\mod 1$.
Take $W_1(x,y)=xy$, and $W_2\=W_1^{\gf}$. Then $W_1$ and $W_2$
are graphons on $\oi$, and $\dcut(W_1,W_2)=0$.
However, there is no \mpp{} $\psi:\oi\to\oi$ such that $W_1=W_2^\psi$ \aex,
and as a consequence, the infima in
\refT{Tcut}\ref{tcutmpp}--\ref{tcutinterval} are not attained. (See
\refL{Lcut0}.)  
In fact, if such a $\psi$ existed, then
$W_1=(W_1^{\gf})^\psi=W_1^{\gf\circ\psi}$ \aex, which 
implies (\eg{} by considering the marginal $\intoi W(x,y)\dd y=x/2$)
that $\gf(\psi(x))=x$ \aex, and thus $\psi(x)\in\set{x/2,x/2+1/2}$ a.e.
However, if $E\=\psi\qw([0,1/2])$, it follows that for any $a$ and $b$ with
$0<a<b<1$, 
$E\cap[a,b]=\psi\qw([a/2,b/2])$, so $\gl(E\cap[a,b])/(b-a)=1/2$.
In particular, for every $x\in(0,1)$, the density
$\lim_{\eps\to0}\gl(E\cap(x-\eps,x+\eps))/2\eps=1/2$.
On the other hand, by the Lebesgue density theorem, this density is 1 for
\aex{} $x\in E$ and 0 for \aex{} $x\notin E$, a contradiction.
\end{example}

\begin{example}
  \label{Ebadmn}
More generally, let $\gf:\oi\to\oi$ be given by $\gf_n(x)=nx\mod1$, and define
$W_n\=W^{\gf_n}$ with the same $W$ as in \refE{Ebad0}.
If $W_n=W_m^{\psi}$ \aex, then $m\psi(x)\equiv nx \mod 1$ \aex.
Let $E\=\psi\qw(0,1/m)$. Then, for $0<a<b<1/m$,
$\psi\qw([a,b])=\bigcup_{j=0}^{n-1} E\cap([ma/n,mb/n]+j/n)$ (\aex) and thus, if
$\psi$ is \mpp, 
\begin{equation*}
b-a=  \gl\lrpar{\psi\qw([a,b])}
=\sum_{j=0}^{n-1} \gl\lrpar{E\cap \Bigsqpar{\frac{ma+j}n,\frac{mb+j}n}}.
\end{equation*}
Divide by $b-a$, take
$a=x-\eps$ and
$b=x+\eps$, and let $\eps\to0$. The Lebesgue
differentiation theorem implies that for
\aex{} $x\in(0,1/m)$, 
\begin{equation*}
1
=\frac mn \sum_{j=0}^{n-1} \Bigett{\frac{mx+j}n\in E}.  
\end{equation*}
Since the sum is an integer for each $x$, this implies
that $n$ is a multiple of $m$. Conversely, if $n=m\ell$ for an integer
$\ell$, then $\gf_n=\gf_m\circ\gf_\ell$, and thus
$W_n=(W^{\gf_m})^{\gf_\ell}=W_{m}^{\gf_\ell}$.
Consequently, all $W_n$ are equivalent (being pull-backs of $W_1$), and
there exists a \mpp{} $\psi:\oi\to\oi$ such that $W_n=W_m^{\psi}$ \aex{}
if and only if $n$ is a multiple of $m$.
In particular, $W_2$ and $W_3$ are equivalent, but neither of them is a
pull-back of the other.
\end{example}

However, equivalence is characterized by sequences of pull-backs.
We begin with a simple result.
\begin{theorem}\label{Tchain}
  Let\/ $W'$ and $W''$ be kernels defined of \ps{s} $\sss'$ and $\sss''$.
Then $W'\equ W''$, \ie{} $\dcut(W',W'')=0$, if and only if there exists a
finite sequence of kernels $W_i$ defined on \ps{s} $\sss_i$,
$i=0,\dots,n$, with $W_0=W'$ and $W_n=W''$, such that for each $i\ge1$,
either 
$W_{i-1}=W_{i}^{\gf_i}$ \aex{} for some \mpp{}
$\gf_i:\sss_{i-1}\to\sss_{i}$, 
or
$W_i=W_{i-1}^{\psi_i}$ \aex{} for some \mpp{}
$\psi_i:\sss_i\to\sss_{i-1}$.
\end{theorem}

\begin{proof}
Suppose that $W'\equ W''$. We show that we can construct such a sequence
with $n=4$. We thus take $W_0\=W'$ and $W_4\=W''$.
By \refL{Lrep}, we can find $W_1$ and $W_3$ on Borel \ps{s} $\sss_1$ and
$\sss_3$ such that $W_0=W_1^{\gf_1}$ and $W_4=W_3^{\psi_4}$ for some \mpp{}
$\gf_1$ and $\psi_4$. Then $W_1\equ W_0\equ W_4\equ W_3$, so
$\dcut(W_1,W_3)=0$.
By \refT{Tattained}, there exists a \pmm{} $\mu$ on $\sssxz$ such that 
$\cnx{\sssxz,\mu}{W_1^{\pi_1}-W_3^{\pi_3}}=0$, 
where $\pi_i$ is the projection onto $\sss_i$.
Thus, by \refL{Lcut0}, $W_1^{\pi_1}=W_3^{\pi_3}$ a.e.
Hence, we can take $\sss_2\=(\sssxz,\mu)$, $\psi_2\=\pi_2$, $\gf_3\=\pi_3$ and
$W_2\=W_1^{\psi_2}=W_1^{\pi_2}$.

The converse is obvious by \refE{Eequ1} and \refC{C0}. 
\end{proof}

\refE{Ebadmn} shows that we cannot in general do with a single pull-back in
\refT{Tchain}. 
However, we can always do with a chain of length 2 in \refT{Tchain}.
In fact, \citet{BCL:unique} proved the following, more precise and much more
difficult, result.
(We will not use this theorem later; the simpler \refT{Tchain} is sufficient
for our applications.)

\begin{theorem}\label{Teq2}
  Let\/ $W_1$ and $W_2$ be kernels defined of \ps{s} $\sss_1$ and $\sss_2$.
Then $W_1\equ W_2$, \ie{} $\dcut(W_1,W_2)=0$, if and only if there exists a
kernel $W$ on some \ps{} $\sss$ 
and \mpp{} maps $\gf_j:\sss_j\to\sss$
such that $W_j=W^{\gf_j}$ \aex, $j=1,2$.

We can always take $\sss$ to be a Borel space.
If\/ $\sss_1$ and $\sss_2$ are atomless, we may take $\sss=\oi$.
\end{theorem}

\begin{proof}
It suffices to prove the theorem for graphons $W_1$ and $W_2$; the general
case follows easily by  considering the transformations $W_1/(1+W_1)$ and
$W_2/(1+W_2)$. 

We give a proof in \refS{Scanon}. (Except for the final statement, which is
shown below.) 
See also \citet{BCL:unique} for the long and technical original proof.
In their formulation, the space $\sss$ is constructed as a Lebesgue space,
and the maps $\gf_j$ are only assumed to be measurable from the completions
$(\sss_j,\hcF_j,\mu_j)$ of $\sss_j$ to $\sss$.
However, this is easily seen to be equivalent:
If $\sss$ is such a Lebesgue space, then $\sss=(\sss,\cF,\mu)$ is the
completion of some Borel space
$\sss_0=(\sss,\cF_0,\mu)$. We may replace $W$ by an \aex{} equal kernel
that is $\cF_0\times\cF_0$-measurable, \ie, a kernel on the Borel space
$\sss_0$.
Further, since every Borel measurable space is isomorphic to a Borel subset
of $\oi$, see \refT{Tborel}, the map $\gf_j:\sss_j\to\sss_0$ which is
$\hcF_j$-measurable, is \aex{} equal to an $\cF_j$-measurable map $\gf_j'$.
Replacing $\sss$ by $\sss_0$ and $\gf_j$ by $\gf_j'$, we obtain the result
as stated above, with $\sss$ Borel.

For the final statement, suppose that $\sss_1$ and $\sss_2$ are atomless,
and let $W$, $\gf_j$ and $\sss$ be as in the first part of the theorem,
with $\sss$ Borel.
Suppose that $\sss$ has atoms, \ie, points $a\in\sss$ with $\mu\set a>0$.
Replace each such point $a$ by a set $I_a$ which is a copy of the interval
$[0,\mu\set{a}]$ (with Borel $\gs$-field and Lebesgue measure), and let
$\sss'$ be the resulting Borel \ps. There is an obvious map
$\pi:\sss'\to\sss$, mapping each $I_a$ to $a$ and being the identity elsewhere,
and we let $W'\=W^\pi$.
For each atom $a$, and $j=1,2$, let $A_{aj}\=\gf_j\qw(a)\subseteq\sss_j$.
Then $A_{aj}$ is an atomless measurable space, 
and by \refL{Lfinns2} (and scaling), there
is a \mpp{} map $A_{aj}\to I_a$. Combining these maps and the original $\gf_j$,
we find a \mpp{} map $\gf_j':\sss_j\to\sss'$ such that
$\gf_j=\pi\circ\gf_j'$, and thus $W_j=W^{\gf_j}=(W')^{\gf'_j}$ \aex.
Finally, $\sss'$ is an atomless Borel \ps{}, and may thus be replaced by
$\oi$ by \refT{Tborelp}.
\end{proof}

\begin{remark}\label{RintoL}
With a Lebesgue space $\sss$, it is both natural
and necessary to consider maps $\sss_j\to\sss$ that are measurable
with respect to the completion of $\sss_j$, as done in \cite{BCL:unique}.
For example, if $\sss_1=\sss_2=\oi$ with the Borel
$\gs$-field and $W_1(x,y)=W_2(x,y)=xy$, we can take $\sss=\oi$ and
$\gf_j=\iota$,  but if we equip $\sss$ with the Lebesgue
$\gs$-field, then $\gf_j$ is not measurable $\sss_j\to\sss$
(and cannot be modified on a null set to become measurable).
This is just a trivial technicality that is no
real problem, and as seen in the proof above, it can be avoided by using
Borel spaces.
\end{remark}

\refT{Teq2} says that a pair of equivalent graphons always are pull-backs of
a single graphon. We may also try to go in the opposite direction and try to
find a common pull-back of two equivalent graphons. 
As shown by \citet{BCL:unique}, this
is not always possible, see \refE{Ebadequal} below, but it is possible
for  graphons defined on Borel or Lebesgue spaces.
We state this, in several versions, in the next theorem, together
with conditions under which $W_1$ is a pull-back or rearrangement of $W_2$.
(Recall that \refE{Ebadmn} shows that this does not hold in general, 
not even for a nice Borel space like $\oi$.)

If $W$ is kernel defined on a \ps{} $\sss$, we say following
\cite{BCL:unique} that
$x_1,x_2\in\sss$ are \emph{twins} (for $W$) if
$W(x_1,y)=W(x_2,y)$ for
\aex{} $y\in\sss$. We say that $W$ is
\emph{almost twinfree} if there exists a null set
$N\subset\sss$ such that there are no twins
$x_1,x_2\in\sss\setminus N$ with $x_1\neq x_2$.

Various parts of the following theorem are given,
at least for the standard case of graphons on
$\sss_1=\sss_2=\oi$,
in
\citet{SJ209} (as a 
consequence of Hoover's equivalence theorem for representations of
exchangeable arrays \cite[Theorem 7.28]{Kallenberg:symmetries}), 
\citet{BRmetrics}, and
\citet{BCL:unique}.
A similar theorem in the related case of partial orders is given in
\cite{SJ224}.

\begin{theorem}\label{TU}
Let\/ $W_1$ and $W_2$ be kernels defined on Borel \ps{s} $(\sss_1,\mu_1)$ and
$(\sss_2,\mu_2)$.
Then the following are equivalent.
  \begin{romenumerate}
\item\label{TUequ}
$W_1\equ W_2$.
\item\label{TUcoup}
There exist a coupling $(\gf_1,\gf_2)$, \ie, two measure preserving
maps $\gf_j:\sss\to\sss_j$, $j=1,2$, for some \ps{} $\sss$, such that
$W_1^{\gf_1}=W_2^{\gf_2}$ \aex, \ie,
$W_1\bigpar{\gf_1(x),\gf_1(y)}=W_2\bigpar{\gf_2(x),\gf_2(y)}$ \aex.
\item\label{TUphi}
There exist measure preserving maps $\gf_j:\oi\to\sss_j$, $j=1,2$, such that
$W_1^{\gf_1}=W_2^{\gf_2}$ \aex, \ie,
$W_1\bigpar{\gf_1(x),\gf_1(y)}=W_2\bigpar{\gf_2(x),\gf_2(y)}$ \aex{}
on $\oi^2$.
\item\label{TUpsi}
There exists a \mpp{} map $\psi:\sss_1\times\oi\to\sss_2$
such that $W_1^{\pi_1}=W_2^{\psi}$ \aex, 
where $\pi_1:\sss_1\times\oi\to\sss_1$ is the projection, \ie,
$W_1(x,y)=W_2\bigpar{\psi(x,t_1),\psi(y,t_2)}$ for \aex{}
$x,y\in\sss_1$ and  $t_1,t_2\in\oi$.
\item\label{TUxy}
There exists a \pmm{} $\mu$ on $\sssxy$ with marginals $\mu_1$ and $\mu_2$
such that $W_1^{\pi_1}=W_2^{\pi_2}$ \aex{} on $(\sssxy)^2$, \ie, 
$W_1(x_1,y_1)=W_2(x_2,y_2)$ for $\mu$-\aex{} $(x_1,x_2),(y_1,y_2)\in\sssxy$.
  \end{romenumerate}

If\/ $W_2$ is almost twinfree, then these are also equivalent to:
\begin{romenumerateq}
\item\label{TUtwin1}
There 
exists a measure preserving map $\gf:\sss_1\to\sss_2$ such
that
$W_1=W_2^{\gf}$ \aex, \ie{}
$W_1{(x,y)}=W_2\bigpar{\gf(x),\gf(y)}$ \aex{}
on $\sss_1^2$.  
\end{romenumerateq}

If both\/ $W_1$ and\/ $W_2$ are almost twinfree, then these are also
equivalent to: 
\begin{romenumerateq}
\item\label{TUtwin2}
There exists a measure preserving map $\gf:\sss_1\to\sss_2$ such
that
$\gf$ is a bimeasurable bijection of\/
$\sss_1\setminus N_1$ onto\/ $\sss_2\setminus
N_2$ for some null sets $N_1\subset\sss_1$ and
$N_2\subset\sss_2$, and\/ 
$W_1=W_2^{\gf}$ \aex, \ie{}
$W_1{(x,y)}=W_2\bigpar{\gf(x),\gf(y)}$ \aex{}
on $\sss_1^2$.  
If further $(\sss_2,\mu_2)$ is atomless, for example if
$\sss_2=\oi$, then we may take $N_1=N_2=\emptyset$, so $W_1$ is a
rearrangement of $W_2$ and vice versa.
\end{romenumerateq}

The same results hold if\/ $\sss_1$ and $\sss_2$ are Lebesgue spaces, provided
in \ref{TUphi} $\oi$ is equipped with the Lebesgue $\gs$-field, and in
\ref{TUpsi} $\sss_1\times\oi$ has the completed $\gs$-field.
 \end{theorem}

\begin{proof}

We assume that $\sss_1$ and $\sss_2$ are Borel spaces. The Lebesgue space
case follows immediately from this case by replacing $W_1$ and $W_2$ by
(\aex{} equal) Borel kernels, see \refR{RBL2}.

We may also, when convenient, assume that $W_1$ and $W_2$ are graphons by using
again the transformations $W_1/(1+W_1)$ and $W_2/(1+W_2)$.

First note that any of \ref{TUphi}--\ref{TUtwin2} is a special case of
\ref{TUcoup}, and that \ref{TUcoup} implies $W_1\equ W_1^{\gf_1}\equ
W_2^{\gf_2}\equ W_2$; thus any of \ref{TUcoup}--\ref{TUtwin2} implies
\ref{TUequ}. 
We turn to the  converses.

\ref{TUequ}$\implies$\ref{TUcoup},\ref{TUxy}:
Assume $W_1\equ W_2$, \ie, $\dcut(W_1,W_2)=0$.
First, 
by \refT{Tattained}, there exists a coupling $(\gf_1,\gf_2)$ such that
$\cn{W_1^{\gf_1}-W_2^{\gf_2}}= \dcut(W_1,W_2)=0$, and thus, by \refL{Lcut0},
$W_1^{\gf_1}=W_2^{\gf_2}$ \aex. Consequently, \ref{TUcoup} holds.
Moreover, by the same theorem and \refT{Tcut}\ref{tcutxy}, we may take this
coupling $(\gf_1,\gf_2)$ as the projections $(\pi_1,\pi_2)$ for a suitable
measure $\mu$ on $\sssxy$, which shows \ref{TUxy}.

\ref{TUxy}$\implies$\ref{TUphi}:
Since $(\sssxy,\mu)$ is a Borel \ps, \refT{Toi2borel} shows that there
exists a \mpp{} map $\psi:\oi\to\sssxy$, and then
$(\pi_1\circ\psi,\pi_2\circ\psi)$ is a coupling 
defined on $\sss=\oi$, which shows \ref{TUphi}.
(Alternatively, \ref{TUequ}$\implies$\ref{TUphi} follows also easily by
\refT{Toi2borel} from the special case 
$\sss_1=\sss_2=\oi$ showed in \cite{SJ209}.)

\ref{TUequ}$\implies$\ref{TUpsi}:
By \refT{Toi2borel}, there exist
measure preserving maps $\gam_j:\oi\to\sss_j$, $j=1,2$.
Then $W_1^{\gam_1}$  and $W_2^{\gam_2}$ are kernels on $\oi$, 
and 
$W_1^{\gam_1}\equ W_1\equ W_2\equ W_2^{\gam_2}$.
The equivalence
\ref{TUequ}$\iff$\ref{TUpsi}
was shown (for graphons, which suffices as remarked above) in \cite{SJ209}
in the special case 
$\sss_1=\sss_2=\oi$, based on 
\cite[Theorem 7.28]{Kallenberg:symmetries}, and thus 
\ref{TUpsi} holds for $\tWx1$ and $\tWx2$.
In other words,
there exists a measure preserving function $h:\oi^2\to\oi$ such that
$\tWx1(x,y)=\tWx2\bigpar{h(x,z_1),h(y,z_2)}$ for \aex{}
$x,y,z_1,z_2\in\oi$. 
By \refL{LSJ224-7.2} below
(applied to $(\sss_1,\mu_1)$ and $\gam_1$),
there exists a measure preserving map $\ga:\sss_1\times\oi\to\oi$ such that
$\gam_1(\ga(s,u))=s$ a.e.
Hence, for \aex{} $x,y\in\sss_1$ and $u_1,u_2,z_1,z_2\in\oi$,
\begin{equation*}
  \begin{split}
W_1(x,y)
&=
W_1\bigpar{\gam_1\circ \ga(x,u_1),\gam_1\circ \ga(y,u_2)}	
=
\tWx1\bigpar{ \ga(x,u_1), \ga(y,u_2)}	
\\&
=
\tWx2\bigpar{h(\ga(x,u_1),z_1),h(\ga(y,u_2),z_2)}
\\&
=
W_2\bigpar{\gam_2\circ h(\ga(x,u_1),z_1),\gam_2\circ h(\ga(y,u_2),z_2)}.
  \end{split}
\end{equation*}
Finally, let $\beta=(\beta_1,\beta_2)$ be a measure preserving map
$\oi\to\oi^2$, 
and define
$\psi(x,t)\=\gam_2\circ{h\bigpar{\ga(x,\beta_1(t)),\beta_2(t)}}$.

\ref{TUpsi}$\implies$\ref{TUtwin1}:
Since, for \aex{} $x,y,t_1,t_2,t_1'$,
\begin{equation*}
W_2\bigpar{\psi(x,t_1),\psi(y,t_2)}=
W_1(x,y)=W_2\bigpar{\psi(x,t_1'),\psi(y,t_2)}  
\end{equation*}
and $\psi$ is  
measure preserving, it follows that for \aex{}
$x,t_1,t_1'$, $\psi(x,t_1)$ and $\psi(x,t_1')$
are twins for $W_2$. If $W_2$ is almost twin-free, with exceptional null
set $N$, then further
$\psi(x,t_1),\psi(x,t_1')\notin N$ for \aex{}
$x,t_1,t_1'$, since $\psi$ is measure preserving, and consequently
$\psi(x,t_1)=\psi(x,t_1')$ for \aex{} $x,t_1,t_1'$.
It follows that we can choose a fixed $t_1'$ (almost every choice
will do) such that 
$\psi(x,t)=\psi(x,t_1')$ for \aex{} $x,t$. Define
$\gf(x)\=\psi(x,t_1')$.
Then $\psi(x,t)=\gf(x)$ for \aex{} $x,t$, which in particular
implies that $\gf$ is measure preserving, and \ref{TUpsi}
yields $W_1(x,y)=W_2\bigpar{\gf(x),\gf(y)}$ a.e.

\ref{TUtwin1}$\implies$\ref{TUtwin2}:
Let $N'\subset\sss_1$ be a null set such that if
$x\notin N'$, then $W_1(x,y)=W_2(\gf(x),\gf(y))$
for \aex{} $y\in\sss_1$.
If
$x,x'\in\sss_1\setminus N'$ and
$\gf(x)=\gf(x')$, then $x$ and $x'$ are twins
for $W_1$. Consequently, if $W_1$ is almost twinfree
with exceptional null set $N''$, then $\gf$ is injective
on $\sss_1\setminus N_1$ with $N_1\=N'\cup N''$. 
Since $\sss_1\setminus N_1$ and
$\sss_2$ are Borel spaces, \refT{Tinjection} shows that the injective map
$\gf:\sss_1\setminus N_1\to\sss_2$ has measurable
range and is a bimeasurable bijection 
$\gf:\sss_1\setminus N_1\to\sss_2\setminus N_2$
for some measurable set $N_2\subset\sss_2$. Since
$\gf$ is measure preserving, $\mu_2(N_2)=0$. 
 
If $\sss_2$ has no atoms, we may take an uncountable null set
$N_2'\subset\sss_2\setminus N_2$. Let
$N_1'\=\gf\qw(N_2')$. Then $N_1\cup N_1'$ and
$N_2\cup N_2'$ are uncountable Borel spaces so 
they are isomorphic and
there is a
bimeasurable bijection $\eta:N_1\cup N_1'\to N_2\cup
N_2'$. Redefine $\gf$ on $N_1\cup N_1'$ so that
$\gf=\eta$ there; then $\gf$ becomes a bijection $\sss_1\to\sss_2$.
\end{proof}

\begin{remark}
  A probabilistic reformulation of \ref{TUcoup}, along the lines of
  \refR{Rcoupling},  is that there exists a
  coupling $(X,Y)$ of random variables with the distributions $\mu_1$ on
  $\sss_1$ and $\mu_2$ on $\sss_2$, such that if $(X',Y')$ is an independent
  copy of $(X,Y)$, then $W_1(X,X')=W_2(Y,Y')$ a.s.
Similarly, \ref{TUxy} says that there exists a distribution (\ie, \pmm)
$\mu$ on $\sssxy$ with 
marginals $\mu_1$ and $\mu_2$ such that if $(X,Y)$ and $(X',Y')$ are
independent with the same distribution $\mu$, then $W_1(X,X')=W_2(Y,Y')$ a.s.
\cite{BCL:unique}.
\end{remark}

\begin{remark}
  In \ref{TUpsi}, the seemingly superfluous variables $t_1$ and $t_2$ act as
  extra randomization; \ref{TUpsi} thus yields a kind of ``randomized
  pull-back'' using a ``randomized \mpp{} map'' $\psi$, even when no 
suitable \mpp{} map as in \ref{TUtwin1} exists. It is an instructive
  exercise to see how this works for \refE{Ebad0}; we leave this to the reader. 
\end{remark}

The proof above uses the following consequence of the transfer theorem
\cite[Theorem 6.10]{Kallenberg}. 

\begin{lemma}
  \label{LSJ224-7.2}
Suppose that $(\sss,\mu)$ is a Borel probability space and that
$\gamma:\oi\to\sss$ is a measure preserving function.
Then there exists a measure preserving function $\ga:\sss\times\oi\to\oi$
such that $\gamma\bigpar{\ga(s,y)}=s$ for $\mu\times \gl$-\aex{}
$(s,y)\in\sss\times\oi$. 
\end{lemma}

\begin{proof}
Let $\eta:\oi\to\oi$ and $\txi:\sss\to\sss$ be the identity maps
  $\eta(x)=x$, $\txi(s)=s$, and let $\xi=\gamma:\oi\to\sss$.
Then $(\xi,\eta)$ is a pair of random variables, 
defined on the probability space $(\oi,\gl)$,
with values in $\sss$ and $\oi$, respectively;
further, $\txi$ is a random variable defined on $(\sss,\mu)$ with
  $\txi\eqd\xi$.
By the transfer theorem \cite[Theorem 6.10]{Kallenberg}, there
  exists a measurable function $\ga:\sss\times\oi\to\oi$ such that 
if $\teta(s,y)\=\ga(\txi(s),y)=\ga(s,y)$, then $(\txi,\teta)$ is
  a pair of random 
  variables defined on $\sss\times\oi$ with
$(\txi,\teta)\eqd(\xi,\eta)$.
Since $\xi=\gamma(\eta)$, this implies $\txi=\gamma(\teta)$ \aex, 
and thus $s=\txi(s)=\gamma\bigpar{\ga(s,y)}$ \aex{}
\end{proof}

There are several other, quite different, characterizations of equivalence.
We  give several important conditions from
\cite{BCL:unique},
\cite{BCLSV1} and \cite{SJ209}
that use
the homomorphism densities $t(F,W)$ and the random graphs $G(n,W)$ defined in
\refApp{Ahomo} and \refApp{Arg}.

\begin{theorem}\label{Teq1}
Let\/ $W$ and $W'$ be two graphons (possibly defined on different \ps{s}).
  Then the following are equivalent: 
  \begin{romenumerate}
  \item \label{teq1equ}
$W\equ W'$.
  \item  \label{teq1dcut}
$\dcut(W,W')=0$.
  \item  \label{teq1dl}
$\dl(W,W')=0$.
  \item \label{teq1t}
$t(F,W)=t(F,W')$ for every simple graph $F$.
  \item \label{teq1tmulti}
$t(F,W)=t(F,W')$ for every loopless multigraph $F$.
  \item \label{teq1gnw}
The random graphs $G(n,W)$ and $G(n,W')$ have the same distribution for
every finite $n$.
  \item \label{teq1g00w}
The infinite random graphs $G(\infty,W)$ and $G(\infty,W')$ have the same
distribution. 
  \end{romenumerate}
\end{theorem}

\begin{proof}
\ref{teq1equ}$\iff$\ref{teq1dcut}: This is just our definition of $\equ$.

\ref{teq1dcut}$\implies$\ref{teq1dl}:
If $\dcut(W,W')=0$, let $W_0,\dots,W_n$ be a chain of graphons as in
\refT{Tchain} (or \refT{Teq2}).   
We have $\dl(W_i,W_i^\gf)=0$ for any pull-back of a graphon
$W_i$, and thus $\dl(W_{i-1},W_i)=0$ for every $i\ge1$. Hence $\dl(W,W')=0$
by the triangle inequality \refL{Ltriangle}.

\ref{teq1dl}$\implies$\ref{teq1dcut}: Trivial.

\ref{teq1dcut}$\implies$\ref{teq1tmulti}:
This is immediate from \eqref{t} for a pull-back, and the general case follows
again by  \refT{Tchain}.

\ref{teq1tmulti}$\implies$\ref{teq1t}: Trivial.

\ref{teq1t}$\iff$\ref{teq1gnw}: 
The distribution of $G(n,W)$ is determined by the family $\set{t(F,W):|F|\le
  n}$ of homomorphism densities for
all (simple) graphs $F$ with $|F|\le n$, and conversely, \cf{} \refR{Rgnw}.

\ref{teq1gnw}$\iff$\ref{teq1g00w}: 
The distribution of $G(\infty,W)$ is determined by the family of
distributions of the restrictions $G(\infty,W)|_{[n]}$ to the first $n$
vertices, for $n\ge1$, and conversely. However, $G(\infty,W)|_{[n]}=G(n,W)$.
See \cite{SJ209} for details.

\ref{teq1t}$\implies$\ref{teq1dcut}: 
See \cite{BCLSV1} or, for a different
proof, \cite{BCL:unique}. (This is highly
non-trivial.) 
Alternatively, \ref{teq1g00w}$\implies$\ref{teq1dcut} follows from
\cite[Theorem 7.28]{Kallenberg:symmetries}, 
see \cite[Proof of Theorem  7.1]{SJ209}. 
\end{proof}

\begin{remark}\label{Rt=cut}
  One of the central results in \cite{BCLSV1} is that,
for graphons $W_1,W_2,\allowbreak\dots$ and $W$, 
 $\dcut(W_n,W)\to0$
if and only if $t(F,W_n)\to t(F,W)$
for every (simple) graph $F$. 
(Taking $W_n=W_{G_n}$ for a sequence of graphs with $|G_n|\to\infty$, this 
says in particular that 
$G_n\to W \iff t(F,G_n)\to t(F,W)$ for every graph $F$, see
\refApp{Alimits}.) 
As pointed out in \cite{BRmetrics}, this equivalence is equivalent to the
corresponding equivalence \ref{teq1dcut}$\iff$\ref{teq1t} in \refT{Teq1}.

One way to see this 
is to
define a new semimetric on the class $\cWx$ of graphons by
\begin{equation*}
\dt(W,W')\=\sumni2^{-n}|t(F_n,W)-t(F_n,W')| ,   
\end{equation*}
where $F_1,F_2,\dots$ is some
(arbitrary but fixed) enumeration of all 
unlabelled (simple) graphs. By \refT{Teq1}, $\dt(W,W')=0\iff W\equ
W'\iff\dcut(W,W')=0$, so $\dt$ is, just as $\dcut$, a
metric on the quotient space $\bwx$. Moreover, the easy result 
\refL{Ltcont}
that each
$W\mapsto t(F_n,W)$ is continuous for $\dcut$ implies that
$\dt$ is continuous on $(\bwx,\dcut)$, so
the topology on $\bwx$ defined by $\dt$ is weaker than the topology defined
by $\dcut$. (Equivalently, the identity map $(\bwx,\dcut)\to(\bwx,\dt)$ is
continuous.) However, since $(\bwx,\dcut)$ is compact, this implies that
topologies are the same, \ie, that the metrics $\dcut$ and $\dt$ are
equivalent on $\bwx$, which is the result we want.
(This argument in \cite{BRmetrics} is essentially the same, but stated
somewhat differently. Another equivalent version is to consider the mapping 
$\bwx\to\oi^\infty$ given by $W\mapsto (t(F_n,W))_{n=1}^\infty$, see
\cite{SJ209}; this map is 
continuous and, by \refT{Teq1}, injective, so again by compactness it is a
homeomorphism onto some subset.) 
Note the importance of the compactness of $\bwx$ in these arguments. 
\end{remark}

The \emph{distribution} of a kernel $W$ defined on a \ps{} $(\sss,\mu)$
is the distribution of $W$ regarded
as a random variable defined on $\sssq$, \ie, 
the push-forward of $\mu^2$ by $W$, or equivalently
the \pmm{} on $\bbR$
that makes $W:\sssq\to\bbR$ \mpp{},
see
\refR{Rpush}. 

\begin{corollary}\label{Cdistr}
  If $W_1$ and $W_2$ are two equivalent graphons, defined on two \ps{s}
  $\sss_1$ and $\sss_2$, then $W_1$ and $W_2$ have the same distributions.
 In particular, $\int_{\sss_1^2} W_1^k=\int_{\sss_2^2} W_2^k$ for every $k\ge1$.
\end{corollary}

\begin{proof}
  The conclusion obviously holds if $W_1$ is \aex{} equal to a pull-back
  $W_2^\gf$ of  $W_2$, or conversely. The general case follows by
\refT{Tchain} (or \refT{Teq2}) and transitivity. 
Alternatively, we may use \refT{Teq1} and
observe that $\int W_\ell^k=t(M_k,W_\ell)$ if $M_k$ is the
multigraph consisting of $k$ parallel edges, see \refE{Ek}.
\end{proof}

  Note that, for any $k>1$, $W\to\int_\sssq W^k$ is \emph{not} continuous in the
  cut norm, see \refE{EW2}.

Finally we give, as promised above, the counter-example by
\citet{BCL:unique}, showing that  the 
condition that the spaces are Borel (or Lebesgue) is needed in \refT{TU}.

\begin{example}
  \label{Ebadequal}
Let $A\subseteq\oi$ be a non-measurable set such that the outer measure
$\gl^*(A)=1$ and the inner measure $\gl_*(A)=0$. (Equivalently, every
measurable set contained in $A$ or in its complement has measure 0.)
Let $\cL_A\=\set{B\cap A:B\in\cL}$, the \emph{trace} of the Lebesgue
$\gs$-field on $A$. Then the outer Lebesgue measure $\gl^*$ is a probability
measure on $(A,\cL_A)$, and the injection $\iota:A\to\oi$ is \mpp.
(See \eg{} \cite[Exercises 1.5.8 and 1.5.11]{Cohn}.)

Let $W(x,y)\=xy$. $W$ is a graphon on $\oi$, so its pull-back $W_1\=W^\iota$,
which equals the restriction of $W$ to $A\times A$, is a graphon on
$\sss_1\=(A,\cL_A,\gl^*)$, and $W_1\equ W$.
The complement $A\comp\=\oi\setminus A$ satisfies the same condition as $A$, 
so we
may also define $\sss_2\=(A\comp,\cL_{A\comp},\gl^*)$, and let $W_2\equ W$ be
the restriction of $W$ to $A\comp\times A\comp$.

Then $W_1\equ W\equ W_2$, so $W_1\equ W_2$.
However, suppose that $(\gf_1,\gf_2)$ is a coupling of $\sss_1$ and $\sss_2$,
defined on some space $\sss$,
such that $W_1^{\gf_1}= W_2^{\gf_2}$ \aex.
Then the marginal $\mi{W_1^{\gf_1}}$ 
equals the pull-back $(\mi{W_1})^{\gf_1}$ of the marginal
$\mi W_1:\sss_1\to\oi$, but the marginal of $W_1$ is 
$$\mi W_1(x)=\mi W(x)\=\intoi W(x,y)\dd y=x/2;$$ 
hence 
$\mi{W_1^{\gf_1}}(x)=\gf_1(x)/2$ for all $x\in\sss$.
Similarly, $\mi{W_2^{\gf_2}}(x)=\gf_2(x)/2$ for all $x\in\sss$.
Our assumption $W_1^{\gf_1}= W_2^{\gf_2}$ \aex{} implies that the marginals
are equal \aex, and thus $\gf_1(x)/2=\gf_2(x)/2$ \aex; consequently,
$\gf_1(x)=\gf_2(x)$ for \aex{} $x\in\sss$. This is a contradiction since for
every $x$, $\gf_1(x)\in A$ while $\gf_2(x)\in A\comp$.

Consequently, for every coupling
$(\gf_1,\gf_2)$ of $\sss_1$ and $\sss_2$ we have $W_1^{\gf_1}\neq W_2^{\gf_2}$
on a set of positive measure and thus $\cn{W_1^{\gf_1}- W_2^{\gf_2}}>0$ by
\refL{Lcut0}. 
Hence, the infima 
in \refT{Tcut}\ref{tcut1}--\ref{tcutxy2} are not attained, and 
none of \refT{TU}\ref{TUcoup}--\ref{TUtwin2} holds, although \ref{TUequ}
does.
\end{example}

\section{Pure graphons and a canonical version of a graphon}\label{Scanon}

We present here a way to select an essentially unique, canonical choice of
graphon among all equivalent graphons corresponding to a graph limit;
more precisely we construct a graphon that is determined uniquely up to
\aex{} rearrangements. 
This construction
is based on \citet{LSz:topology}, although formulated somewhat differently.
This will also lead to a new proof of \refT{Teq2}, a proof which we find
simpler than the original one.

For convenience, we consider only graphons, although the construction
extends to general kernels with very few modifications.

Let $W$ be a graphon on a \ps{} $(\sss,\cF,\mu)$.
For each $x\in\sss$, the section $W_x$ is defined by
\begin{equation}\label{wx}
  W_x(y)\=W(x,y),
\qquad y\in\sss.
\end{equation}
Thus $W_x$ is a measurable function $\sss\to\oi$, and in particular
$W_x\in\lsfmu$.

Let $\psiw:\sss\to\lsfmu$ be the map defined by
$\psiw(x)\=W_x$. 
By a standard monotone class argument
(using \eg{} the version of the monotone class theorem in 
\cite[Theorem  A.1]{SJII}), 
see also \cite[Lemma III.11.16]{Dunford-Schwartz},
$\psiw:\sss\to\lsfmu$ is measurable.

Let $\muw$ be 
the push-forward $\mu^{\psiw}$ of $\mu$ by $\psiw$, \ie,
the \pmm{} on $\lsfmu$
that makes $\psiw:(\sss,\mu)\to(\lsfmu,\muw)$ \mpp{},
see \refR{Rpush}; explicitly, 
\begin{equation}
  \label{muw}
\muw(A)=\mu(\psiw\qw(A)),
\qquad A\subseteq\lsfmu.
\end{equation}
Further, let $\sssw$ be the support of $\muw$, \ie,
\begin{equation}\label{sssw}
  \sssw
\=\set{f\in\lsfmu:\muw(U)>0 
\text{ for every neighbourhood $U$ of $f$}}.
\end{equation}
$\sssw$ is a subset of $\lsfmu$, and we equip it with the induced metric,
given by the norm in $\lsfmu$, and the Borel $\gs$-field generated by the
metric topology.

\begin{theorem}\label{Tcanon}
  \begin{thmenumerate}
  \item \label{tcanon1}
$\sssw$ is a complete separable metric space, $\muw$ is a probability
  measure on $\sssw$ and\/ $\psiw(x)\in\sssw$ for $\mu$-\aex{} $x\in\sss$.
We can thus regard $\psiw$ as a mapping $\sss\to\sssw$ (defined a.e.);
then $\psiw:(\sss,\mu)\to\sssmuw$ is \mpp.

\item  \label{tcanonsupp}
$\muw$ has full support on\/ $\sssw$, \ie, if\/ $U\subseteq\sssw$ is open
  and non-empty, then $\muw(U)>0$.

\item \label{tcanonrange}
The range of $\psiw$
is dense in\/ $\sss_W$.
More precisely, 
$\psiw(\sss)\cap\sssw=\set{W_x:x\in\sss}\cap\sssw$ is a dense subset of\/
$\sssw$.  

\item \label{tcanonkry3}
$\displaystyle
\sss_W\subseteq\set{f\in\lsfmu:0\le f\le1 \text{ a.e.}}.
$

\item \label{tcanonw}
There exists a graphon $\hW$ on $\sssmuw$ such that the pull-back
$\hW^{\psiw}=W$ a.e.; this graphon $\hW$ is unique up to \aex{} equality.
  \end{thmenumerate}
In particular, $W\cong\hW$.
\end{theorem}

\begin{proof}
  Recall first that $\lsfmu$ is a Banach space, and thus a complete metric
  space. In many cases, $\lsfmu$ is separable (for example if $\sss=\oi$ or
  another Borel space); however, there are cases when $\lsfmu$ is
  non-separable, see \refApp{Asep},
and in order to be completely
  general, we have to include some technical details on separability
  below; these can be ignored when $\sss$ is a Borel space (and at the first
  reading). 

Recall also that if $B$ is a Banach space, then $L^1(\sss,\cF,\mu;B)$ is the
Banach space of  functions 
$f:\sss\to B$ that are measurable and essentially separably valued,
\ie, there exists a separable subspace $B_1\subseteq B$
such that $f(x)\in B_1$ for \aex{} $x$, and further
$\ints\|f\|_B\dd\mu<\infty$, see \eg{} 
\cite[Chapter III, in particular  Section III.6]{Dunford-Schwartz} 
or the summary in \cite[Appendix C]{SJII}.
(Note that \cite{Dunford-Schwartz} uses a definition of measurability which
implicitly includes essential separability, see
\cite[Lemma III.6.9]{Dunford-Schwartz}.) 

Returning to our setting, we remarked above that 
$\psiw:\sss\to\lsfmu$ is measurable; furthermore it is bounded since
$\normll{\psi(x)}=\normll{W_x}\le1$,
and
a monotone class argument (again using \eg{}
\cite[Theorem A.1]{SJII}) shows that 
$\psiw$ is
separably valued; 
thus 
$\psiw\in L^1\bigpar{\sss,\cF,\mu;\lsfmu}$.
In fact, 
see \cite[III.11.16--17]{Dunford-Schwartz}, 
the mapping $W\mapsto\psiw$ extends to $L^1(\sss\times\sss,\mu\times\mu)$,
and
more generally to
$L^1(\sss_1\times\sss_2,\cF_1\times\cF_2,\mu_1\times\mu_2)$ for a product 
of any two probability spaces (or, more generally, $\gs$-finite measure
spaces),
and this yields an isometric isomorphism
\begin{equation}\label{l1l1}
L^1(\sss_1\times\sss_2,\cF_1\times\cF_2,\mu_1\times\mu_2)
\cong
L^1\bigpar{\sss_1,\cF_1,\mu_1;
L^1(\sss_2,\cF_2,\mu_2)}.  
\end{equation}

As just said, $\psiw$ is separably valued, \ie, there exists a separable
subspace $B_1\subseteq B\=\lsfmu$ such that $\psiw(x)\in B_1$ for all
$x\in\sss$.
We may replace $B_1$ by $\overline B_1$, and we may thus assume that $B_1$
is a closed subspace of $B$, and thus a Banach space. 
Then $\muw(B\setminus B_1)=\mu(\psiw\qw(B\setminus B_1))=\mu(\emptyset)=0$,
and it follows from \eqref{sssw} that
$\sssw=\supp(\muw)\subseteq B_1$ and, more precisely,
\begin{equation}\label{sssw1}
  \sssw
\=\set{f\in B_1:\muw(U)>0 
\text{ for every open $U\subseteq B_1$ with $f\in U$}}.
\end{equation}
Let $\cA$ be the family of all open subsets $U$ of $B_1$ such that $\muw(U)=0$.
Then \eqref{sssw1} shows that $\sssw=B_1\setminus\bigcup_{U\in\cA}U$.
The union $\bigcup_{U\in\cA}U$ is open, so this shows that $\sssw$ is a
closed subset of 
$B_1$, 
and thus a complete separable metric space as asserted.
Moreover, since $B_1$ is separable, this union equals the union of some
countable subfamily;
hence
$\muw\bigpar{\bigcup_{U\in\cA}U}=0$ and $\muw(\sssw)=1$, so $\muw$ is a
probability measure on $\sssw$.

By the definition of $\muw$,
$\muw(\sssw)=\mu\set{x:\psiw(x)\in\sssw}$, so this also shows that
$\psiw(x)\in\sssw$ for $\mu$-\aex{} $x$. Thus we can modify $\psiw$ on a
null set in $\sss$ so that $\psiw:\sss\to\sssw$, and then $\psiw$ is \mpp{}
by the definition of $\muw$.

This proves \ref{tcanon1}.
Next, if $U$ is an open subset of $\sssw$ with $\muw(U)=0$, then $U=V\cap
\sssw$ for some open $V\subset B_1$. Then $\muw(V)=\muw(U)=0$, and thus
$V\in\cA$, so $V\subseteq B_1\setminus\sssw$ and $U=V\cap\sssw=\emptyset$,
which proves \ref{tcanonsupp}.

If $U\subseteq\sssw$ is open and nonempty, then $\muw(U)>0$ by 
\ref{tcanonsupp} and thus 
$\psiw\qw(U)\neq\emptyset$ by \eqref{muw}; hence
$U\cap\psi(\sssw)\neq\emptyset$, which shows \ref{tcanonrange}.
 
The set $Q\=\set{f\in\lsfmu:0\le f\le1 \text{ a.e.}}$ is a closed subset of
$\lsfmu$. 
Since $W_x(y)=W(x,y)\in\oi$ for every $x$ and $y$, it follows that
$\psi(x)\in Q$ for every $x$, and \ref{tcanonrange} implies \ref{tcanonkry3}.

To show \ref{tcanonw}, let $\cfw\subseteq\cF$ be the $\gs$-field
on $\sss$ induced by $\psiw$, \ie,
\begin{equation}
  \cfw\=\bigset{\psiw\qw(A):A\subseteq\lsfmu \text{ is measurable}}.
\end{equation}

By definition, $\psiw$ is measurable $(\sss,\cfw,\mu)\to\lsfmu$, so we can
regard $\psiw$ as an element of, using \eqref{l1l1},
\begin{equation}
L^1\bigpar{\sss,\cfw,\mu;
L^1(\sss,\cF,\mu)}
\cong
L^1(\sss\times\sss,\cfw\times\cF,\mu\times\mu).
\end{equation}
 This shows the existence of
$W_1\in L^1(\sss\times\sss,\cfw\times\cF,\mu\times\mu)$ such that $W_1=W$
 a.e.
Consequently, the conditional expectation
\begin{equation*}
  \E(W\mid \cfw\times\cF)
=   \E(W_1\mid \cfw\times\cF)
= W_1 = W \text{\quad a.e.}
\end{equation*}
By symmetry, also 
$  \E(W\mid \cF\times\cfw)=W$ a.e., and thus
\begin{equation*}
  \E(W\mid \cfw\times\cfw)
= \E( \E(W\mid \cfw\times\cF)\mid\cF\times\cfw)
=  W \text{\quad a.e.}
\end{equation*}
Hence, $W=W_2$ a.e., where $W_2:\sssq\to\oi$ is 
$\cfw\times\cfw$-measurable, which implies that $W_2=\tW^{\psiw}$ for some
measurable $\tW:\sssw^2\to\oi$;  we can symmetrize $\tW$ to obtain the
desired graphon $\hW(x,y)\=(\tW(x,y)+\tW(y,x))/2$. 

If $\hW_1:\sssw\to\oi$ is another graphon such that 
$\hW_1^{\psiw}=W=\hW^{\psiw}$ $\mu\times\mu$-a.e., then 
$\hW_1=\hW$ $\muw\times\muw$-a.e., by the definitions of pull-back and $\muw$.

Finally, $W\cong \hW$ since $W$ is \aex{} equal to a pull-back of $\hW$.
\end{proof}

\begin{remark}\label{Rborel}
  Sine $\sss_W$ is a complete separable metric space, the probability space
  $(\sss_W,\muw)$ is a Borel space, see \refApp{SSBorel}.
\end{remark}

Following \cite{LSz:topology}, but using our notations, we make the
following definition:

\begin{definition}
  A graphon $W$ on $\sss$ is \emph{pure} if the mapping $\psiw$ is a
  bijection $\sss\to\sssw$.
\end{definition}

Note that $\psiw$ is injective $\iff$ $W$ is twinfree (see \refS{Sequiv}).
It follows easily that a graphon is pure if and only if it is twinfree and
the metric 
$r(x,y)\=\normll{W(x,\cdot)-W(y,\cdot)}$ on $\sss$ is complete,
and further $\mu$ has full support in the metric space $(\sss,r)$. 
(Then, automatically, $(\sss,r)$ is separable.)
See further \cite{LSz:topology}.

\begin{remark}\label{Rpurae}
Let $W,W'$ be graphons on the same \ps{} $\sss$ with $W'=W$ a.e.
Then 
$W_x(y)=W'_x(y)$ for \aex{} $y$, for \aex{} $x$;
in other words, $\psiw(x)=\psi_{W'}(x)$ for \aex{} $x$.
Consequently, $\muw=\mu_{W'}$, and thus also $\sssw=\sss_{W'}$.
We have $\widehat{W'}^{\psi_{W}}=\widehat{W'}^{\psi_{W'}}=W'=W$ \aex, and thus
$\widehat{W'}=\hW$ \aex{} by the uniqueness statement in \refT{Tcanon}.

Note that if $W$ is pure and $W'=W$ a.e., then $W'$ is not necessarily
pure; however,
$W'$ is pure if $\mu\set{y:W(x,y)\neq W'(x,y)}=0$ for every $x$
(and not just for almost every $x$).
\end{remark}

We let $\hW$ denote the graphon constructed in \refT{Tcanon}. Note that $\hW$
is defined only up to \aex{} equivalence, so we have some freedom in
choosing $\hW$. We will show (\refL{Lpure}) that
there is a  choice of $\hW$ that is a pure graphon.

\begin{lemma}\label{LP1}
Let\/ $W_1$ and\/ $W_2$ be two graphons defined on \ps{s} $(\sss_1,\mu_1)$ and
$(\sss_2,\mu_2)$, respectively, and
$\gf:\sss_1\to\sss_2$ is a \mpp{} mapping such that $W_1=W_2^{\gf}$ a.e.
Then
the pull-back map $\gf^*:f\mapsto f^\gf$ is an isometric \mpp{} bijection of\/
$\sss_{W_2}$ onto $\sss_{W_1}$, and\/ $\hW_1^{\gf^*}=\hW_2$ a.e.
\end{lemma}

\begin{proof}
  By \refR{Rpurae}, we may replace $W_1$ by $W_2^\gf$ and thus assume
  $W_1=W_2^\gf$ everywhere and not just a.e.

Since $\gf$ is \mpp, $\gf^*$ is an isometric injection $L^1(\sss_2,\mu_2)\to
L^1(\sss_1,\mu_1)$.

If $x\in\sss_1$, then the composition $\gf^*\circ \psi_{W_2}\circ\gf$ maps
$x\in \sss_1$ to
\begin{equation*}
  \gf^*\bigpar{\psi_{W_2}(\gf(x))}
=
  \gf^*\bigpar{W_{2,\gf(x)}}
=
\bigpar{W_{2,\gf(x)}}^\gf,
\end{equation*}
which is the mapping
\begin{equation*}
x'\mapsto
\bigpar{W_{2,\gf(x)}}^\gf(x')
=W_2\bigpar{\gf(x),\gf(x')}
=W_2^\gf(x,x')
=W_1(x,x')
=W_{1,x}(x'),
\end{equation*}
and thus
\begin{equation}\label{sw}
  \gf^*\bigpar{\psi_{W_2}(\gf(x))}
=W_{1,x}
=\psi_{W_1}(x).  
\end{equation}
In other words, $\gf^*\circ \psi_{W_2}\circ\gf =\psi_{W_1}$.
Since $\psi_{W_1}$ , $\psi_{W_2}$ and $\gf$ are \mpp, it follows that for
any $A\subseteq L^1(\sss_1,\mu_1)$,
\begin{equation}
  \label{cp}
  \begin{split}
\hskip-1em
\mu_{W_2}\bigpar{(\gf^*)\qw(A)}	
&=
\mu_{2}\bigpar{\psi_{W_2}\qw\bigpar{(\gf^*)\qw(A)}}	
=
\mu_{1}\bigpar{\gf\qw\bigpar{\psi_{W_2}\qw\bigpar{(\gf^*)\qw(A)}}}	
\hskip-2em
\\&
=
\mu_{1}\bigpar{\psi_{W_1}\qw(A)}
=\mu_{W_1}(A).	
  \end{split}
\end{equation}
Hence, $\gf^*:L^1(\sss_2,\mu_2)\to L^1(\sss_1,\mu_1)$ is \mpp.

In particular, \eqref{cp} shows that 
$\mu_{W_2}\bigpar{(\gf^*)\qw(\sss_{W_1})}=\mu_{W_1}(\sss_{W_1})=1$. Moreover, 
$(\gf^*)\qw(\sss_{W_1})$ is closed in $L^1(\sss_2,\mu_2)$ since 
$\sss_{W_1}$ is closed in $L^1(\sss_1,\mu_1)$ and $\gf^*$ is continuous, and
thus it follows, see \eqref{sssw}, that
\begin{equation*}
  \sss_{W_2}=\supp(\mu_{W_2})\subseteq(\gf^*)\qw( \sss_{W_1}).
\end{equation*}
In other words, $\gf^*:\sss_{W_2}\to\sss_{W_1}$.

Next, since $\gf^*$ is an isometry and $\sss_{W_2}$ is a complete metric
space (by \refT{Tcanon}\ref{tcanon1}),
$\gf^*(\sss_{W_2})$ is a complete subset of the metric space $\sss_{W_1}$,
and thus $\gf^*(\sss_{W_2})$ is closed.
By \eqref{cp}, 
\begin{equation*}
\mu_{W_1}\bigpar{\gf^*(\sss_{W_2})}
=  
\mu_{W_2}\bigpar{(\gf^*)\qw(\gf^*(\sss_{W_2}))}
= 
\mu_{W_2}\bigpar{\sss_{W_2}}=1.
\end{equation*}
Thus by \eqref{sssw} again  
(or by \refT{Tcanon}\ref{tcanonsupp}),
\begin{equation*}
  \sss_{W_1}=\supp(\mu_{W_1})\subseteq\gf^*( \sss_{W_2}).
\end{equation*}
Hence $\gf^*$ is a bijection $\sss_{W_2}\to\sss_{W_1}$.

Finally, by \eqref{sw},
\aex{} on $\sss_1\times\sss_1$,
\begin{equation*}
\bigpar{\bigpar{ \hW_1^{\gf^*}}^{\psi_{W_2}}}^{\gf}
=
\bigpar{ \hW_1}^{\gf^*\circ\psi_{W_2}\circ\gf}
=
\bigpar{ \hW_1}^{\psi_{W_1}}
=
W_1
=
\bigpar{ W_2}^{\gf},
\end{equation*}
and thus $\bigpar{ \hW_1^{\gf^*}}^{\psi_{W_2}}=W_2$ \aex{} on
$\sss_2\times\sss_2$. 
Consequently,
$ \hW_1^{\gf^*}=\hW_2$ \aex{} by the uniquness statement in
\refT{Tcanon}\ref{tcanonw}. 
\end{proof}

\begin{lemma}\label{Lpure}
  For any graphon $W$, $\hW$ in \refT{Tcanon}
can be chosen to be a pure graphon on $\sssw$.
\end{lemma}

\begin{proof}
Let $W$ be defined on $\sss$.
The construction in \refT{Tcanon} yields the graphon $\hW$ defined on
$\sssw\subseteq\lsfmu$. 
We repeat the construnction, starting with $\hW$ on $\sssw$, and obtain the
graphon $\widehat{\hW}$ on $\sss_{\hW}$, where $\sss_{\hW}\subseteq
L^1(\sssw,\muw)$. Since $\psiw:\sss\to\sssw$ is \mpp{} and $\hW^{\psiw}=W$
\aex{} by \refT{Tcanon},  it follows by \refL{LP1} that $\psiw^*$ is an
isometric bijection of $\sss_{\hW}$ onto $\sssw$;
thus $(\psiw^*)\qw$ is a bijection $\sss_W\to\sss_{\hW}$.

We will show that we can modify $\hW$ on a null set so that
$\psihw=(\psiw^*)\qw$. 

For \aex{} $x\in\sss$, we have $\psiw(x)\in \sssw\subseteq \lsfmu$
and then $\psi_{\hW}(\psiw(x))$ is by \eqref{wx} the function in
$L^1(\sssw,\muw)$ given by 
\begin{equation*}
  \psi_{\hW}(\psiw(x))(g)=\hW\bigpar{\psiw(x),g},
\qquad g\in\sssw.
\end{equation*}
Consequently, the pull-back map $\psiw^*$ in \refL{LP1} maps this to
the function on $\sss$ given by, for \aex{} $x$ and $y$,
\begin{equation*}
  \begin{split}
\psiw^*(\psi_{\hW}(\psiw(x)))(y)
&=\psi_{\hW}\bigpar{\psiw(x)}\bigpar{\psiw(y)}
=\hW\bigpar{\psiw(x),\psiw(y)}
\\&
=\hW^{\psiw}(x,y)
=W(x,y)
=\psiw(x)(y);
  \end{split}
\end{equation*}
thus
$\psiw^*(\psi_{\hW}(\psiw(x)))=\psiw(x)$ for \aex{} $x\in\sss$.

Let 
$$
A\=\set{f\in\sssw:\psiw^*(\psi_{\hW}(f))= f}.
$$ 
We have just shown that
$\mu\set{x:\psiw(x)\in A}=1$, 
so by \eqref{muw}
$\muw(A)=1$.
Thus, $\psiw^*(\psi_{\hW}):\sssw\to\sssw$ equals the identity map a.e.

Since $\psiw^*$ is a bijection,
$\psi_{\hW}=(\psiw^*)\qw$ on $A\subseteq\sss_W$.
The idea is to modify $\psi_{\hW}$ on the null set $\sss_W\setminus A$ such
that this equality holds everywhere. The space $\sss_{\hW}$ is included in a
separable subspace $B_1\subseteq L^1(\sss_W,\mu_W)$, and by \refL{Leval},
there exists a measurable evaluation map $\Phi:B_1\times\sss_W\to\bbR$ such
that $\Phi(F,g)=F(g)$ for every $F\in B_1$ and $\mu_W$-\aex{} $g\in\sss_W$.
Define
\begin{equation*}
  H(f,g)\=\Phi\bigpar{\psixqw(f),g},
\qquad f,g\in\sss_W;
\end{equation*}
then $H:\sss_W\times\sss_W\to\bbR$ is measurable and 
for every $f\in\sss_W$, 
\begin{equation}
  \label{kry1}
H(f,g)=\psixqw(f)(g)
,\qquad\text{for \aex{} } g\in\sss_W.
\end{equation}

For every $f\in\sssw$, $\psixqw(f)\in\sss_{\hW}$,  
so by \eqref{kry1} and \refT{Tcanon}\ref{tcanonkry3},
$0\le H(f,g)\le 1$ for \aex{} $g\in\sss_W$.
Let $\bh(f,g)\=\min\bigset{\max\set{H(f,g),0},1}\in\oi$.
Then, for every  $f\in\sss_W$, by \eqref{kry1},
\begin{equation}
  \label{kry4}
\bh(f,g)=
H(f,g)=\psixqw(f)(g)
,\qquad\text{for \aex{} } g\in\sss_W.
\end{equation}
Thus
$\bh$ has the desired sections.
We define a graphon $\hW_1$  on $\sss_W$ by
\begin{equation}\label{kry5}
  \hW_1(x,y)\=
  \begin{cases}
	\hW(f,g),&f,g\in A;\\
\bh(f,g),&f\notin A, g\in A;\\
\bh(g,f),&f\in A, g\notin A;\\
0,&f,g\notin A.
  \end{cases}
\end{equation}
Then $\hW_1=\hW$ \aex, because
$\mu(A)=1$, so we may replace $\hW$ by $\hW_1$ in
\refT{Tcanon}\ref{tcanonw}.
Moreover, if $f\in A$ then $\hW_1(f,g)=\hW(f,g)$ for \aex{} $g$, and thus
$\psi_{\hW_1}(f)=\psi_{\hW}(f)=\psixqw(f)$.

If $f\notin A$, then $\hW_1(f,g)=\bh(f,g)=\psixqw(f)(g)$ for \aex{} $g$
by \eqref{kry5} and \eqref{kry4},
and
thus $\psi_{\hW_1}(f)=\psixqw(f)$ in this case too.

Consequently, $\psi_{\hW_1}=\psixqw$ is a bijection
$\sss_W\to\sss_{\hW}=\sss_{\hW_1}$, where the final equality is by
\refR{Rpurae}.  
\end{proof}

\begin{theorem}\label{Tpure-rearr}
  Two graphons $W_1$ and $W_2$ are equivalent 
if and only if\/ $\hW_1$ is an \aex{} rearrangement of\/ $\hW_2$
by a \mpp{} bijection $\sss_{W_1}\to\sss_{W_2}$ that further can be taken to
be an isometry.
\end{theorem}

In other words, $W_1\cong W_2$ 
if and only if there is an
  isometric \mpp{} bijection $\gf:\sss_{W_1}\to\sss_{W_2}$ such that
$\hW_2^{\gf}=\hW_1$ a.e.

  \begin{proof}
Consider the class $\cwm$
of all graphons that are defined on a probability space that is also a
metric space. 
	Define $W_1\equiv W_2$ if $W_1,W_2\in\cwm$ and
$W_1$ is \aex{} equal to a rearrangement of
	$W_2$ by an isometric \mpp{} bijection; this is an equivalence relation
	on $\cwm$.
\refL{LP1} shows that if $W_1$ is a pullback of $W_2$, then
$\hW_1\equiv\hW_2$.

If $W_1\cong W_2$, then \refT{Tchain} yields a chain of pullbacks linking $W_1$
and $W_2$, and thus $\hW_1\equiv \hW_2$.

Conversely, if $\hW_1$ equals a rearrangement of $\hW_2$ a.e., then
$W_1\cong\hW_1\cong\hW_2\cong W_2$ by \refT{Tcanon}\ref{tcanonw}.
  \end{proof}

  \begin{corollary}
\label{Cglunten}
If\/ $W_1$ and $W_2$ are equivalent graphons, then $\sss_{W_1}$ and
$\sss_{W_2}$ are isometric metric spaces. 
\nopf
  \end{corollary}

\begin{theorem}\label{Tpure}	
Every graphon is equivalent to a pure graphon.
Two pure graphons $W_1$ and $W_2$ are equivalent
if and only if they are \aex{} rearrangements of each other.
 \end{theorem}

\begin{proof}
If $W$ is a graphon, then $W\cong \hW$ for a pure graphon $\hW$ by
\refT{Tcanon}\ref{tcanonw} and \refL{Lpure}.

If $W_1$ is pure, then $\psi_{W_1}$ is a bijection so $W_1$ is an
\aex{} rearrangement of $\hW_1$ by \refT{Tcanon}\ref{tcanonw}.
The same applies to $W_2$, and if further $W_1\cong W_2$, then
$\hW_1\cong\hW_2$ and \refT{Tpure-rearr} yields that $\hW_1$ is an \aex{}
rearrangement of $\hW_2$. Since being an \aex{} rearrangement is an
equivalence relation, this shows that $W_1$ is an \aex{} rearrangement of
$W_2$, and conversely.
\end{proof}

\begin{proof}[Proof of \refT{Teq2}]
Suppose that $W_1\cong W_2$.
By \refT{Tpure-rearr}, $\hW_1$ is an \aex{} rearrangement of $\hW_2$.
Further, $W_1$ is \aex{} equal to a pull-back
of $\hW_1$ by \refT{Tcanon}, so by composition, $W_1$ is \aex{} equal to a
pull-back of 
$\hW_2$, and so is $W_2$ by \refT{Tcanon} again.
This proves \refT{Teq2} (except the last sentence, which was shown in
\refS{Sequiv}) 
for graphons, which suffices as remarked earlier.
Note that $\hW_2$ is defined on $\sss_{W_2}$, which by \refR{Rborel} is a
Borel space.
\end{proof}

By \refC{Cglunten}, every graph limit $\gG$, \ie{} every element of the
quotient space 
$\bwx$,  defines a complete separable metric space
$\sssw$ by taking any graphon $W$ that represents $\gG$; this
metric space 
is uniquely defined up to isometry. Hence metric and topological properties
of $\sssw$ are invariants of graph limits.
See \citet{LSz:topology}
for some relations between such properties of $\sssw$
and combinatorial properties of the graph limit; 
it would be interesting to find further such results.

\begin{example}\label{Efinite}
A trivial example is that a graph limit is of finite type, \ie{} it can be
represented by a step graphon, if and only if $\sssw$ is a finite set, see
\refE{Estep}.   
Theorems \refand{Tpure-rearr}{Tcanon} imply that every graphon equivalent to
a step graphon is \aex{} equal to a step graphon. 
\end{example}

\refT{Tpure} shows that we can regard pure graphons as the canonical
choices among all graphons representing a given graph limit. By considering
only pure graphons, equivalence boils down to \aex{} equality and
rearrangements, and every graphon $W$ has a pure version constructed as
$\hW$. This is theoretically pleasing.
(Nevertheless, for many applications it is more convenient to use other
graphons, for example defined on $\oi$, regardless of whether they are pure
or not.)

\begin{remark}\label{RLp}
  In this section, we have used mappings into $\lsmu$ and have constructed
$\sssw$ as a subset of $\lsmu$. We could just as well use $L^2(\sss,\mu)$,
  or $L^p(\sss,\mu)$ for any $p\in[1,\infty)$.
In fact, 
by \refT{Tcanon}\ref{tcanonkry3}
$\sssw\subset L^p(\sss,\mu)$ 
and
the different $L^p$-metrics are equivalent on $\sssw$ 
by H\"older's inequality; it follows easily that the construction above
yields 
the same space $\sssw$ for any $L^p$ with $1\le p<\infty$, with a different
but equivalent metric and thus the same topology.
\end{remark}

\subsection{The weak topology on $\sssw$}\label{SSweak}

Since $\sssw\subset L^2(\sss,\mu)$, the inner product $\innprod{f,g}\=\ints
fg\dd\mu$ is defined and continuous on $\sssw\times\sssw$. We define
further,
again following \citet{LSz:topology} in principle but not in all details,
\begin{equation}\label{rww}
  \rww(f,g)\=\int_{\sssw}\lrabs{\innprod{f-g,h}}\dd\muw(h),
\qquad f,g\in\sssw.
\end{equation}
Since $\psiw:\sss\to\sssw$ is \mpp, and $\psiw(x)=W_x$,
this can also be written
\begin{equation}
  \label{rww1}
\rww(f,g)=\int_{\sss}\lrabs{\innprod{f-g,W_x}}\dd\mu(x).
\end{equation}
Since $\innprod{f,g}$ is continuous and bounded on $\sssw\times\sssw$, 
it follows by dominated
convergence that $\rww(f,g)$ is continuous on $\sssw\times\sssw$.
We will soon see (in \refT{Trww})
that it is a metric. We let $\rw$ denote the original
metric on $\sssw$, \ie{} the $L^1$-norm: 
\begin{equation}\label{rw}
 \rw(f,g)\=\normll{f-g}\=\ints|f(x)-g(x)|\dd\mu(x). 
\end{equation}
We have $0\le W_x\le 1$ and thus
$|\innprod{f-g,W_x}|\le\normll{f-g}$, so by \eqref{rww1} and \eqref{rw},
\begin{equation}\label{rww-rw}
  \rww(f,g)\le\rw(f,g).
\end{equation}

\begin{remark}
  If $W$ is pure, so $\psiw$ is a bijection, then $\rww$ induces a metric on
  $\sss$ which we also denoted by $\rww$; explicitly,
  \begin{equation*}
	\begin{split}
\rww(x,y)
&\=	  
\rww(W_x,W_y)
= \ints|\innprod{W_x-W_y,W_z}|\dd\mu(z)
\\&\phantom:
=\ints\lrabs{\ints\bigpar{W(x,u)-W(y,u)}W(z,u)\dd\mu(u)}\dd\mu(z)
\\&\phantom:
=\ints\bigabs{W\circ W(x,z)-W\circ W(y,z)}\dd\mu(z)
\\&\phantom:
=\norm{W\circ W(x,\cdot)-W\circ W(y,\cdot)}_{\lsmu}
	\end{split}
  \end{equation*}
where $W\circ W(x,y)\=\ints W(x,u)W(u,y)\dd\mu(u)$. Thus,
if $T_W$ is the integral operator with kernel $W$, then $W\circ W$ is the
kernel of the integral operator $T_W\circ T_W$, which explains the notation.
\end{remark}

Recall that the \emph{weak topology} 
$\gs=\gs_{L^\infty}$ on $\lsmu$ is the topology generated by
the linear functionals $f\mapsto\innprod{f,h}=\ints fh\dd\mu$ for $h\in
L^\infty(\sss,\mu)$. 
In general,
if $X$ and $Y$ are two subsets of $L^1(\sss)$ such that $\ints |fh|<\infty$
when $f\in X$ and $h\in Y$, let $(X,\gs_Y)$ denote $X$ with the weak topology
generated by the linear functionals $f\mapsto \innprod{f,h}$, $h\in Y$.
Since the elements of $\sssw$ are uniformly bounded functions by
\refT{Tcanon}\ref{tcanonkry3}, it is well-known,
see \refL{Lweak}(i),
that the weak topology on $\sssw$ also is
generated by $f\mapsto\innprod{f,h}$ for $h\in L^1(\sss,\mu)$
(this is the \emph{weak${}^*$ topology} on $L^\infty(\sss,\mu)$ restricted to
$\sssw$), or by  
the subset $h\in L^p(\sss,\mu)$ (this is the weak topology on
$L^{q}(\sss,\mu)$, where $1/p+1/q=1$, restricted to the subset $\sssw$, \cf{}
\refR{RLp}).
Thus, 
\begin{equation}\label{tops}
(\sssw,\gs)
=(\sssw,\gs_{L^\infty})
=(\sssw,\gs_{L^1})
=(\sssw,\gs_{L^2}).  
\end{equation}
We let $\ssswb$ be the closure of $\sssw$ in $\lsmu$ in the weak topology.
It follows by \refT{Tcanon}\ref{tcanonkry3} that 
$\ssswb\subseteq\set{f\in\lsmu:0\le f\le1 \text{ a.e.}}$, and thus, by
\refL{Lweak}(i) again,
\begin{equation}\label{topsb}
(\ssswb,\gs)
=(\ssswb,\gs_{L^\infty})
=(\ssswb,\gs_{L^1})
=(\ssswb,\gs_{L^2}).  
\end{equation}
Moreover, the weak closure $\ssswb$ is the same as the weak
closure of $\sssw$ in $L^p(\sss,\mu)$ for any $p<\infty$, and the weak${}^*$
closure in $L^\infty(\sss,\mu)$.

Recall also that two metrics $r_1$ and $r_2$ on  the same space are
\emph{equivalent} if they induce the same topology, \ie, if
$r_1(x_n,x)\to0\iff r_2(x_n,x)\to0$ for any point $x$ and sequence $(x_n)$ in
the space; the metrics are 
\emph{uniformly equivalent} if
$r_1(x_n,y_n)\to0\iff r_2(x_n,y_n)\to0$ for any sequences $(x_n)$ and $(y_n)$.

\citet{LSz:topology} showed essentially the following.
\begin{theorem}\label{Trww}
  \begin{thmenumerate}
  \item \label{trwwm}
$\rww$ is a metric on\/ $\sssw$ and it defines the weak topology $\gs$ on
	$\sssw$.  The same holds on the weak closure $\ssswb$.
 \item \label{trwwc}
The metric space $(\ssswb,\rww)$ is compact. Thus 
$(\ssswb,\rww)$ is the completion of 
$(\sssw,\rww)$. In particular, $\ssswb=\sssw$ if and only if 
$(\sssw,\rww)$ is complete.
  \item \label{trww<}
The inequality  $\rw\ge\rww$ holds, and thus the identity mapping
$(\sssw,\rw)\to(\sssw,\rww)$ is uniformly continuous.
  \item \label{trwwcc}
$\sssw$ is compact if and only if the metrics\/ $\rw$ and\/ $\rww$ are
equivalent on $\sssw$ and further $\sssw$ is weakly closed, $\ssswb=\sssw$.
  \item \label{trww=}
The metrics\/ $\rw$ and $\rww$ are uniformly equivalent on $\sssw$ if and only
if\/ 
$\sssw$ is compact for the norm topology given by $\rw$.
  \end{thmenumerate}
\end{theorem}

It seems more difficult to characterize when 
$\rw$ and $\rww$ are  equivalent on $\sssw$, see Examples 
\ref{Enoncompact1}--\ref{Enoncompact2} below. 

Before proving the theorem, we introduce more notation.
Using the fact that
$\sssw\subset L^2(\sss,\mu)$ (\cf{} \refR{RLp}), 
let $A$ be the closed linear span of $\sssw$ in $L^2(\sss,\mu)$,
and let $B$ be the unit ball of $A$; thus $\sssw\subseteq B$.
We extend the definition \eqref{rww} of $\rww$ to all $f,g\in A$.

\begin{lemma}\label{Lrww}
  \begin{thmenumerate}
  \item 
$\rww$ is a metric on $A$.
  \item 
The metric $\rww$ defines the weak topology\/ $\gs_{L^2}$ on
	$B$. In other words, $(B,\rww)=(B,\gs_{L^2})$ as topological spaces.
  \item The metric space $(B,\rww)$ is compact.
  \end{thmenumerate}
\end{lemma}

\begin{proof}
  \pfitem{i}
Symmetry and the triangle inequality are immediate from the definition
\eqref{rww}. Suppose that $\rww(f,g)=0$ for some $f,g\in A$. Since
$h\mapsto|\innprod{f-g,h}|$ is continuous on $\sssw$, and its integral
  $\rww(f,g)$ is 0, it follows from \refT{Tcanon}\ref{tcanonsupp} that
  $\innprod{f-g,h}=0$ for every $h\in\sssw$.
The set   $\set{h\in L^2:\innprod{f-g,h}=0}$ is a closed linear subspace of
$L^2(\sss,\mu)$, and thus it contains $A$; \ie{}
$\innprod{f-g,h}=0$ for every $h\in A$.
In particular,
\begin{equation*}
  \ints|f-g|^2\dd\mu=\innprod{f-g,f-g}=0.
\end{equation*}
Thus $f-g=0$ \aex, \ie{} $f=g$ in $A\subseteq L^2$. Hence $\rww$ is a metric.

\pfitem {ii}
If $h\in L^2(\sss)$ and $h_2\in A$ is the orthogonal projection of $h$,
then $\innprod{f,h}=\innprod{f,h_2}$ for every $f\in A$. 
Consequently, $\gs_{L^2}=\gs_{A}$ on $A$.

Let $D$ be a countable dense subset of $\sssw$; then $D$ is total in $A$,
and thus $A$ is a separable Hilbert space.
It is a standard fact that the unit ball $B$ of $A$ with the weak topology
$\gs_A=\gs_{L^2}$
then is a compact metric space. (It is compact  
by the Banach--Alaoglu theorem \cite[Theorem  V.4.2]{Dunford-Schwartz}, 
and metric by \cite[Theorem V.5.1]{Dunford-Schwartz}.  
Explicitly, $\gs_A=\gs_D$ on $B$, by the same argument as in the
proof of 
\refL{Lweak},
and if $D=\set{h_1,h_2,\dots}$, we can define a
metric on $(B,\gs_{L^2})=(B,\gs_A)$ by
$d(f,g)\=\sum_i 2^{-i}|\innprod{f-g,h_i}|$.)

We next show that the identity map $(B,\gs_{L^2})\to(B,\rww)$ is continuous. 
Since, as just shown, $(B,\gs_{L^2})$ is metrizable, it suffices to consider
sequential continuity.
Thus assume that $f_n,f\in B$ and $f_n\to f$ in $\gs_{L^2}$. Then
$\innprod{f_n-f,h}\to0$ as \ntoo{} for every $h\in\sssw\subset L^2$, 
and thus $\rww(f_n,f)\to0$ by \eqref{rww} and dominated convergence.

The identity map $(B,\gs_{L^2})\to(B,\rww)$ is thus a continuous bijection of
a compact space onto Hausdorff space, and it is thus a
homeomorphism.
Consequently,
$(B,\gs_{L^2})=(B,\rww)$.

\pfitem{iii} 
A consequence of (ii) and its proof, where we showed that
$(B,\gs_{L^2})$ is compact.
\end{proof}

\begin{proof}[Proof of \refT{Trww}]

\pfitemref{trwwm}
Since $\sssw\subseteq B\subset A$, it follows by \refL{Lrww} that
$\ssswb\subseteq B$.
Hence, using \refL{Lrww} again and
\eqref{topsb},  $\rww$ is a metric on $\ssswb$ and 
$(\ssswb,\gs)=(\ssswb,\gs_{L^2})=(\ssswb,\rww)$.

\pfitemref{trwwc}
An immediate consequence of \refL{Lrww}, together with standard facts on
compact and complete metric spaces (see \eg{} \cite[Section 4.3]{Engelking}).

\pfitemref{trww<}
This is just \eqref{rww-rw}.

\pfitemref{trwwcc}
If $(\sssw,\rw)$ is compact, then 
the identity
mapping $(\sssw,\rw)\to(\sssw,\rww)$, which is continuous by \ref{trww<},
is a homeomorphism. The metrics are thus equivalent. 
Furthermore, $(\sssw,\gs)=(\sssw,\rww)$ is compact and thus closed in the
weak topology on $L^1(\sss,\mu)$.

Conversely, if $\sssw=\ssswb$, then $(\sssw,\rww)$ is compact by \ref{trwwc},
and if further the metrics are equivalent, then $(\sssw,\rw)$ is compact too.

\pfitemref{trww=}
If $(\sssw,\rw)$ is compact, then 
the metrics $\rw$ and $\rww$ on $\sssw$ are equivalent
as seen in the proof of \ref{trwwcc}.
Moreover, as is easily
seen (\eg{} \cite[Theorem 4.3.32]{Engelking}), two equivalent metrics on a
compact metric space are uniformly equivalent.

Conversely, if $\rw$ and $\rww$ are uniformly equivalent, then
$(\sssw,\rww)$ is a complete metric space, 
since $(\sssw,\rw)$ is by \refT{Tcanon};
hence $\ssswb=\sssw$ by \ref{trwwc}, and thus $\sssw$ is compact by
\ref{trwwc} again.
\end{proof}

The following analogue of \refC{Cglunten} shows that also the metric space
$(\sssw,\rww)$ and its completion, the compact metric space
$(\ssswb,\rww)$, are invariants of graph limits.

\begin{theorem}\label{Trww=}
  If\/ $W_1$ and\/ $W_2$ are equivalent graphons, then
  $(\sss_{W_1},r_{W_1\circ W_1})$ and\/   $(\sss_{W_2},r_{W_2\circ W_2})$
are isometric metric spaces, and so are
the compact metric spaces
  $(\sssbx{W_1},r_{W_1\circ W_1})$ and   $(\sssbx{W_2},r_{W_2\circ W_2})$.
\end{theorem}

\begin{proof}
  By \refT{Tchain}, it suffices to prove this in the case when $W_1$ is a
  pull-back of $W_2$ as in \refL{LP1}. In this case, for any
  $f,g\in\sss_{W_2}$,
using \eqref{rww} and
the fact that $f\mapsto f^\gf$ is a \mpp{} bijection of $\sss_{W_2}$
onto $\sss_{W_1}$ by \refL{LP1},
  \begin{equation*}
	\begin{split}
r_{W_1\circ W_1}(f^\gf,g^\gf)	
&=
\int_{\sss_{W_1}}\innprod{f^\gf-g^\gf,h}\dd\mu_{W_1}(h)
\\&
=
\int_{\sss_{W_2}}\innprod{f^\gf-g^\gf,k^\gf}\dd\mu_{W_2}(k)
\\&
=
\int_{\sss_{W_2}}\innprod{f-g,k}\dd\mu_{W_2}(k)
\\&
=
r_{W_2\circ W_2}(f,g);
	\end{split}
  \end{equation*}
thus the bijection $f\mapsto f^\gf$ is an isometry also 
$(\sss_{W_2},r_{W_2\circ W_2})\to(\sss_{W_1},r_{W_1\circ W_1})$.
This extends to an isometric bijection
of  $(\sssbx{W_2},r_{W_2\circ W_2})$ onto $(\sssbx{W_1},r_{W_1\circ W_1})$ 
by \refT{Trww}\ref{trwwc}.
\end{proof}

We say that a graphon $W$ is \emph{compact} if $\sssw$ is a compact metric
space with the standard $L^1$ metric $\rw$,
and \emph{weakly compact} if $(\sssw,\gs)=(\sssw,\rww)$ is compact.
By \refC{Cglunten} and \refT{Trww=}, the same then
holds for every equivalent graphon, 
so we may say that a graph limit is [weakly] compact if some, and thus any,
representing graphon is [weakly] compact.

Not every graphon is compact. Moreover, this can happen both with
$(\sssw,\gs)$ compact and $(\sssw,\gs)$ non-compact, as shown by the following
examples (inspired by a similar example in \cite{LSz:topology}).
Note that exactly one of the two conditions in \refT{Trww}\ref{trwwcc} fails in
each of the two examples.

\begin{example}\label{Enoncompact1}
  Let $\sss\=\setoi^\infty=\bigset{x=(x_i)_0^\infty:x_i\in\setoi}$
(the Cantor cube, which is homeomorphic to the Cantor set)
with the product measure $\mu\=\nu^\infty$, where $\nu\set0=\nu\set1=1/2$.
We write $\sss=\sss_0\cup\sss_1$, where $\sss_j\=\set{x\in\sss:x_0=j}$.
Note that there is a \mpp{} map $\oi\to\sss$ given by the binary expansion,
so the examples below can be translated to examples on $\oi$ by taking
pull-backs. 

If $F$ is a function $\sss_0\times\sss_1\to[-1,1]$, we define a graphon $W$
on $\sss$ by
\begin{equation*}
  W(x,y)\=
  \begin{cases}
\frac12+\frac12F(x,y),&x\in\sss_0,\,y\in\sss_1;	\\
\frac12+\frac12F(y,x),&x\in\sss_1,\,y\in\sss_0;	\\
\frac12,&x,y\in\sss_0 \text{ or }x,y\in\sss_1.	
  \end{cases}
\end{equation*}

Define $F_x(y)=F(x,y)$ for $x\in\sss_0$, $y\in\sss_1$ and
$\FF_x(y)=F(y,x)$ for $x\in\sss_1$, $y\in\sss_0$; 
thus $F_x\in L^1(\sss_1)$ for $x\in\sss_0$ and 
$\FF_x\in L^1(\sss_0)$ for $x\in\sss_1$.
 
Regard $L^1(\sss_0)$ and $L^1(\sss_1)$ as subspaces of $L^1(\sss)$ in the
obvious way (extending functions by 0).
Define the maps $\Phi_0:\sss_0\to L^1(\sss_1)$ and 
$\Phi_1:\sss_1\to L^1(\sss_0)$ by $\Phi_0(x)=F_x$ and $\Phi_1(x)=\FF_x$;
then
\begin{equation}
  \label{e2}
\psiw(x)
=W_x
=\tfrac12+\tfrac12\Phi_j(x),\qquad \text{for } x\in\sss_j.
\end{equation}
Let $\mu_j$ be
the push-forward of $\mu$ by $\Phi_j$;
this is a measure on $L^1(\sss_{1-j})\subset L^1(\sss)$ with total mass $1/2$.
Let $X_j\subset L^1(\sss_{1-j})\subset L^1(\sss)$ be the support of $\mu_j$.
It follows from \eqref{e2} that the map $f\mapsto\frac12+\frac12f$ is \mpp{}
$(L^1(\sss),\mu_0+\mu_1)\to(L^1(\sss),\muw)$, and thus
\begin{equation}\label{sssxx}
  \sssw=\bigset{\tfrac12+\tfrac12f:f\in X_0\cup X_1}.
\end{equation}

Define the functions $h_i:\sss_1\to\set{-1,1}$ by
$h_i(x)=2x_i-1$, where $x\mapsto x_i$ is the $i$:th coordinate function.
Let $\ell(x)\=\inf\set{i:x_i=1}$ (defined \aex{} on $\sss$)
and take
\begin{equation}\label{fex}
  F(x,y)\=h_{\ell(x)}(y).
\end{equation}
(Thus $W(x,y)=y_{\ell(x)}$ for $x\in\sss_0$, $y\in\sss_1$.)

Then $\set{F_x:x\in\sss_0}=\set{h_i:i\ge1}$. The induced measure $\mu_0$ on
$L^1(\sss_1)$ is thus a discrete measure with atoms $h_i$ (each with
positive measure), so
\begin{equation*}
  X_0=\supp\mu_0=\overline{\set{h_i:i\ge 1}} = {\set{h_i:i\ge 1}},
\end{equation*}
since $\set{h_i}$ is closed in $L^1(\sss_1)$ because 
$\norm{h_i-h_j}_{L^1(\sss_1)}=1/2$
when $i\neq j$.

If $y,z\in\sss_1$, then
\begin{equation*}
  \normll{\FF_y-\FF_z}=\int_{\sss_0}|F(x,y)-F(x,z)|\dd\mu(x)
=\sum_{i=1}^\infty 2^{-i-1}|y_i-z_i|.
\end{equation*}
It is easily seen that this is a metric on $\sss_1$ which defines the
product topology. Hence 
$\Phi_1: y\mapsto\FF_y$ is a homeomorphism of $\sss_1$ onto
$\set{\FF_y:y\in\sss_1}\subset L^1(\sss_0)$, and consequently 
$\set{\FF_y:y\in\sss_1}$ is a compact subset of $L^1$. Since further
$\Phi_1:(\sss_1,\mu)\to (L^1(\sss),\mu_1)$ is \mpp{}, and $\mu$ has full support
on $\sss_1$, it follows that 
$X_1=\supp \mu_1 = \set{\FF_y:y\in\sss_1}\cong \sss_1\cong \sss$
(where $\cong$ denotes homeomorphisms.)
Note also that $X_0$ and $X_1$ are disjoint; in fact, they have distance 1
in $L^1$.

It follows that $\sssw\cong X_0\cup X_1\cong \bbN \cup \sss$, \ie, 
$\sssw$ is homeomeorphic to the disjoint union of the
Cantor cube (or Cantor set) and  a sequence 
of discrete points. Thus
$\sssw$ is not compact. 

With the weak topology $\gs$, we have $(X_1,\gs)=(X_1,\rw)$ because
$(X_1,\rw)$ is compact. Moreover, the sequence $(h_i)$ is orthonormal in
$L^1(\sss_1,2\mu)$ (for convenience normalizing the measure on $\sss_1$),
and thus $h_i\to0$ weakly in $L^2$ as $i\to\infty$. It follows by
\refL{Lrww}(ii) that $\rww(h_i,0)\to0$. For the corresponding elements
$g_i\=\frac12+\frac12h_i\in\sss$, see \eqref{sssxx}, we have
$\rww(g_i,\frac12)\to0$. It follows that $(\sssw,\gs)$ consists of a 
compact set homeomorpic to $\sss$, and a sequence $(g_i)$ converging to
$\frac12$. Since $\frac12\notin\sssw$, it follows that 
$(\sssw,\gs)=(\sssw,\rww)$ is not
compact; 
moreover, the identity map $(\sssw,\rw)\to(\sssw,\rww)$ is a
homeomorphism, so
$(\sssw,\gs)=(\sssw,\rw)\cong\bbN\cup\sss$.
Thus $\rw$ and $\rww$ are equivalent on $\sssw$ but not
uniformly equivalent. (Just as \set{1,2,\dots} and \set{1,1/2,1/3,\dots}, 
both with the usual metric on $\bbR$, are equivalent but not uniformly so.)

The weak closure $\ssswb=\sssw\cup\set{\frac12}$ is the one-point
compactification of $\sssw$.
\end{example}

\begin{example}\label{Enoncompact2}
  We modify the preceding example by taking, instead of \eqref{fex},
  \begin{equation}
	F(x,y)\=
	\begin{cases}
0, & \text{if } x_1=x_2=1,\\
h_{\ell(x)}(y) &\text{otherwise}.	  
	\end{cases}	
  \end{equation}
The only significant difference from the preceding example is that now $X_0$
also contains the function 0, and $\sssw$ thus the function $\frac12$; note
that $h_i\to0$ weakly and thus in $\rww$ but not in $\rw$. 
In the norm topology. $X_0=\set{g_i}\cup\set{\frac12}$ is still an infinite
discrete set, and thus 
$\sssw\cong X_0\cup X_1 \cong \bbN\cup\sss$ as in \refE{Enoncompact1}.
(We have added one isolated point to $\sssw$.)

In the weak topology, however, $X_0$ now consists of a convergent sequence
and its limit point, and thus $(X_0,\gs)$ is compact and homeomorphic to 
the one-point compactification $\bbNoo$ of $\bbN$ (or, equivalently, to
$\set{1/n:n\in\bbN}\cup\set0$ with the usual topology).
Thus
$(\sssw,\gs)\cong X_0\cup X_1 \cong \bbNoo\cup\sss$.  
(Compared to \refE{Enoncompact1}, we have added the point at infinity in the
one-point compactification.)
In particular,
$(\sssw,\rww)=(\sssw,\gs)$ is 
compact but 
$(\sssw,\rw)$ is not, and the two topologies are different so the
metrics are not equivalent.
The weak closure $\ssswb=\sssw$.
\end{example}

\section{Random-free graphons}\label{Srf}

  \citet{LSz:topology} have studied the class of graph limits represented by 
\oivalued{} graphons (and the corresponding graph properties); with a slight 
  variation of their terminology
we call such graphons and graph limits \emph{random-free}
(a reason for the name is given in \refR{Rrf}):

\begin{definition}
  A \emph{random-free} graphon is a graphon $W$ with values in $\setoi$ a.e.
\end{definition}

By \refC{Cdistr}, every graphon equivalent to a random-free graphon is
random-free. Note that every graphon $W_G$ defined by a graph as in
\refE{Ewg1} is random-free. 
(A reason for the name random-free is given in \refR{Rrf}.)

\begin{example}
  \label{Ethreshold}
It is shown by \citet{SJ238} that every graph limit that is a limit of a
sequence of threshold graphs can be represented by a graphon that is
random-free (and has a monotonicity property, studied further in \cite{SJqm}).
Hence every representing graphon
is random-free, \ie, if $G_n$ are threshold graphs and $W$ is a graphon such
that $G_n\to W$, then $W$ is random-free.
\end{example}

\begin{example}\label{Einterval}
  It is shown by \citet{SJinterval} that every graph limit that is a limit
  of a sequence of interval graphs can be represented by the graphon
  $W(x,y)\=\ett{x\cap y\neq\emptyset}$ on the space 
$\sss\=\set{[a,b]:0\le a\le b\le1}$ of all closed subintervals of $\oi$,
  equipped with some Borel \pmm{} $\mu$. (Note that $\sss$ and $W$ are
  fixed, but $\mu$ varies.) Hence every graphon representing an interval
  graph limit is random-free. (This  includes the threshold graph
  limits in \refE{Ethreshold} as a subset.  The explicit
  representations in \cite{SJ238} and \cite{SJinterval} are different, however.)
\end{example}

\begin{lemma}
  \label{Lrf}
Let\/ $W$ be a graphon.
Then the following are equivalent.
\begin{romenumerate}
\item $W$ is random-free.
\item $\intsq W(1-W)=0$.
\item $\intsq W^2=\intsq W$.
\end{romenumerate}
\end{lemma}
\begin{proof}
  This is trivial, noting that $W$ is random-free if and only if $W(1-W)=0$
  \aex, and that $W(1-W)\ge0$ for every graphon.
\end{proof}

 Recall that $W\mapsto\intsq W^2$ is not continuous
for $\dcut$, see \refE{EW2}; we therefore cannot conclude that the set
of random-free graphons is closed. 
In fact, it is not; on the contrary, this set is dense in the space of all
graphons. 

\begin{lemma}\label{Ldense}
  The set of random-free graphons is dense in the space of all graphons.
In other words, given any graphon $W$, on any \ps{} $\sss$, there exists a
sequence of random-free graphons $W_n$ such that $\dcut(W_n,W)\to0$.
\end{lemma}

\begin{proof}
By \refR{Rgn},
  there exists a sequence $(G_n)$ of graphs such that  $\dcut(W_{G_n},W)\to0$.
Each $W_{G_n}$ is random-free.
\end{proof}

In contrast, the set is closed in the stronger metric $\dl$.

\begin{lemma}\label{Ldlclosed}
  The set of random-free graphons is closed in the space of all graphons
  equipped with the metric $\dl$.
In other words, if\/ $W$ and\/ $W_n$ are graphons, on any \ps{s}, 
such that $\dl(W_n,W)\to0$, and every $W_n$ is random-free, 
then $W$ is random-free.
\end{lemma}
\begin{proof}
  Let $F(x)\=x(1-x)$. Then $F:\oi\to\oi$ and 
$|F'(x)|\le1$ so
$|F(x)-F(y)|\le|x-y|$ for
  $x,y\in\oi$. It follows easily that if 
$W_n$ and $W$ are graphons with
$\dl(W_n,W)\to0$, then $\dl(F(W_n),F(W))\le\dl(W_n,W)\to0$ and 
$|\int F(W_n)-\int F(W)|\to0$. Since  $W_n$ is random-free, 
$\int F(W_n)=0$ by \refL{Lrf} for each $n$,  and thus  $\int F(W)=0$.
By \refL{Lrf} again,
this shows that $W$ is random free.
\end{proof}

We continue to investigate the metric $\dl$ in connection with random-free
graphons. 

\begin{lemma}
  \label{LT2}
Let\/ $W_1$ and\/ $W_2$ be graphons on a probability space $\sss$,
and let $W_1'$ be a random-free $n$-step graphon on the same space. Then
\begin{equation}\label{emma}
  \norm{W_1-W_2}\qliss
\le n^2\cn{W_1-W_2} +2\norm{W_1-W_1'}\qliss.
\end{equation}
\end{lemma}
\begin{proof}
  Let $\set{A_i}_1^n$ be a partition of $\sss$ such that $W_1'$ is constant
0 or 1  on each $A_i\times A_j$.

If $W_1'=0$ on $A_i\times A_j$, then
\begin{multline*}
  	\iint_{A_i\times A_j}|W_1'-W_2|
=	
  \iint_{A_i\times A_j}W_2
\le
  \iint_{A_i\times A_j}W_1+\cn{W_1-W_2}
\\
=
  \iint_{A_i\times A_j}|W_1-W_1'|+\cn{W_1-W_2}.
\end{multline*}

If $W_1'=1$ on $A_i\times A_j$, then
{\multlinegap=0pt
\begin{multline*}
	\iint_{A_i\times A_j}|W_1'-W_2|
=	
  \iint_{A_i\times A_j}(1-W_2)
\le
  \iint_{A_i\times A_j}(1-W_1)+\cn{W_1-W_2}
\\
=
  \iint_{A_i\times A_j}|W_1-W_1'|+\cn{W_1-W_2}.
  \end{multline*}
 }
 
Thus, in both cases
$	\iint_{A_i\times A_j}|W_1'-W_2|
\le  \iint_{A_i\times A_j}|W_1-W_1'|+\cn{W_1-W_2}$,
and summing over all $i$ and $j$ yields
\begin{equation*}
\normll{W_1'-W_2}
\le  \normll{W_1-W_1'}+n^2\cn{W_1-W_2}.
\end{equation*}
The result follows by
$\normll{W_1-W_2}\le \normll{W_1-W_1'}+\normll{W_1'-W_2}$.
\end{proof}

\begin{remark}
  \label{RT2}
In particular, if $W_1$ is a random-free $n$-step graphon and $W_2$ an
arbitrary graphon on the same \ps{}, then
\begin{equation}\label{anna}
  \normll{W_1-W_2}\le n^2\cn{W_1-W_2}.
\end{equation}
The constant $n^2$ in \refL{LT2} and \eqref{anna}
is good enough for our purposes, but it is
not the best possible, and it may easily be improved.
In fact, an inspection of the proof shows that if we let
$a_{ij}\=\int_{A_i\times A_j}(W_1-W_2)$, then we have simply
estimated $|a_{ij}|\le\cn{W_1-W_2}$ and thus
$\sum_{i,j}|a_{ij}|\le n^2\cn{W_1-W_2}$. To obtain a better estimate, 
we use an inequality by \citet{Littlewood}, see also \cite{Blei}
and \cite[\S6.2]{Varopoulos}, which yields
\begin{equation}
  \sum_i \Bigpar{\sum_j|a_{ij}|^2}\qq
\le \sqrt3\,\sup_{\eps_i,\eps'_j=\pm1} \biggabs{\sum_i\sum_j\eps_i\eps'_ja_{ij}}
\le \sqrt3\, \cntwo{W_1-W_2}.
\end{equation}
Consequently, by the \CSineq,
\begin{equation}
  \sumin \sumjn|a_{ij}|
\le
  \sumin n\qq \Bigpar{\sumjn|a_{ij}|^2}\qq
\le \sqrt3\,\sqrt{n}\, \cntwo{W_1-W_2},
\end{equation}
which shows that $n^2$ in \eqref{emma} and \eqref{anna} can be replaced by
$\sqrt{3n}$. 
Furthermore, the constant $\sqrt3$, which is implicit in 
\cite{Littlewood}, 
has been improved to $\sqrt2$ by \citet{Szarek}. (Szarek actually proved that
$\sqrt2$ is the sharp constant in Khinchin's inequality, which implies
Littlewood's, see \cite{Blei}. See also \cite{Haa2} for related results.)
Consequently,  $n^2$ in \eqref{emma} and \eqref{anna} can be replaced by
$\sqrt{2n}$.
 
This is, within a numerical constant, the best constant in these inequalities, 
as shown by the following examples which all yield
a lower bound of order $n\qq$.
\end{remark}

\begin{example}
  \label{Ehadamard}
Let $W$ be a symmetric Hadamard matrix of order
$n$ (\ie, a matrix with $\pm1$ entries and all rows ortogonal);
such matrices exists at least if $n=2^k$ for some $k$. 
(Take tensor
powers of $\lrpar{\begin{smallmatrix}
 1&\phantom-1\\1&-1 
\end{smallmatrix}}
$.)
We have $W=W_+-W_-$
where $W_{\pm}$ are graphons on $[n]$. (We equip $[n]$
with the uniform probability distribution.)

Then $\normll{W}=1$, since $|W|=1$.
In order to estimate $\cntwo{W}$, let
$f$ and $g$ be two functions $[n]\to[-1,1]$, see the definition 
\eqref{cutnorm2}.
Write $W=(w_{ij})_{i,j=1}^n$ and change notation to $a_i=f(i)$, $b_j=g(j)$;
thus $|a_i|,|b_j|\le1$.

Since $W$ is a Hadamard matrix, the 
normalized matrix $n\qqw W$ is orthogonal, and is thus an isometry as an
operator in $\bbR^n$ (with the usual Euclidean norm); 
hence, $W$ has norm $\sqrt n$.
Consequently,
\begin{multline*}
\int_{[n]^2} W(x,y)f(x)g(y)\dd\mu(x)\dd\mu(y)
=n\qww \sum_{i,j=1}^n a_iw_{ij} b_j
\\
\le n\qww \sqrt{n} \lrpar{\sumin a_i^2}\qq\lrpar{\sumin b_j^2}\qq
\le n\qqw.  
\end{multline*}
Hence 
$\cntwo{W} \le n\qqw$ while $\normll{W}=1$, and thus 
the best constant in \eqref{anna} is at least $\sqrt n$ for $n$ such that a
symmetric Hadamard matrix exists,
and hence at least $\sqrt{n/2}$ for any $n$.
See further 
\cite{Littlewood} and
\cite[\S6.3]{Varopoulos}.
\end{example}

\begin{example}
  \label{Epaley}
Let $q$ be a prime power with $q\equiv 1\pmod 4$ and consider the Paley
graph $P_q$, see \cite[Section 13.2]{Bollobas}; the vertex set of $P_q$ is
the finite field $\bbF_q$ and there is an edge $xy$ if $x-y$ is a square in
$\bbF_q$. Let $W_1\=W_{P_q}$ and $W_2=1/2$; then $W_1$ is a random-free
$q$-step graphon, and $\normll{W_1-W_2}=1/2$, since $W_1-W_2=\pm1/2$
everywhere.
By \cite[Theorem 13.13]{Bollobas} (and its proof, or \refL{Lcut345G} below),
$\cnone{W_1-W_2} = O(q^{-1/2})$.
Hence, the constant in \eqref{anna} is at least $\Omega(q\qq)$, for $n=q$ of
this type. Since primes of the type $4k+1$ are dense in the natural numbers,
it follows again that the constant is  $\Omega(n\qq)$ for all
$n$. 
\end{example}

\begin{example}
We can use a random graph $G=G(n,1/2)$ and let $W_1\=W_G$ and,
again, $W_2\=1/2$. Thus $\normll{W_1-W_2}=1/2$.
(Note that the Payley graph in \refE{Epaley} is an example of a quasirandom
graph, so the two examples are related.)

We use for convenience the version $\cnx4\fyll$ of the cutnorm in
\refApp{Acut}.
If $S,T\subset[n]$ are disjoint, then $n^2\int_{S\times T} W_G$ is the
number of edges between $S$ and $T$, and has thus a binomial distribution
$\Bi(st,1/2)$ where $s\=|S|$ and $t\=|T|$. Hence, a Chernoff bound
\cite[Remark 2.5]{JLR} shows that, for any $c>0$, 
\begin{multline*}
  \P\Bigpar{\Bigabs{\int_{S\times T} (W_1-W_2)} > cn\qqw}
=
  \P\Bigpar{\Bigabs{n^2\int_{S\times T} (W_G-\E W_G)} > cn\qqc}
\\
\le 2\exp\Bigpar{-\frac{2(cn\qqc)^2}{st}}
\le 2\exp\Bigpar{-\frac{2c^2n^3}{n^2/4}}
=  2\exp\Bigpar{-8c^2n},	  
\end{multline*}
since $st\le s(n-s)\le n^2/4$. 
There are $3^n$ pairs $S,T$ of disjoint subsets, and thus
\begin{equation*}
  \P\bigpar{\cnx4{W_1-W_2} > cn\qqw}
\le  2\cdot3^n\exp\bigpar{-8c^2n}
=  2\exp\bigpar{(\log 3-8c^2)n}.
\end{equation*}
Consequently, choosing for simplicity $c=1$, so $8c^2>\log3$, with high
probability $\cnx4{W_1-W_2} \le n\qq$, and thus by \refL{Lcnx} 
\begin{equation*}
\cnx1{W_1-W_2} \le 4n\qqw = 8n\qqw \normll{W_1-W_2},
\end{equation*}
showing that the best constant in \eqref{anna} is at least $\frac18n\qq$ (for
$\cnx1\fyll$). 
\end{example}

\begin{lemma}
  \label{LT3}
Let\/ $W$ and\/ $W_1,W_2,\dots$ be graphons on a probability space $\sss$, and
assume that $W$ is random-free. Then $\cn{W_n-W}\to0$ as \ntoo{} if and only
if
$\normlss{W_n-W}\to0$.
\end{lemma}

\begin{proof}
  Assume $\cn{W_n-W}\to0$.
$W$ is the indicator $\etta_A$ of a measurable set $A\subseteq\sssq$. Any
  such set can be approximated in measure by a finite disjoint union of
  rectangle sets $\bigcup_i A_i\times B_i$, and we may assume that this set
  is symmetric since $A$ is; in other words, given any
  $\eps>0$, there exists a \oivalued{} step graphon $W'$ such that
  $\normll{W-W'}<\eps$. 
Let the corresponding partition have $N=N(\eps)$ parts.
\refL{LT2} then yields
\begin{equation*}
  \normll{W-W_n}\le N^2\cn{W-W_n}+2\eps\to2\eps
\end{equation*}
as \ntoo. Hence, $\limsup_\ntoo\normll{W-W_n}=0$.

The converse is obvious.
\end{proof}

\begin{lemma}
  \label{LT3q}
Let\/ $W$ and\/ $W_1,W_2,\dots$ be graphons defined on some \ps{s}, and
assume that $W$ is random-free. Then $\dcut(W_n,W)\to0$ as \ntoo{} if and only
if $\dl(W_n,W)\to0$.
\end{lemma}

\begin{proof}
Assume that $\dcut(W,W_n)\to0$.
  By replacing the graphons by equivalent ones, we may by \refT{Trep} assume
  that all graphons are defined on $\oi$. By \refT{Tcut}, we may then find
  \mpb{s} $\gf_n:\oi\to\oi$ such that
  $\cn{W-W_n^{\gf_n}}<\dcut(W,W_n)+1/n\to0$.
Hence, $\normlss{W-W_n^{\gf_n}}\to0$ by \refL{LT3}, and thus 
 $\dl(W,W_n)\to0$.

The converse is obvious.
\end{proof}

\begin{theorem}
  Let\/ $W$ be a graphon. Then $W$ is random-free if and only if 
$\dl(W_{G_n},W)\to0$ for some sequence of graphs $G_n$.
\end{theorem}

\begin{proof}
There exists a sequence of graphs $G_n$ with $\dcut(W_{G_n},W)\to0$ by
\refR{Rgn}. 
  If $W$ is random-free, then $\dl(W_{G_n},W)\to0$ by \refL{LT3q}.

The converse follows by \refL{Ldlclosed}, since each $W_{G_n}$ is random-free.
\end{proof}

\begin{theorem}\label{T3F}
  Let\/ $W$ be a graphon. Then the following are equivalent.
  \begin{romenumerate}
  \item \label{t3f1}
$W$ is random-free.
  \item \label{t3f2}
$\int W_n^2\to\int W^2$ whenever $(W_n)$ is a sequence of graphons such that
	$\dcut(W_n,W)\to0$. 
  \item \label{t3fm}
$t(F,W_n)\to t(F,W)$ for every multigraph $F$
whenever $(W_n)$ is a sequence of graphons such that
	$\dcut(W_n,W)\to0$. 
  \end{romenumerate}
\end{theorem}

\begin{proof}
  \ref{t3f1}$\implies$\ref{t3fm}:
If $W$ is random-free and $\dcut(W_n,W)\to0$, 
then Lemma \ref{LT3q} yields
$\dl(W_n,W)\to0$, 
and thus $t(F,W_n)\to t(F,W)$ for every multigraph $F$ by
\refL{Ltcont1}. 

  \ref{t3fm}$\implies$\ref{t3f2}: Immediate by taking $F$ to be a double
edge, see \refE{Ek}. 

  \ref{t3f2}$\implies$\ref{t3f1}:
Take a sequence of graphs $G_n$ such that $G_n\to W$, see \refR{Rgn}; thus
$\dcut(W_{G_n},W)\to0$. Hence $\int W_{G_n}\to \int W$. Further, every
$W_{G_n}$ is \setoi-valued, so $W_{G_n}^2=W_{G_n}$; hence
\begin{equation*}
  \int W_{G_n}^2=\int W_{G_n} \to \int W.
\end{equation*}
If \ref{t3f2} holds, then 
also
$\int W_{G_n}^2 \to \int W^2$. Hence $\int
W^2=\int W$, 
so $W$ is random-free by \refL{Lrf}.
\end{proof}

Finally, we mention two characterizations of random-free graphons in terms of
the finite or infinite random graph $G(n,W)$ defined in \refApp{Arg}.
First the finite case and entropy.

\begin{theorem}\label{Trgent}
  Let $W$ be a graphon. Then $W$ is random-free if and only if the entropy
$\ent(G(n,W))=o(n^2)$ as \ntoo.
\end{theorem}
\begin{proof}
  This is an immediate consequence of \refT{Tent}, since $h\ge0$ and
thus the right-hand side of \eqref{tent}  vanishes if and only $h(W(x,y))=0$
\aex,  which is equivalent to $W(x,y)\in\setoi$ \aex.
\end{proof}

\begin{problem}
We may, as in \cite[(15.30)]{Aldous85}, ask for the exact growth rate of
$\ent(G(n,W))$ 
for a random-free graphon $W$. 
It is easily seen that if $W$ is a step graphon, then $\ent(G(n,W))=O(n)$; 
we
conjecture that the converse holds too. As another example, for the ``half
graphon'' $W(x,y)=\ett{x+y>1}$ on $\oi$, it can be shown (\eg{} using
\cite[Corollary 6.6]{SJ238}) that $\ent(\gnw)=n\log n+O(n)$.
\end{problem}

We represent the 
infinite random graph $G(\infty,W)$  by the family of indicator variables
$J_{ij}\=\ett{ij\text{ is an edge}}$, $1\le i<j\le\infty$. We define the
\emph{shell} $\gs$-field (or \emph{big tail} $\gs$-field) to be the
intersection 
\begin{equation}\label{cS}
\cS\=\bigcap_{n=1}^\infty\gs\set{J_{ij}: i<j,\, j\ge n}  
\end{equation}
of the $\gs$-fields
generated by all $J_{ij}$ where at least one index is ``big''.
Recall that a random variable is \emph{\as{} $\cS$-measurable} 
(or \emph{essentially $\cS$-measurable})
if it is \as{} equal to an $\cS$-measurable variable; equivalently, it is
measurable for the completion $\widehat\cS$ of $\cS$.

\begin{theorem} The following are equivalent for a graphon $W$:
  \begin{romenumerate}
  \item $W$ is random-free.
  \item 
The infinite random graph  $G(\infty,W)$ 
is \as{} $\cS$-measurable.
  \item 
The indicator $J_{12}\=\ett{12\text{ is an edge in $G(\infty,W)$}}$
is \as{} $\cS$-mea\-sur\-able.
  \end{romenumerate}
\end{theorem}
\begin{proof}
  This is the symmetric version of \cite[Proposition 3.6]{Aldous}, see also
\cite[(14.15) and p.~133]{Aldous85} and \cite[(4.9)]{DF}. Since the details
for the symmetric case are not given in these references, we give some of
them for completeness.

First, note that we can write the definition of $G(\infty,W)$ in
\refApp{Arg} as 
\begin{equation}
  \label{xiij}
J_{ij}=\ett{\xiij\le W(X_i,X_j)}, 
\end{equation}
where $\xiij$, for $1\le i<j$, and $X_i$, for $i\ge1$, all are independent,
and $X_i$ has  distribution $\mu$ on $\sss$ while $\xiij$ is uniform on $\oi$. 

(i)$\implies$(iii):
If $W$ is random-free, then \eqref{xiij} simplifies to $J_{ij}=W(X_i,X_j)$.
Consider the array 
$(J_{2i-1,2j})_{i,j=1}^{\infty}=(W(X_{2i-1},X_{2j})_{i,j=1}^{\infty}$, where
the first index is odd and the second even; this is a separately
exchangeable array, and by 
\cite[Proposition 3.6]{Aldous} (or as a simple consequence of
\cite[Proposition 7.31]{Kallenberg:symmetries}), 
it is \as{} $\cS'$-measurable for the shell
$\gs$-field of this array. Since $\cS'\subseteq \cS$, (iii) follows.

(iii)$\iff$(ii): $\cS$ is invariant under finite permutations, so the
exchangeability implies that every $J_{ij}$ is $\cS$-measurable if $J_{12}$
is.
The converse is trivial.

(iii)$\implies$(i):
It follows from \eqref{cS} and \eqref{xiij} that $\xi_{12}$ is independent
of $\cS$. If (iii) holds, then $J_{12}$ is thus independent of $\xi_{12}$,
which by \eqref{xiij} implies that 
$J_{12}=\E(J_{12}\mid X_1,X_2)=W(X_1,X_2)$ \as, so $W$ is \oivalued{} \aex.
\end{proof}

\appendix
\section{Special probability spaces}\label{Aspaces}

\subsection{Atoms}
An \emph{atom} in a probability space $(\sss,\mu)$
is a subset $A$ with $\mu(A)>0$ such
that every subset $B\subseteq A$ satisfies $\mu(B)=0$ or $\mu(B)=\mu(A)$.

We say that $\sss$ is \emph{atomless} if there are no atoms. 

\begin{lemma}
  \label{Lfinns}
If $(\sss,\mu)$ is an atomless \ps, then there exists a family
$(A_r)_{r\in\oi}$ of measurable sets such that 
$\mu(A_r)=r$ for every $r\in\oi$, and further 
$A_r\subseteq A_s$ if $r<s$
(\ie, the family is increasing).
\end{lemma}

\begin{proof}
  Consider families $(A_r)_{r\in E}$ with these properties, defined on some
arbitrary subset $E$ of $\oi$. By Zorn's lemma, there exists a maximal
family; we claim that then $E=\oi$. In fact, $0,1\in E$, since we otherwise
could enlarge the family by defining $A_0=\emptyset$ or $A_1=\sss$.
Further, $E$ is closed, since otherwise there would exists $r\notin E$ and a
sequence $r_n\in E$ such that either $r_n\upto r$ or $r_n\downto r$; in the
first case we can define $A_r\=\bigcup_n A_{r_n}$, and in the second case
$A_r\=\bigcap_n A_{r_n}$. Finally, if $E\neq\oi$, the complement
$\oi\setminus E$ thus is open, and thus a disjoint union of open intervals.
Let $(a,b)$ be one of these intervals. Then $a,b\in E$, and $A_b\setminus A_a$ 
is a
set of measure $b-a>0$. Since $\mu$ is atomless, there exists a subset
$C\subseteq A_b\setminus A_a$ with $0<\mu(C)<b-a$, but in this case, the family
could be extended by $A_{a+\mu(C)}\=A\cup C$, so we again contradict the
maximality of the family. Hence $E=\oi$, which completes the proof.
\end{proof}

We also give a reformulation in terms of a map to $\oi$.

\begin{lemma}
  \label{Lfinns2}
If $(\sss,\mu)$ is an atomless \ps, then there exists a \mpp{} map
$\gf:\sss\to\oi$.
\end{lemma}

\begin{proof}
  Let $(A_r)_r$ be as in \refL{Lfinns},  and define
  $\gf(x)\=\inf\set{r\in\oi:x\in A_r}$ (assuming as we may that  $A_1=\sss$).
\end{proof}

\begin{lemma}\label{Latomless}
  If $\gf:\sss_1\to\sss_2$ is \mpp{} and $\sss_2$ is atomless, then $\sss_1$
  is atomless too.
\end{lemma}

\begin{proof}
  Let $(A_r)_{r\in\oi}$ be a family of subsets of $\sss_2$ with the
  properties in \refL{Lfinns}.
Then $B_r\=\gf\qw(A_r)$ defines a family of subsets of $\sss_1$ with the
same properties. Suppose that $A\subseteq\sss_1$ is an atom.
Then, for each $r$, $\mu(A\cap B_r)=0$ or $\mu(A)$. Let
$r_0\=\sup\set{r:\mu(A\cap B_r)=0}$, and take any $r_-<r_0$ and $r_+>r_0$.
(If $r_0=0$, take $r_-=0$, and if $r_0=1$, take $r_+=1$.)
Then $\mu(A\cap B_{r_-})=0$ and $\mu(A\cap B_{r_+})=\mu(A)$, so
\begin{equation*}
  \begin{split}
\mu(A)&=\mu(A\cap B_{r_+})-\mu(A\cap B_{r_-})
=\mu(A\cap (B_{r_+}\setminus B_{r_-}))
\\&
\le
\mu(B_{r_+}\setminus B_{r_-})
=r_+-r_-.	
  \end{split}
\end{equation*}
This is a contradiction, since $\mu(A)>0$ while $r_+-r_-$ can be arbitrarily
small. 
\end{proof}

In the opposite direction, there are typically many
\mpp{} maps from an atomless space $\sss_1$ into a space with atoms.
Simple examples are the
trivial map onto a one-point space, and the indicator function of a subset
$B\subseteq\sss_1$ seen as a map $(\sss_1,\mu)\to(\setoi,\nu)$, which is
\mpp{} if 
$\nu\set1=\mu(B)$.

\subsection{Borel spaces}\label{SSBorel}

To define Borel spaces, it is simplest to begin with measurable spaces,
without any particular measures.

We say that two measurable spaces $(\sss,\cF)$ and $(\sss',\cF')$ are
\emph{isomorphic} if there is a bimeasurable 
bijection $\gf:\sss\to\sss'$, \ie, a bijection such that both
$\gf$ and $\gf\qw$ are measurable. (Similarly, two probability spaces
$(\sss,\cF,\mu)$ and $(\sss',\cF',\mu')$ are isomorphic if there exists a
bimeasurable bijection that further is \mpp.)

A measurable space is \emph{Borel} 
(also called \emph{standard} \cite{Cohn} or \emph{Lusin} \cite{DM})
if it is isomorphic to a Borel
subset of a Polish space (\ie, a complete metric space) with its Borel
$\gs$-field. A probability space $(\sss,\cF,\mu)$ is \emph{Borel} if
$(\sss,\cF)$ is a Borel measurable space; equivalently, if it is isomorphic
to a Borel subset of a Polish space equipped with a Borel measure.

In fact, we do not need arbitrary Polish spaces here; the following theorem
shows that it suffices to consider subsets of $\oi$.
We tacitly assume that $\oi$ and other Polish spaces are equipped with
their Borel $\gs$-fields.
\begin{theorem}\label{Tborel}
  The following are equivalent for a measurable space $(\sss,\cF)$, and thus
each property  characterizes Borel measurable spaces.
  \begin{romenumerate}
  \item 
$(\sss,\cF)$ is isomorphic to a Borel subset of a Polish space.
  \item 
$(\sss,\cF)$ is isomorphic to a Polish space.
  \item 
$(\sss,\cF)$ is isomorphic to a Borel subset of $\oi$.
  \item 
$(\sss,\cF)$ is either countable (with all subsets	measurable), 
or isomorphic to $\oi$.
  \end{romenumerate}
\end{theorem}
For a proof, see \eg{} \cite[Theorem 8.3.6]{Cohn} or 
\cite[Theorem  I.2.12]{Parthasarathy}.
An essentially equivalent statement is that any two Borel measurable spaces
with the same cardinality are isomorphic.

Hence, a Borel probability space is either countable or isomorphic to $\oi$
equipped with some Borel probability measure. Consequently we can, when
dealing with Borel spaces, restrict
ourselves to $\oi$ without much loss of generality (the countable case is
typically simple), but for applications it is convenient to allow 
general Borel spaces.

\begin{remark}
  \label{Rcantor}
Another simple Borel space is the \emph{Cantor cube} $\cC\=\setoi^\infty$
(which up to homeomorphism is the same as the usual \emph{Cantor set});
this is a compact metric space, and thus a Polish space. 
Since $\cC$ is uncountable, it is by 
\refT{Tborel}
isomorphic to $\oi$ as measurable spaces;
consequently we may replace $\oi$ by $\cC$ in \refT{Tborel}.
\end{remark}

One important property of Borel spaces is the following theorem by
Kuratowski, showing that a 
measurable bijection is bimeasurable, and thus an isomorphism.
\begin{theorem}\label{Tinjection}
Let\/ $\sss$ and $\sss'$ be Borel measurable spaces.
If $f:\sss\to\sss'$ is a bijection that is measurable, then
$f\qw:\sss'\to\sss$ is measurable, and thus $f$ is an isomorphism.

More generally,
if
$f:\sss\to\sss'$ is a measurable injection, then the image $f(\sss)$ is a
measurable subset of\/ $\sss'$ and $f$ is an isomorphism of\/ $\sss$  onto
$f(\sss)$. 
\end{theorem}
For a proof, see \eg{} \cite[Proposition 8.3.5 and Theorem 8.3.7]{Cohn};
see also further results in \cite[Sections 8.3 and 8.6]{Cohn}.

Let us now add measures to the spaces.
There is a version of \refT{Tborel} for probability spaces.
For simplicity we begin with the atomless case.
Recall that $\gl$ denotes the Lebegue measure.
\begin{theorem}
  \label{Tborelp}
If $(\sss,\mu)$ is an atomless Borel \ps, then there exists a \mpb{} 
of $(\sss,\mu)$ onto $(\oi,\gl)$.
\end{theorem}

In other words, all atomless Borel \ps{} are isomorphic, as measure spaces.

\begin{proof}
  Since $\sss$ is atomless, every point has measure 0 and thus every
  countable subset has measure 0; in   particular, $\sss$ cannot be
  countable.
By \refT{Tborel}(iv), there exists a bimeasurable bijection $\gf_1$ of
$\sss$ onto $\oi$. 
This maps the measure $\mu$ onto some Borel measure $\nu$ on $\oi$.

Since $\nu$ has no atoms, $x\mapsto\nu([0,x])$ is a continuous
non-decreasing map of $\oi$ onto itself.
We let $\psi:\oi\to\oi$ be its right-continuous inverse defined by
\begin{equation}\label{psi}
  \psi(t)\=\sup\bigset{x\in\oi:\nu([0,x])\le t}.
\end{equation}
Then $\nu([0,\psi(t)])=t$ for every $t\in\oi$, which implies that $\psi$ is
strictly increasing. Hence, $\psi$ is injective and measurable, and by
\refT{Tinjection}, $\psi$ is a bimeasurable bijection of $\oi$ onto some
Borel subset $B\=\psi(\oi)$.

It follows from \eqref{psi} that, for all $s,t\in\oi$, 
$\psi(t)\ge s\iff\nu([0,s])\le t$, and thus 
$\psi\qw([0,s))=[0,\nu([0,s]))$.
Hence, 
\begin{equation*}
  \gl\bigpar{\psi\qw([0,s))}=\nu\bigpar{[0,s]}=\nu\bigpar{[0,s)},
\qquad s\in\oi,
\end{equation*}
which implies that $\gl^\psi=\nu$ (see \refR{Rpush} for the notation), \ie,
that $\psi:(\oi,\gl)\to (\oi,\nu)$ is \mpp. 

Consequently, $\psi$ is a 
\mpp{} bijection
$\psi:(\oi,\gl)\to (B,\nu)$. Choose an uncountable null set $N\subseteq\oi$
(for example the Cantor set). Then $N':=\psi(N)$ is an uncountable null set
in $(B,\nu)$. The restriction of $\psi$ to $\oi\setminus N$ is a \mpp{}
bijection onto $B\setminus N'$. 
Further, $N$ and $N'\cup B\comp$, where $B\comp\=\oi\setminus B$, are both
uncountable Borel subsets of $\oi$, and thus by \refT{Tborel}, both are
isomorphic as measurable spaces to $\oi$, and thus to each other. Hence
there exists a measurable bijection $\psi_1:N\to N'\cup B\comp$.

Define $\psi_2:\oi\to\oi$ by $\psi_2(x)=\psi(x)$ when $x\notin N$ and 
$\psi_2(x)=\psi_1(x)$ when $x\in N$. Then $\psi$ is a \mpp{} bijection
$(\oi,\gl)\to(\oi,\nu)$. Consequently, $\psi_2\qw\circ\gf$ is a \mpp{}
bijection of $(\sss,\mu)$ onto $(\oi,\gl)$.
\end{proof}

It is easy to handle atoms too. An atom in a Borel \ps{} is, up to a null
set, just a single point with a point mass; hence, a Borel space is atomless
if and only if it has no point masses, \ie{} no point with positive measure. 
In any Borel \ps{} there is at most a countable number of point masses,
and removing them we obtain an atomless Borel measure space.
This leads to the following characterization.

\begin{theorem}
  \label{Tborelpatom}
A \ps{} is Borel if and only if it is isomorphic, by a \mpb, to one of the
following spaces.
\begin{romenumerate}
\item \label{tborelpatomcount} 
A countable set $\cD=\set{x_i}_{i=1}^n$ (where $n\le\infty$),
with all subsets measurable and the discrete measure given by 
$\mu(A)=\sum_{i:x_i\in  A}p_i$, 
for some $p_i\ge0$. (Necessarily $\sum_i p_i=1$.)
\item \label{tborelpatomcount0} 
The disjoint union $\cD\cup N$, where $\cD$ is as in \ref{tborelpatomcount} 
and $N$ is a null set given by any given uncountable Borel measurable space 
equipped with zero measure. 
(We may choose for example $N=\oi$ with zero measure, or the Cantor set with
$\gl$, which vanishes there.)
\item \label{tborelpatomcont} 
The disjoint union of a closed interval $([0,r],\gl)$ with $0<r\le1$
and a countable
  set $\cD$ as in \ref{tborelpatomcount} (possibly empty); in this case 
$r+\sum_i p_i=1$, and we may further assume that each $p_i>0$.
\end{romenumerate}
\end{theorem}

\begin{proof}
If $(\sss,\mu)$ is a Borel \ps, let $D\=\set{x\in\sss:\mu\set x>0}$ and
$\sss'\=D\comp=\sss\setminus D$.
Then $D$ is countable, and $(\sss',\mu)$ is atomless. 
Let $r\=\mu(\sss')$. If $r>0$, then by \refT{Tborel} and a scaling,
$(\sss',\mu)$ is isomorphic to $[0,r]$, which yields 
\ref{tborelpatomcont}. If $r=0$, then $\sss'$ is a null set. If further
$\sss'$ is uncountable, then $(\sss',\mu)=(\sss',0)$ is isomorphic to
$(N,0)$ for any uncountable Borel space by \refT{Tborel}(iv),
which yields \ref{tborelpatomcount0}. Finally, if $\sss'$ is countable, then
$\sss$ is countable and \ref{tborelpatomcount} holds. 

The converse is obvious.
\end{proof}

\begin{theorem}\label{Toi2borel}
  If\/ $\sss$ is a Borel probability space, then there is a \mpp{} map
  $\oi\to\sss$. 
\end{theorem}
\begin{proof}
  It suffices to show this for the spaces in 
\refT{Tborelpatom}\ref{tborelpatomcount}--\ref{tborelpatomcont}, and for
these it is easy to construct explicit maps. (For each $x\in\cD$, map a
suitable interval of length $\mu\set x$ to $x$; in \ref{tborelpatomcont},
map $[0,r]$ onto itself by the identity map.) 
\end{proof}

\subsection{Lebesgue spaces}\label{SSLebesgue}

A \emph{Lebesgue \ps} is a \ps{} that is  the completion of a Borel \ps;
equivalently (see \refT{Tborel}), it is isomorphic to a Polish space (or,
equivalently, a Borel subset of a Polish space) equipped with the completion
of a Borel measure. 

\refT{Tborelpatom} leads directly to the following characterization.
\begin{theorem}
  \label{TLeb}
A \ps{} is Lebesgue if and only if it is isomorphic, by a \mpb, to one of the
spaces given in \refT{Tborelpatom}, with the modifications that in
\ref{tborelpatomcount0} all subsets of $N$ are measurable (with measure $0$),
and in \ref{tborelpatomcont} 
the interval $[0,r]$ is equipped with the Lebesgue $\gs$-field $\cL$.
\nopf
\end{theorem}

In other words, every Lebesgue \ps{} is, possibly ignoring a null sets,
isomorphic to either a countable discrete space, an interval
$([0,r],\cL,\gl)$, or a disjoint union of an interval and a countable
discrete part.

\begin{corollary}
  \label{CLeb}
An atomless Lebesgue space is isomorphic to $(\oi,\cL,\gl)$.
\end{corollary}
\begin{proof}
  Immediate from either \refT{TLeb} or \refT{Tborel}.
\end{proof}

\begin{remark}
  Lebesgue spaces were introduced by \citet{Rohlin}
by a different, intrinsic, definition, see also  \citet{Haezendonck}.
The equivalence  to the definition above follows from 
\cite[\S2.4]{Rohlin} or
\cite[Remark 2, p.~250]{Haezendonck}.
\end{remark}

\section{Graph limits}\label{Alimits}
As said in the introduction, graph limits were introduced by
  \citet{LSz} and further developed by \citet{BCLSV1,BCLSV2}.
The central idea in  graph limit theory is to assign limits to (some)
sequences $G_n$ of (unlabelled) graphs with $|G_n|\to\infty$.
Part of the importance of this notion is the fact that several different
definitions of convergence turn out to be equivalent.
One definition is the following, which has the advantage that it easily is
adapted to many 
other situations such as hypergraphs, bipartite graphs, directed graphs,
compactly decorated graphs and posets, see 
\cite{Austin,BCLSV1,BCLSV2,SJ209,ElekSz,Hoppen,SJ224,KR,LSz:topology,LSz:compact}.

For each $k\le |G_n|$, let $G_n[k]$ be the random induced subgraph  of $G_n$
with $k$ vertices obtained by selecting $k$ (distinct) vertices
$v_1,\dots,v_k\in G_n$ at random (uniformly); we regard $G_n[k]$ as a
labelled graph with the vertices labelled $1,\dots,k$; equivalently, we
regard $G_n[k]$ as a graph with vertex set \set{1,\dots,k}.

\begin{definition}\label{Dlimit}
A sequence of graphs $(G_n)$ with $|G_n|\to\infty$ converges if for each fixed
$k$, the distribution of the random graph $G_n[k]$ converges as \ntoo.
\end{definition}

In other words, for each $k$ and each labelled graph $G$ with $|G|=k$, we
require that 
$\lim_\ntoo\P(G_n[k]=G)$ exists.


Given this notion of convergence, graph limits can be defined abstractly, as
equivalence classes of convergent sequences of graphs. Equivalently,
one can easily introduce a metric on the set of unlabelled finite graphs
such that the convergent sequences become the Cauchy sequences in the
metric, and then construct the completion of this metric space.

It turns out that the space of limits can be identified with the
quotient space $\bwx\=\bigcup_\sss\www(\sss)/\equ$ defined in
\refS{Scoupling}, see
\citet{LSz} and \citet{BCLSV1}.
In other words, every graph limit is represented by a graphon, but
non-uniquely, since every equivalent graphon represents the same graph
limit. (Conversely, non-equivalent graphons represent different graph limits.)

Moreover, convergence to graph limits can be described by the cut metric.
If $(G_n)$ is a sequence of graphs with $|G_n|\to\infty$, 
and $W$ is a graphon, then $G_n$ converges to 
the graph limit represented by $W$
if and only if $\dcut(W_{G_n},W)\to0$, where $W_{G_n}$ is as in \refE{Ewg1}.
In this case we also say that $(G_n)$ converges to $W$, and write $G_n\to W$
(remembering the non-uniqueness of $W$).

\begin{remark}\label{Rgn}
In particular, 
for every graphon $W$, there exist sequences of graphs $(G_n)$ such that
$G_n\to W$. (One construction of such $G_n$ is the random construction in
\refApp{Arg} below.)
\end{remark}

Convergence to graph limits can also be described by the homomorphism densities
defined in \refApp{Ahomo}: $G_n\to W$ if and only if $t(F,G_n)\to t(F,W)$
for every simple graph $F$.

For details and many other results,
 see \citet{BCLSV1};
for further aspects, see \eg{} 
\citet{Austin},
\citet{BRmetrics},
\citet{BCLSV2}, 
\citet{SJ209},
\citet{LSz:topology},
and \refApp{Arg} below.

\section{Homomorphism densities}\label{Ahomo}

Define, following \cite{BCLSV1} and \cite{LSz},
for a graphon (or, more generally, any bounded symmetric function)
$W:\sssq\to\oi$ and a simple graph 
$F$ vith vertex set $V(F)$ and edge set $E(F)$,
the \emph{the homomorphism density} 
\begin{equation}\label{t}
 t(F,W)\=\int_{\sss^{V(F)}}\prod_{ij\in E(F)}W(x_i,x_j)
 \dd\mu(x_1)\dotsm\dd\mu(x_{|F|}).
\end{equation}
If $X_i$ are \iid{} random variables with values in $\sss$ and distribution
$\mu$, we can write \eqref{t} as 
\begin{equation}
t(F,W)\=\E \prod_{ij\in E(F)}W(X_i,X_j).  
\end{equation}

The homomorphism densities can be defined for graphs too by
$t(F,G)\=t(F,W_G)$. It is easily seen that $t(F,G)$ is the proportion of
maps $V(F)\to V(G)$ that are graph homomorphisms (or, equivalently, the
probability that a random map $V(F)\to V(G)$ is a graph homomorphism.
(This explains the name homomorphism density.)

The homomorhism denisities have a central place in the graph limit theory.
In particular, as shown in \cite{BCLSV1},
$G_n\to W$ if and only if $t(F,G_n)\to t(F,W)$ for every
simple graph $F$.

The definition \eqref{t} makes sense also for loopless multigraphs $F$,
where we allow repeated edges. (Loops are not allowed, since 
we want $t(F,W)=t(F,W')$ when $W=W'$ \aex, and this rules out
a factor $W(x_i,x_i)$ in \eqref{t}.)

\begin{example}\label{Ek}
Let $M_k$ be the multigraph with 2 vertices connected by $k$ parallel edges.
Then $t(M_k,F)=\intsq W^k$.  
\end{example}

We have seen in \refT{Teq1} that $t(F,W)=t(F,W')$ when
$W\equ W'$, for every multigraph $F$. In other words, the mapping $W\mapsto
t(F,W)$ yields a well-defined mapping 
on the quotient space $\bwx\=\cW^*/\equ$, which is the same as the space of
graph limits, see \refApp{Alimits}.

\begin{lemma}\label{Ltcont}
  The mapping $W\mapsto t(F,W)$ is 
continuous on $(\bwx,\dcut)$ if and only if $F$ is a simple graph.
\end{lemma}

In other words, if $\dcut(W_n,W)\to0$, then $t(F,W_n)\to t(F,W)$ for every
simple graph $F$. However, if $F$ is a multigraph with parallel edges, then
$\dcut(W',W)=0$ implies $t(F,W')=t(F,W)$, but
$\dcut(W_n,W)\to0$ does not imply $t(F,W_n)\to t(F,W)$.

\begin{proof}
It is easy to see that $W\mapsto t(F,W)$ is continuous in $\dcut$ for every
simple $F$, see \cite{BCLSV1} or \cite{LSz}; more precisely, for any
graphons $W$ and $W'$,
\begin{equation}\label{tlip}
  \abs{t(F,W)-t(F,W')}\le|E(F)|\,\dcut(W,W').
\end{equation}

For the converse,
suppose that the loopless multigraph $F$ is not simple,
and let $F'$ be the simple graph obtained by identifying parallel edges in
  $F$. Thus $V(F')=V(F)$, but $|E(F')|<|E(F)|$.

Let $W$ be the constant graphon $1/2$ defined on  $\oi$,
and let $G_n$ be a sequence of graphs such that $G_n\to W$.
(See \refR{Rgn}. Such sequences are known as quasirandom, see \cite{LSz}.
For example, $G_n$ can be a realization of the random graph
$G(n,1/2)$, see \refApp{Arg}.)

Let $W_{G_n}$ be the graphon corresponding to $G_n$ as in \refE{Ewg1};
we thus have $\dcut(W_{G_n},W)\to0$. On the other hand, 
$W_{G_n}$ is \oivalued, and thus $t(F,W_{G_n})=t(F',W_{G_n})$ by \eqref{t}.
Hence, using the already proved part of the lemma for $F'$,
\begin{equation*}
  t(F,W_{G_n})=t(F',W_{G_n})\to t(F',W)=2^{-|E(F')|}>2^{-|E(F)|}=t(F,W).
\qedhere
\end{equation*}
\end{proof}

\begin{example}\label{EW2}
In particular, $W\mapsto t(K_2,W)=\intsq W$ is continuous in the cut metric,
but
$W\mapsto t(M_2,W)=\intsq W^2$, see \refE{Ek}, is not.
More generally, $W\mapsto \intsq W^k$ is not continuous for any $k>1$.
\end{example}

If we use the stronger metric $\dl$, we have continuity for multigraphs
too. (This metric is, however, much less useful.)

\begin{lemma}\label{Ltcont1}
  The mapping $W\mapsto t(F,W)$ is 
continuous on $(\bwx,\dl)$ for every loopless multigraph.
\end{lemma}

We omit the easy proof, similar to the proof for $\dcut$ and simple graphs
in \cite{BCLSV1} or \cite{LSz}.

\section{Graphons and random graphs}\label{Arg}

Let $W$ be a graphon, defined on some \ps{} $\sss$.
For $1\le n\le\infty$, 
let $[n]=\set{i\in\bbN: i\le n}$; thus $[n]=\setn$ if $n$ is finite and
$[\infty]=\bbN$. We define a random graph $G(n,W)$ with vertex set $[n]$
by first taking an \iid{} sequence $\set{X_i}_{i=1}^n$ of random points in
$\sss$ with the 
distribution $\mu$, and then, given this sequence, letting $ij$ be an edge
in $G(n,W)$ with probability $W(X_i,X_j)$; 
for a given sequence $(X_i)_i$,
this is done independently for
all pairs $(i,j)\in[n]^2$ with $i<j$.
(I.e., we first sample $X_1,X_2,\dots$ at random, and then toss a biased coin
for each possible edge.)

The random graphs $G(n,W)$ thus generalize the standard random graphs $G(n,p)$
obtained by taking $W=p$ constant.
Note that we may construct $G(n,W)$ for all $n$ by first constructing
$G(\infty,W)$ and then taking the subgraph induced by the first $n$ vertices.

This construction was introduced in
graph limit theory in  \cite{LSz} and \cite{BCLSV1}.
(For other uses, see \eg{} \cite{SJ178} and \cite{DF}.)

\begin{remark}
  \label{Rgnw}
If $F$ is a labelled graph, then the homomorphism density $t(F,W)$
in \eqref{t}
equals the probability that $F$ is a labelled subgraph of $G(\infty,W)$
(or of $G(n,W)$ for any $n\ge|F|$).

In particular, this shows that the family  
$\bigpar{t(F,W)}_F$
and the distribution of $G(\infty,W)$ determine each other;
see further \refT{Teq1} and \cite{SJ209}.
\end{remark}

\begin{remark}
\label{Rrf}
  If $W$ is a random-free graphon, \ie, $W(x,y)\in\setoi$ \aex, then the
  construction of $G(n,W)$ simplifies. We sample \iid{} $X_1,X_2,\dots$ as
  before, and draw an edge $ij$ if and only if $W(X_i,X_j)=1$; thus the
  second random step in the construction disappears. 
(This explains the
  name ``random-free''; of course, $G(n,W)$ still is random, but it is now a
  deterministic function of the random $X_i$.)  
\end{remark}

The infinite random graph $G(\infty,W)$ is an exchangeable random graph,
\ie, its distribution is invariant under permutations of the vertices,
and every exchangeable random graph is a mixture of such graphs, \ie, it can
be obtained by this construction with a random $W$. 
This is an instance of the representation theorem for exchangeable arrays by
\citet{Aldous} and \citet{Hoover}, see also \citet{Kallenberg:symmetries}.
Moreover, by \refT{Teq1}, if $W'$ is another graphon, then $G(\infty,W)$
and $G(\infty,W')$ have the same distribution if and only if $W\cong W'$.
Consequently, the mapping $W\mapsto G(\infty,W)$ gives a bijection between
the set
$\bwx=\cWx/\equ$ of
equivalence classes of graphons
and a subset $\cXo$ of the set $\cX$ of distributions of exchangeable
infinite random graphs; this subset $\cXo$ is easily charaterized in
several different ways, for example as follows.
\begin{lemma}\label{Lexch}
For an exchangeable infinite random graph $\bG$,  the following are
equivalent, and thus all characterize
$\cL(\bG)\in\cXo$.
\begin{romenumerate}
\item 
$\bG\eqd G(\infty,W)$ for some
graphon $W$.
\item 
The distribution $\cL(\bG)$ is an extreme point in $\cX$.
\item 
$\bG$ is ergodic: every property that is (\as) invariant under left-shift
(\ie, delete vertex $1$ and its edges and relabel the remaining vertices
  $i\mapsto i-1$) 
has probability $0$ or $1$.
\item 
Every property of $\bG$ that is (\as) invariant under finite permutations of the
vertices 
has probability $0$ or $1$.
\item \label{LexchInd}
For any two disjoint subsets of vertices $V_1$ and $V_2$, the induced
subgraphs $\bG|_{V_1}$ and $\bG|_{V_2}$ are independent.
\end{romenumerate}
\end{lemma}
\begin{proof}
 See \cite{SJ209} and
\cite{Kallenberg:symmetries}.
\end{proof}

\subsection{Graph limits and random graphs}

There is also a simple connection between graph limits and 
exchangeable infinite random graphs. By \refD{Dlimit}, if $(G_n)$ is a
convergent sequence of graphs with $|G_n|\to\infty$, then for each $k$ there
exists a random graph $G[k]$ on the vertex set $[k]$ such that $G_n[k]\dto
G[k]$. The distributions of $G[k]$ for different $k$ are consistent, so by
Kolmogorov's extension theorem, there exists a random infinite graph $\bG$ 
on $[\infty]$ such
that $G[k]\eqd \bG|_{[k]}$, \ie, $G_n[k]\dto \bG|_{[k]}$. 
Each $G_n[k]$ has an  exchangeable distribution, and thus so has each
$G[k]$; consequently, $\bG$ is an exchangeable infinite random graph;
furthermore, it is easily seen that $\bG$ satisfies
\refL{Lexch}\ref{LexchInd}, and 
thus its distribution belongs to $\cXo$.
Thus every graph limit can be represented by 
an exchangeable infinite random
graph with distribution in $\cXo$.
Conversely,
if $\bG$ is any exchangeable infinite random
graph with a distribution in $\cXo$, then the induced subgraphs
$G_n\=\bG|_{[n]}$ \as{} satisfy 
$G_n[k]\dto\bG|_{[k]}$ for every $k$, as can be seen from the limit theorem
for reverse martingales \cite{SJ209} or directly \cite{LSz}; 
thus the sequence $(G_n)$ converges \as,
and its limit is represented by the infinite random graph $\bG$.

This yields a bijection between the set of graph limits and 
the set $\cXo$,
characterized in \refL{Lexch}, 
of distributions of exchangeable infinite random graphs.

This connection between graph limits and 
(distributions of) exchangeable infinite random graphs combines with the 
connection above between (equivalence classes of) graphons and 
(distributions of) exchangeable infinite random graphs to prove the central
fact stated in \refApp{Alimits}
that there is a bijection between graph limits and equivalence classes of
graphons;  see further
\cite{Austin}, 
\cite{SJ209}, 
\cite{KR}, 
\cite{LSz:rg}.

In particular, 
for any graphon $W$, we have \as{} $G(n,W)\to W$  as \ntoo, in the sense of
\refApp{Alimits} \cite{LSz}, \cf{} \refR{Rgn}.

\begin{remark}
  This method of proving the connection between graph limits and graphons
through the use of exchangeable infinite random graphs as an intermediary
generalizes immediately to several extensions of the theory, and it may be
used to find the correct analogue of graphons in new situations.
See for example \cite{Austin} (hypergraphs) and \cite{SJ209} (bipartite
graphs, directed graphs).

Another example is compact decorated graphs \cite{LSz:compact}, which are
graphs with edges labelled by elements of a fixed second-countable compact
space (\ie, a  compact metrizable space \cite[Theorem 4.2.8]{Engelking}) $\cK$; 
this includes several interesting cases. 
$\cK$-decorated graph limits are defined as in \refD{Dlimit}, now with
$\cK$-decorated graphs.
The arguments sketched above show that there is a bijection between
$\cK$-decorated graph limits and distributions of exchangeable $\cK$-decorated
infinite random graphs satisfying the properties in \refL{Lexch}, and a
further bijections to equivalence classes of graphons, where the graph\-ons
now take their values in the space $\cP(\cK)$ of Borel probability measures
on $\cK$. 
(The representation theorem in \cite{Kallenberg} yields a representation
where the label of $ij$ is $f(X_i,X_j,\xi_{ij})$ for some fixed function
$f:\oi^3\to\cX$ with $X_i$ and $\xi_{jk}$ uniform on $\oi$ and independent of
each other; it is easily seen that this leads to an
equivalent representation by $\cP(\cK)$-valued  graphons
$W:\oiq\to\cP(\cK)$.)
For a different proof, see \cite{LSz:compact}.
Many results in Sections \ref{Scoupling}--\ref{Sequiv} above extend to this
case, but we leave that to the reader.

In fact, the arguments above on the equivalences
work for any Polish space $\cK$, also non-compact;
however, compactness implies that the resulting space of decorated graph
limits is compact, which is important for some results.
\end{remark}

\subsection{Entropy}
If we regard $\gnw$ as a labelled random graph, we may identify it with
the collection $(J_{ij})_{i<j}$ of the $\binom n2$ edge indicators 
$J_{ij}\=\ett{ij\text{ is an edge}}$, $1\le i<j\in[n]$. 
For finite $n$, $\gnw$ is thus a discrete random variable with 
$2^{\binom n2}$ possible outcomes.
Recall that for any discrete random variable $Z$, with outcomes (in any
space) having probabilities $p_1,p_2,\dots$, say, its \emph{entropy}
$\ent(Z)$ is defined by
\begin{equation*}
  \ent(Z)\=-\sum_i p_i\log p_i.
\end{equation*}
We also write $\ent(Z_1,\dots,Z_n)$ for the entropy of a vector
$(Z_1,\dots,Z_n)$, and $\ent(Z\mid Z')$ for the entropy of the conditioned
random variable $(Z\mid Z')$.

The following asymptotic calculation of the entropy of \gnw{} is a special
case of the symmetric version of the formula in 
\cite[Remarks,  p.~146]{Aldous85}. 
Let
\begin{equation*}
  h(p)\=-p\log p -(1-p)\log(1-p),\qquad p\in\oi;
\end{equation*}
thus the entropy of a \oivalued{} random variable $Z\in\Be(p)$ is $h(p)$.
Note that $h$ is continuous on $\oi$ with $0\le h(p)\le\log2$ and
$h(0)=h(1)=0$.  
\begin{theorem}\label{Tent}
  Let $W$ be a graphon, defined on a \ps{} $(\sss,\mu)$. Then, as \ntoo,
  \begin{equation}\label{tent}
\frac{\ent(G(n,W))}{\binom n2} 
\to \iint_{\sssq} h\bigpar{W(x,y)}\dd\mu(x)\dd\mu(y). 	
  \end{equation}
\end{theorem}

\begin{proof}
  If we condition on $X_1,\dots,X_n$, then $J_{ij}$ are independent and each
  $J_{ij}\in\Be(p_{ij})$ with $p_{ij}=W(X_i,X_j)$.
Thus, using in the calculations
here and below some simple standard results on entropy, 
\begin{equation*}
  \begin{split}
\ent\bigpar{\gnw\mid X_1,\dots,X_n}	
&= \sum_{i<j}\ent\bigpar{J_{ij}\mid X_1,\dots,X_n}	
= \sum_{i<j}\ent\bigpar{\Be(p_{ij})}	
\\&
=\sum_{i<j} h(p_{ij})
=\sum_{i<j} h(W(X_i,X_j)).
  \end{split}
\end{equation*}
Hence,
\begin{equation*}
  \begin{split}
\ent\bigpar{\gnw}	
&\ge\E 
\ent\bigpar{\gnw\mid X_1,\dots,X_n}	
=\E\sum_{i<j} h(W(X_i,X_j))
\\&
=\binom n2 \iint_{\sssq} h\bigpar{W(x,y)}\dd\mu(x)\dd\mu(y). 	
  \end{split}
\end{equation*}
Thus the \lhs{} of \eqref{tent} is greater than or equal to the \rhs{}
for every $n\ge2$.

To obtain a corresponding upper bound, we for convenience assume that
$\sss=(0,1]$, as we may by \refT{Trep} (noting that $\iint h(W)$ is preserved
by pull-backs, and thus by equivalence, see \refT{Tchain}).

Fix an integer $m$ and let $M_i\=\ceil{m X_i}$. Thus $M_i=k\iff X_i\in I_{km}$.
We have
\begin{equation}\label{e1}
  \begin{split}
\ent\bigpar{\gnw}	
&\le
\ent\bigpar{\gnw,M_1,\dots,M_n}	
\\&
=
\ent\xpar{M_1,\dots,M_n}+\E\bigpar{\ent\bigpar{\gnw\mid M_1,\dots,M_n}}.
  \end{split}
\end{equation}
Since $M_1,\dots,M_n$ are independent and uniformly distributed on
\set{1,\dots,m},
\begin{equation}
  \label{emm}
\ent\xpar{M_1,\dots,M_n}=\sum_{i=1}^n\ent(M_i)=n\log m.
\end{equation}
Moreover,
\begin{equation}\label{ems}
  \begin{split}
\ent\bigpar{\gnw\mid M_1,\dots,M_n}
\le \sum_{i<j}\ent\bigpar{J_{ij}\mid M_1,\dots,M_n}
=
 \sum_{i<j}\ent\bigpar{J_{ij}\mid M_i,M_j}.
  \end{split}
\end{equation}

Define, for $k,l=1,\dots,m$,
\begin{equation*}
  \begin{split}
  w_m(k,l)&\=\E\bigpar{W(X_1,X_2)\mid M_1=k,\,M_2=l}
=m^2\int_{I_{km}}\int_{ I_{lm}}W(x,y)\dd x\dd y,	
  \end{split}
\end{equation*}
the average of $W$ over $I_{km}\times I_{lm}$, and let 
\begin{equation*}
  W_m(x,y)\=w_m(k,l) \quad \text{if }x\in I_{km},\, y\in I_{lm}.
\end{equation*}
Thus $W_m(X_i,X_j)$ equals the conditional expectation
$\E\bigpar{W(X_1,X_2)\mid M_1,M_2}$.

Given $M_i=k$ and $M_j=l$,
\begin{equation*}
  \P(J_{ij}=1)=\E\bigpar{W(X_1,X_2)\mid M_1=k,\,M_2=l}
=  w_m(k,l),
\end{equation*}
and thus 
\begin{equation*}
\ent\bigpar{J_{ij}\mid M_i=k,\,M_j=l}=h\bigpar{w_m(k,l)}.
\end{equation*}
Consequently,
\begin{equation}\label{emmsan}
\E\bigpar{\ent\bigpar{J_{ij}\mid M_i,\,M_j}}
=m\qww\sum_{k,l=1}^m h\bigpar{w_m(k,l)}
=\iint_{\oiq} h\bigpar{W_m(x,y)}\dd x\dd y.
\end{equation}
Combining \eqref{e1}--\eqref{emmsan}, we obtain
\begin{equation*}
\ent(\gnw)
\le n \log m + \binom n2\iint_{\oiq} h\bigpar{W_m(x,y)}\dd x\dd y.
\end{equation*}
and thus, for every $m\ge1$,
\begin{equation*}
\limsup_{\ntoo}{\binom n2}\qw\ent(\gnw)
\le \iint_{\oiq} h\bigpar{W_m(x,y)}\dd x\dd y.
\end{equation*}
Now let $m\to\infty$. Then $W_m(x,y)\to W(x,y)$ \aex, and thus the \rhs{}
tends to $\iint h(W)$ by dominated convergence.
\end{proof}

\section{Other versions of the cut norm}\label{Acut}

There are several other versions of the cut norm that are equivalent 
to the versions in \eqref{cutnorm1} and \eqref{cutnorm2}
within constant factors or, in \refSS{SSoperator}, at least in a weaker sense.

\subsection{Restrictions on the pairs of subsets}
First, we may restrict the subsets $S$ and $T$ of $\sss$
in \eqref{cutnorm1} in various
ways. 
\citet[Section 7]{BCLSV1} give
three versions where it is assumed that, respectively, $S=T$, $S$ and $T$
are disjoint, and $S$ and $T$ are the complements of each other, \ie,
\begin{align}
 \cnx3 W &\= \sup_{S}
  \Bigabs{\int_{S \times S} W(x,y) \dd\mu(x)\dd\mu(y)},
\label{cutnorm3}
\\
 \cnx4 W &\= \sup_{S\cap T=\emptyset}
  \Bigabs{\int_{S \times T} W(x,y) \dd\mu(x)\dd\mu(y)},
\label{cutnorm4}
\\
 \cnx5 W &\= \sup_{S}
  \Bigabs{\int_{S \times  S\comp} W(x,y) \dd\mu(x)\dd\mu(y)}.
\label{cutnorm5}
\end{align}

These have natural combinatorial interpretations for graphs as follows.
For a graph $G$ with vertex set $V$ and edge set $E$, we define, for
$A,B\subseteq V$,
\begin{equation}
e(A,B)= e_G(A,B)\=\lrabs{\bigset{(x,y)\in A\times B:\set{x,y}\in E}};
\end{equation}
we also write $e_G(A)\=e_G(A,A)$. (Thus, if $A$ and $B$ are
disjoint, then $e(A,B)$ is the number of edges between $A$ and $B$. On the
other hand, $e(A)$ is twice the number of edges in $A$.)

\begin{lemma}
  \label{Lcut345G}
Let $G_1$ and $G_2$ be two graphs on the same vertex set $V$, and let $n\=|V|$.
Then, for both versions $\wgv G$ and $\wgi G$,
\begin{align}
  \cnx3{W_{G_1}-W_{G_2}}&=
n\qww\max_{A\subseteq V} \bigabs{e_{G_1}(A) - e_{G_2}(A)},
\\
\cnx4{W_{G_1}-W_{G_2}}&=n^{-2}\max_{A\cap B=\emptyset}
 \bigabs{ e_{G_1}(A,B) - e_{G_2}(A,B) },
\\
\cnx5{W_{G_1}-W_{G_2}}&= n^{-2}\max_{A\subseteq V}\, 
 \bigabs{ e_{G_1}(A, A\comp) - e_{G_2}(A,A\comp) }.
\end{align}
\end{lemma}

In particular, $\cnx5{W_{G_1}-W_{G_2}}$ measures directly the maximal
difference in size of cuts in $G_1$ and $G_2$, which explains the name ``cut
norm''. 

\begin{proof}
For $\wgv G$ this is immediate, since for every $S,T\subseteq\sss=V$ we have
$\int_{S\times T} \wgv{G_\ell}=n\qww e_{G_\ell}(S,T)$.

For $\wgi G$, 
let $\bigpar{\aaa\ell_{ij}}_{ij}$ be the adjacency matrix of $G_\ell$, so 
$\aaa\ell_{ij}\=\bigett{\set{i,j}\in E(G_\ell)}$.
If $S,T\subseteq\oi$, let $s_i\=\gl(S\cap I_{in})$,
$t_j\=\gl(T\cap I_{jn})$. 
Then
\begin{equation}
  \label{qk}
\int_{S\times T}\bigpar{W_{G_1}-W_{G_2}}
=\biggabs{\sum_{i,j=1}^n s_it_j \bigpar{\aaa1_{ij}-\aaa2_{ij}}}.
\end{equation}
It follows that
\begin{align}
  \cnx3{W_{G_1}-W_{G_2}}&=
\sup_{0\le s_i\le 1/n} 
\biggabs{\sum_{i,j=1}^n s_is_j \bigpar{\aaa1_{ij}-\aaa2_{ij}}},
\\
  \cnx4{W_{G_1}-W_{G_2}}&=
\sup_{\substack{0\le s_i\le 1/n\\0\le u_j\le 1}} 
\biggabs{ 
 \sum_{i,j=1}^n s_iu_j(1-s_j) \bigpar{\aaa1_{ij}-\aaa2_{ij}}},
\\
  \cnx5{W_{G_1}-W_{G_2}}&=
\sup_{0\le s_i\le 1/n}
\biggabs{\sum_{i,j=1}^n s_i(1-s_j)\bigpar{\aaa1_{ij}-\aaa2_{ij}}}.
\end{align}
Since $\aaa1_{ii}=\aaa2_{ii}=0$, the diagonal terms in these sums 
vanish, and thus the sums are affine functions of each $s_i$ and (for
$\cnx4\cdot$) $u_i$. Hence, the suprema are attained when all $s_i$ are
either 0 or $1/n$, and $u_i$ 0 or 1, \ie, when $S$ and $T$ are unions
$S=\bigcup_{i\in A} I_{in}$ and $T=\bigcup_{j\in B} I_{jn}$ for some
$A,B\subseteq V$, but then
$\int_{S\times T} \wgi{G_\ell}=n\qww e_{G_\ell}(A,B)$, so we obtain the same
result as for $\wgv G$.
\end{proof}

\begin{lemma}[\cite{BCLSV1}]\label{Lcnx}
  If\/ $\sss$ is atomless and $W\in L^1(\sssq)$ is symmetric, then 
the norms $\cnx iW$, $i=1,...,5$, are equivalent.
More precisely,  
\begin{align}
 \cnone{W}& \le \cntwo{W} \le 4\cnone{W} &&\text{for all\/ $\sss$ and $W$};
\label{lcn2}
\\
\frac12\cnx1 W & \le \cnx3 W \le \cnx1 W && \text{if\/ $W$ is symmetric;}
\label{lcn3}
\\
\frac14\cnx1 W & \le \cnx4 W \le \cnx1 W && \text{if\/ $\sss$ is atomless;}
\label{lcn4}
\\
\frac23\cnx4 W & \le \cnx5 W \le \cnx4 W && \text{if\/ $W$ is symmetric.}
\label{lcn5}
\end{align}
\end{lemma}
\begin{proof}
The inequalities \eqref{lcn2} were given  in \eqref{cn1=2}.
For the others, the \rhs s are trivial. 

For the \lhs s,
  let $W(S,T)\=\int_{S\times T}W$. Then \eqref{lcn3} follows from
  $W(S,T)=W(T,S)$ and
  \begin{multline*}
	W(S,T)+W(T,S)
=
W(S\cup T,S\cup T)+W(S\cap T,S\cap T)
\\
-W(S\setminus T,S\setminus T)-W(T\setminus S,T\setminus S).
  \end{multline*}

For \eqref{lcn4} we randomize. Let $(A_i)_{i=1}^n$ be a partition of $\sss$
with $\mu(A_i)=1/n$ for each $i$ (such partitions exist when $\sss$ is atomless
as a consequence of \refL{Lfinns}), 
let $I$ be a random subset of \set{1,\dots,n} defined by including
each element with probability $1/2$, independently of each other,
and define a random subset $B$ of $\sss$ by $B\=\bigcup_{i\in I} A_i$.
Then, for any $S,T\subseteq\sss$, 
$\E\bigabs{ W(S\cap B,T\setminus B)}\le\cnx4 W$. Moreover,
\begin{equation*}
  \begin{split}
  \E W(S\cap B,T\setminus B)&=\sum_{i\neq j}\frac14 W(S\cap A_i,T\cap A_j)
\\&
=\frac14 W(S,T)-\frac14\sum_{i} W(S\cap A_i,T\cap A_i)	.
  \end{split}
\end{equation*}
The last sum is the integral of $W$ over a subset of $\sssq$ of measure
$1/n$,
so it tends to 0 as \ntoo. 
Consequently, $\frac14 W(S,T)\le\cnx4 W$, and \eqref{lcn4} follows.

For \eqref{lcn5}, assume that $S\cap T=\emptyset$. Let
$R\=(S\cup T)\comp$. The result follows from
\begin{equation*}
  \begin{split}
W(S,T)+W(T,S)=W(S,T\cup R)+W(T,S\cup R)-W(S\cup T,R).	
\qedhere
  \end{split}
\end{equation*}
\end{proof}

\begin{remark}\label{Rcnx}
  Some restrictions are necessary in \refL{Lcnx}. 
For example, if $W$ is anti-symmetric ($W(x,y)=-W(y,x)$),  then
$\cnx3 W=0$, so \eqref{lcn3} does not hold for arbitrary $W$.
More generally, if $\tW(x,y)\=\frac12(W(x,y)+W(y,x))$ is the symmetrization
of $W$, then $\cnx3\cdot$ never distinguishes between $W$ and $\tW$, so
$\cnx3\cdot$ is appropriate only for symmetric $W$.

Similarly, if $\sss$ has an atom $A$ and $W(x,y)\=\ett{x,y\in A}$, then
$\cnx4 W=0$ and \eqref{lcn4} does not hold.
Hence, in general $\cnx4\cdot$ and $\cnx5\cdot$ are not appropriate for
spaces with atoms. (However, they work well also for $\wgv G$ for graphs
$G$, because $\wgv G(x,x)=0$ for every $x$, see \refL{Lcut345G} and its proof.)

If $W$ is anti-symmetric and the marginal $\ints W(x,y)\dd\mu(y)=0$, then
\begin{equation}\label{jesper}
\int_{S\times S\comp} W  
=\int_{S\times\sss} W  -\int_{S\times S} W
=0
\end{equation}
for every $S$, so $\cnx5 W=0$ and \eqref{lcn5} does not hold
(unless $W=0$ \aex).
For example, we can take $W(x,y)=\sin(2\pi(x-y))$ on $\oi$, or take
$\sss=\set{1,2,3}$ with $\mu(i)=1/3$ for each $i\in\sss$, and
$W(i,j)\in\set{-1,0,1}$ with
$W(i,j)\equiv i-j\pmod 3$.
(In fact, if $\sss$ is atomless, then $\cnx5 W=0$ if and only if $W$ is 
anti-symmetric and its marginals vanish \aex. To see this, note that if
$\cnx5{W(x,y)}=0$, then $\cnx5{W(y,x)}=0$ as well, and thus 
$\cnx5{\tW}=0$. By \refL{Lcnx}, then $\cnx2{\tW}=0$ and thus $\tW=0$
\aex. By \eqref{jesper}, $\int_S\mii W=\int_{S\times\sss} W =0$ for every
$S\subseteq\sss$, and thus $\mii W=0$ a.e.)
Cf.\ \cite[Section 9]{SJ234}.
\end{remark}

\begin{remark}
  If $W\equ W'$, then $\cnx3 W=\cnx3{W'}$; this is easily seen first for
  pull-backs by the argument in the proof of \refL{L1}, and then in general
  by \refT{Tchain}. The same holds for $\cnx4\cdot$ and $\cnx5\cdot$
  provided $W$ and $W'$ are defined on atomless spaces, using also a
  randomization argument similar to the one in the proof of
  \refL{Lcnx}. However, this is not true in general for spaces with
  atoms. For a trivial example, let $W=1$ on $\oi$ and $W'=1$ on  one-point
  space; then $\cnx4{W'}=\cnx5{W'}=0$.
\end{remark}

\begin{remark}
  The constants in \eqref{lcn2}--\eqref{lcn5} are best possible. Examples
  with equality in the left or right inequalities
are given by the following matrices, interpreted as
  functions on $\oiq$, with each row or column in an $n\times n$-matrix
  corresponding to an interval $I_{in}$ of length $1/n$ (we could use a
  space $\sss$ with $n$ points, but we want $\sss$ to be atomless):
  \begin{itemize}\newcommand\x{\phantom-}
  \item [\eqref{lcn2}:] $(1)$, \quad $\smatrixx{\x1 & -1\\-1&\x1}$;
  \item [\eqref{lcn3}:] $\smatrixx{-1&0&\x1 \\ \x0 & 1 &\x0\\ \x1 & 0 &	-1}$, 
\quad $(1)$;
  \item [\eqref{lcn4}:] $(1)$, \quad $\smatrixx{\x1 & -1\\-1&\x1}$;
  \item [\eqref{lcn5}:] $\smatrixx{\x0&\x 3& -1\\\x3&\x0&-1\\-1&-1&\x0}$, 
\quad $(1)$.
  \end{itemize}
 \end{remark}

\subsection{Complex and Hilbert space valued functions}
Another set of versions of the cut norm use \eqref{cutnorm2} but consider
other sets of functions $f$ and $g$. For example, 
we may
take the supremum over all complex-valued functions $f$ and $g$ with 
$|f|,|g|\le1$, \ie{} 
\begin{equation}\label{cutnorm2c}
 \cnc W \= \sup_{\substack{f,g:\sss\to \bbC\\\normoo{f},\normoo{g}\le1}}
  \Bigabs{\int_{\sss^2} W(x,y)f(x)g(y)\dd\mu(x)\dd\mu(y)}.
\end{equation}
It is easily seen  that $\cnc W \le2\cntwo W$, which can be improved to
\cite{Krivine}
\begin{equation}\label{cutnorm2cG2}
  \cntwo W \le \cnc W \le \sqrt2 \cntwo W,
\end{equation}
which is best possible. (For an example, consider a two-point space
$\sss=\set{1,2}$ with $\mu\set1=\mu\set2=1/2$, and let $W_1(x,y)=1/2$ and
$W_2(x,y)=\ett{x=y=1}$. Then $\cntwo{W_1-W_2}=1/4$ but
$\cnc{W_1-W_2}=\sqrt2/4$, obtained by taking $f=g=(1,\ii)$ in
\eqref{cutnorm2c}.) 

An interesting version is to allow $f$ and $g$ to take values in the unit
ball of an arbitrary Hilbert space $H$ and define
\begin{equation}\label{cutnorm2h}
 \cnh W \= \sup_{\substack{f,g:\sss\to H\\\normoo{f},\normoo{g}\le1}}
  \Bigabs{\int_{\sss^2} W(x,y)\innprod{f(x),g(y)}\dd\mu(x)\dd\mu(y)}.
\end{equation}
(Since we only consider real $W$, it is easy to see that
it does not matter whether we allow real or complex Hilbert spaces in
\eqref{cutnorm2h}.)  
In this case, the equivalence with $\cntwo W$ is a form of the famous
\emph{Gro\-t\-hen\-dieck's inequality} \cite{Grothendieck}, which says that
\begin{equation}
  \cntwo W \le \cnh W \le K_G \cntwo W,
\end{equation}
where the constant $K_G$, the real Grothendieck constant, is known to
satisfy $\pi/2 \le K_G \le \pi/2(\log(1+\sqrt2))\approx1.78221$
\cite{Krivine}.  (The lower bound is improved in an unpublished manuscript
\cite{Reeds}.) 
We also have  $ \cnc W \le \cnh W \le K_G^\bbC \cnc W$,
where $K_G^\bbC$ is the complex Grothendieck constant, known to satisfy
$4/\pi\le K_G^\bbC< 1.40491$ \cite{Haagerup}.
Moreover, $\cnc W$ is obtained by taking only a fixed Hilbert space of
dimension 2 in \eqref{cutnorm2h}.

See \cite{AlonNaor} for an algorithmic use of the version $\cnh\cdot$ of the
cut norm and Grothendieck's inequality.

\subsection{Other operator norms}\label{SSoperator}
If $W$ is a kernel on $\sss$, then it defines an integral operator
$T_W:f\mapsto\ints W(x,y)f(y)\dd\mu(y)$ (for suitable $f$).
We have already noted in \refR{Rtensor} that $\cntwo\cdot$ is the operator norm
of $T_W$ as an operator $L^\infty(\sss)\to L^1(\sss)$, but we may also
consider other spaces. 

Let, for $1\le p,q\le\infty$, $\norm{T}_{p,q}$ denote the norm of $T$ as an
operator $L^p\to L^q$.
\begin{lemma}\label{Lpq}
  If\/ $|W|\le1$, then for all $p,q\in[1,\infty]$,
  \begin{equation*}
	\cntwo{W} 
= \norm{T_W}_{\infty,1}
\le \norm{T_W}_{p,q}
\le \sqrt2 \cntwo{W}^{\min(1-1/p,1/q)}.
  \end{equation*}
Consequently, for any fixed $p>1$ and $q<\infty$, 
if $W_1,W_2,\dots$ and $W$ are graphons defined on the same space $\sss$, then
$\cn{W_n-W}\to0$ if and
only if $\norm{T_{W_n}-T_W}_{p,q}\to0$.
\end{lemma}

\begin{proof}
We know that $\cntwo{W} = \norm{T_W}_{\infty,1}$. Moreover, for any \ps{},
the inclusions $L^\infty\subseteq L^p$ and $L^p\subseteq L^1$ have norm 1, and
thus 
$ \norm{T}_{\infty,1}\le \norm{T}_{p,q}$ for any operator $T$.

Let $\theta\=\min(1-1/p,\,1/q)$, so
$1-\theta\=\max(1/p,\,1-1/q)$, 
and define $p_0,q_0\in[1,\infty]$ by
$1/p=(1-\theta)/p_0$ and $1-1/q=(1-\theta)(1-1/q_0)$.
Further, let $p_1=\infty$ and $q_1=1$.
Then $(1/p,\,1/q)=(1-\theta)(1/p_0,\,1/q_0)+\theta(1/p_1,\,1/q_1)$,
and it follows from the Riesz--Thorin interpolation theorem (see \eg{}
\cite[Theorem 1.1.1]{BL}) that, provided we work with complex $L^p$ spaces,
\begin{equation*}
  \norm{T_W}_{p,q}
\le \norm{T_W}_{p_0,q_0}^{1-\theta}\norm{T_W}_{p_1,q_1}^\theta.
\end{equation*}
By \eqref{cutnorm2c} and \eqref{cutnorm2cG2}, 
  \begin{equation*}
\norm{T_W}_{p_1,q_1}=
 \norm{T_W}_{\infty,1}
=\cnc{W}\le\sqrt2\cntwo{W},
  \end{equation*}
and the assumption $|W|\le1$ implies 
$\norm{T_W}_{p_0,q_0}\le \norm{T_W}_{1,\infty}\le\normoo{W}\le1$.
The result follows.  
\end{proof}

We consider the case $p=q=2$ further, \ie, we regard $T_W$ as an operator on
the Hilbert space $L^2(\sss)$.
If $W$ is bounded (or, more generally, in $L^2(\sssq)$), then $T_W$ is
bounded on $L^2$; it is further compact (and Hilbert--Schmidt) and
selfadjoint (because $W$ is symmetric). Hence $T_W$ has a sequence of
eigenvalues $(\gl_n)$. We define, for $1\le p<\infty$, the Schatten
$S_p$-norm of $T_W$ to be 
\begin{equation}\label{sp}
\norm{T_W}_{S_p}\=\norm{(\gl_n)}_{\ell^p}
=\Bigpar{\sum_n|\gl_n|^p }^{1/p}.
\end{equation}
(See \eg{} \cite{GK}, where also the non-selfadjoint case is treated.)
It is well-known that for $p=2$,  
$\norm\cdot_{S_2}$ equals the Hilbert--Schmidt norm
and thus
\begin{equation}\label{s2}
\norm{T_W}_{S_2}=\norm{W}_{L^2(\sssq)}.  
\end{equation}
If $p=2k$ is an even integer $\ge4$, then \eqref{sp} yields
\begin{equation}\label{c2k}
  \norm{T_W}_{S_{2k}}^{2k}=\sum_n \gl_n^{2k}=\Tr\bigpar{T_W^{2k}}
=t(C_{2k},W),
\end{equation}
where the graph $C_{2k}$ is the cycle of length $2k$.

\goodbreak
\begin{lemma}\label{Lschatten}
\quad  
\begin{romenumerate}
  \item 
For  $2<p<\infty$, if\/ $|W|\le1$, then
  \begin{equation*}
	\cntwo{W} 
= \norm{T_W}_{\infty,1}
\le \norm{T_W}_{2,2}
\le \norm{T_W}_{S_p}
\le \sqrt2 \cntwo{W}^{1/2-1/p}.
  \end{equation*}
Consequently, for any fixed $p>2$,
if\/ $W_1,W_2,\dots$ and $W$ are graphons defined on the same space $\sss$, then
$\norm{T_{W_n}-T_W}_{S_p}\to0$
if and only if
$\cn{W_n-W}\to0$.
\item For $p=2$, if\/ $|W|\le1$, then
  \begin{equation*}
	\normll{W} 
\le \norm{T_W}_{S_2}
= \norm{W}_{L^2}
\le  \normll{W}^{1/2}.
  \end{equation*}
Consequently, 
if\/ $W_1,W_2,\dots$ and $W$ are graphons defined on the same space $\sss$,
then
$\norm{T_{W_n}-T_W}_{S_2}\to0$
if and only if
$\normll{W_n-W}\to0$.
  \end{romenumerate}
\end{lemma}

\begin{proof}
\pfitem{i}
The first inequality is in \refL{Lpq} and the second is trivial, since the
operator norm $\norm{T_W}_{2,2}=\sup_n|\gl_n|$.  
Further, by this and \eqref{sp},
\begin{equation}\label{lic}
  \norm{T_W}_{S_p}^p
=\sum_n|\gl_n|^p
\le
\sum_n|\gl_n|^2 \sup_n|\gl_n|^{p-2}=\norm{T_W}_{S_2}^2\norm{T_W}_{2,2}^{p-2}.
\end{equation}
We have 
$\norm{T_W}_{S_2}=\norm{W}_{L^2(\sssq)}\le1$ by \eqref{s2}, and  
$\norm{T_W}_{2,2}\le\sqrt2\cntwo{W}\qq$ by \refL{Lpq},
and the result follows.
\pfitem{ii}
Immediate by \eqref{s2} and 
standard inequalities (\eg{} \Holder). 
\end{proof}

In particular, by (i) with $p=4$ and \eqref{c2k}, if $|W|\le1$, then
$\cntwo{W}\le t(C_4,W)^{1/4}\le \sqrt2\cntwo{W}^{1/4}$, or
\begin{equation*}
\tfrac14  t(C_4,W) \le \cntwo{W}\le t(C_4,W)\qqqq.
\end{equation*}
This was proved in \cite[Lemma 7.1]{BCLSV1} (by a slightly diferent
argument, using a version of \eqref{tlip}), where also an application is
given.

\begin{remark}
  There is no corresponding result for $p<2$. In fact, $\norm{T_W}_{S_p}$
  may be infinite for a graphon $W$. To see this, let first $W$ be constant
  $1/2$ on $\oi$ and let $(G_n)$ be a quasirandom sequence of graphs with
  $G_n\to W$. Let $W_n\=W_{G_n}$, so $\dcut(W_n,W)\to0$. 
By \cite[Lemma	5.3]{BCLSV1}, we may label the graphs $G_n$ such that
$\cn{W_n-W}\to0$.

By \eqref{s2}, $\norm{T_{W_n-W}}_{S_2}=\norm{W_n-W}_{L^2}=1/2$.
On the other hand, arguing as in \eqref{lic},
\begin{equation*}
\norm{T_{W_n-W}}_{S_2}^2
\le
\norm{T_{W_n-W}}_{S_p}^p\norm{T_{W_n-W}}_{2,2}^{2-p}
\le \sqrt2\norm{T_{W_n-W}}_{S_p}^p\cntwo{W_n-W}^{2-p}.
\end{equation*}
Since the \lhs{} is constant and the last factor tends to 0, it follows that
$ \norm{T_{W_n-W}}_{S_p}\to\infty$. Further, $\normoo{W_n-W}\le1$.
It is now an easy consequence of the closed graph theorem that there exist
bounded functions $W$ on $\oiq$ such that
$\norm{T_{W}}_{S_p}=\infty$, and by linearity there must exist such a
graphon.
(An explicit $W$ is given by a well-known analytic construction
\cite[\S III.10.3, p.~118]{GK}: let
$W(x,y)=f(x-y)$ on $\oiq$, where $f$ is a continuous even function with
period $1$ on $\bbR$ such that $\sum|\widehat f(n)|^p=\infty$ for all $p<2$;
such a function was constructed by \citet{Carleman}, see also
\cite[V.4.9]{Zygmund}.)
\end{remark}

\section{The weak topology on $\cW(\sss)$}

Consider the space $\cW=\cW(\sss)$ of graphons on a fixed \ps{} $\sss$.
We have discussed two different metrics on this space, given by the norms
$\normll\fyll$ and $\cn\fyll$; these give two different
topologies on $\cW(\sss)$.

Another topology on $\cW(\sss)$ is the \emph{weak topology}
$\gs$, regarding $\cW(\sss)$ as a subset of $L^1(\sssq)$. This topology is
generated 
by the functionals
$\chi_h:W\mapsto\int_\sssq hW$ for $h\in L^\infty(\sssq)$, in the standard
sense that it is the weakest topology that makes all these maps continuous.
Actually, since the functions in $\cW(\sss)$ are uniformly bounded, we
obtain the same 
topology from many different families of such functionals.

We state this also for subsets of $L^1(\sss)$
and writing $\chi_h(f)\=\ints hf$ for $h\in L^\infty(\sss)$ and $f\in
L^1(\sss)$. Thus the weak topology on $L^1(\sss)$ (or a subset of it) is the
topology generated by $\chi_h$, $h\in L^\infty(\sss)$.
Recall further that a subset $\cH$ of a topological vector
space is \emph{total} if the set of linear combinations of elements of $\cH$
is dense in the space.

\begin{lemma}\label{Lweak}
  \begin{thmenumerate}
  \item 
Let $\cH$ be a total set in $L^1(\sss)$, and let $\cX$ be a subset of
$\lsmu$ consisting of uniformly bounded functions:
$\sup_{f\in\cX}\normoo{f}<\infty$. 
Then the functionals \set{\chi_h:h\in\cH} generate the weak topology on
$\cX$. 
  \item 
Let $\cH$ be a total set in $L^1(\sssq)$. 
Then the functionals \set{\chi_h:h\in\cH} generate the weak topology on
$\cW(\sss)$. 
\end{thmenumerate}
\end{lemma}

\begin{proof}
\pfitem{i}
  Let $\tau_{\cH}$ be the topology on $\cX$ generated by
  \set{\chi_h:h\in\cH}, and let $\cH'$ be the set of all $g\in L^1(\sss)$
  such that $\chi_g$ is continuous $(\cX,\tau_{\cH})\to\bbR$. 
By the definition of $\tau_{\cH}$, $\cH\subseteq\cH'$; further, $\cH'$ and $\cH$
generate the same topology, \ie, $\tau_{\cH}=\tau_{\cH'}$.

$\cH'$ is clearly a linear subspace of $L^1(\sss)$, and since we have
assumed that $\cH$ is total, $\cH'$ is dense in $L^1(\sss)$. 
If $g\in L^1(\sss)$, there thus exists a sequence $g_n\in\cH'$ with
$\normll{g_n-g}\to0$. Since the functions in $\cX$ are uniformly bounded, this
means that $\chi_{g_n}\to\chi_g$ uniformly on $\cX$, and thus $\chi_g$ too
is $\tau_{\cH}$-continuous; hence $g\in\cH'$. Consequently,
$\cH'=L^1(\sss)$, and thus
$\tau_{\cH}=\tau_{\cH'}=\tau_{L^1(\sss)}$. Thus every total $H\subseteq
L^1(\sss)$ generates the same topology. One such $\cH$ is $L^\infty(\sss)$
which defines the weak topology (by definition).

\pfitem{ii}
This is a special case, since $\sssq$ is another probability space.
\end{proof}

In particular, the weak topology on $\cW(\sss)$ is also the topology generated by the
functionals $W\mapsto\int_{\sssq}hW$, $h\in L^1(\sssq)$, \ie, it equals the
\emph{weak${}^*$ topology} on $\cW(\sss)$, regarded as a subset of
$L^{\infty}(\sssq)$. 

\begin{remark}\label{Rweak}
  Another example of a total set in $L^1(\sssq)$ is the set of
  rectangle indicators $\etta_S(x)\etta_T(y)$ for $S,T\subseteq\sss$. Thus
  the weak topology is also generated by the functionals
  $W\mapsto\int_{S\times T}W$. Note that the metric given by $\cn\fyll$ uses
  the same functionals, but with an important difference:
$\cn{W_n-W}\to0$ if and only if $\int_{S\times T}W_n\to\int_{S\times T}W$
  \emph{uniformly} for all $S,T\subseteq\sss$, 
while $W_n\to W$ in the weak topology if and only if each
$\int_{S\times T}W_n\to\int_{S\times T}W$, without any uniformity
requirement.
(Similarly, $\normll{W_n-W}\to0$ if and only if $\int hW_n\to\int hW$
uniformly for all $h$ with $\normoo{h}\le1$.)
\end{remark}

\begin{lemma}
  The weak topology is weaker than the cut norm topology.
I.e., the identity maps
$(\cW,\normll\,)\to(\cW,\cn\,)\to(\cW,\gs)$ are continuous.
\end{lemma}

\begin{proof}
  Immediate by \refR{Rweak}.
\end{proof}

\begin{theorem}
  The topological space $(\cW(\sss),\gs)$ is compact.
\end{theorem}
\begin{proof}
  $\cW$ is a weak${}^*$ closed subset of the unit ball of 
$L^{\infty}(\sssq)=L^{1}(\sssq)^*$, so this follows from the Banach--Alaoglu
  theorem. 
\end{proof}

One advantage with the weak topology is thus that it is compact,
in contrast to the
topologies defined by the norms $\cn\,$ and
$\normll\,$ which are not compact (in general,
\eg{} if $\sss=\oi$), see \refE{Eweak} below.
(Recall that, nevertheless, the
quotient space $(\bwx,\dcut)$ is compact, and that this is a very important
property.) 

However, a serious drawback with the weak topology is that the quotient map
$\cW(\sss)\to \bwx$ is \emph{not} continuous in the weak topology.
Equivalently, 
the homomorphism densities $t(F,W)$ defined in
\refApp{Ahomo} are \emph{not} continuous in the weak topology (for every
fixed $F$). More precisely, for example $W\mapsto t(K_3,W)$ is not continuous
in the weak topology on $\cW(\oi)$, see \refE{Eweak}. 

\begin{remark}
There are graphs $F$ such that $W\mapsto t(F,W)$ is weakly
continuous (\ie, continuous for $\gs$), 
for example $K_2$ since $t(K_2,W)=\int_\sssq W$.
We show in \refL{Lweakcont} below that 
$K_2$ is
essentially the only such exceptional case.
\end{remark}

\begin{example}
  \label{Eweak}
Take $\sss=\oi$.
Let $g_n(x)=\sgn(\sin(2\pi n x))$ and $W_n(x,y)=\frac12-\frac12g_n(x)g_n(y)$.
Then $g_n(x)\in\set{\pm1}$ and $W_n$ is
\oivalued; in fact, $W_n$ equals $\wgv{K_{n,n}}$ for a complete bipartite
graph $K_{n,n}$.
(A less combinatorial alternative is to take $g_n(x)=\sin(2\pi n x)$.) 

We have $g_n=g_1^{\gf_n}$ and $W_n=W_1^{\gf_n}$, where $\gf_n(x)=nx\mod 1$ as
in \refE{Ebadmn}.
Consequently, $W_n\equ W_1$, and thus 
$W_n= W_1$ in the quotient space $\bwx$,
\ie{} $\dcut(W_n,W_1)=0$; in particular, $W_n\to W_1$ in $(\bwx,\dcut)$.

On the other hand, for any $h\in L^1(\oiq)$,
$\int_{\oiq}h(x,y)g_n(x)g_n(y)\to0$, 
and thus $W_n\to\frac12$ in $(\cW(\oi),\gs)$.

If the quotient map $\cW(\oi)\to\bwx$ were continuous for $\gs$, then
$W_n\to\frac12$ in $\bwx$, and since we already know $W_n\to W_1$ in $\bwx$,
we would have $W_1=\frac12$ in $\bwx$, \ie, $W_1\equ\frac12$, which
contradicts \eg{} \refC{Cdistr}. Consequently, the quotient map is
\emph{not} continuous $(\cW(\oi),\gs)\to(\bwx,\dcut)$.

This also shows that $(\cW(\oi),\cn\fyll)$ and, a fortiori,
$(\cW(\oi),\normll\fyll)$ 
are not compact. Indeed, if one of these spaces were compact, then $W_n$
would have a convergent subsequence in it, and thus in $(\cW,\cn\fyll)$,
with a limit $W$ say. Since both maps $(\cW,\cn\fyll)\to(\cW,\gs)$ and
$(\cW,\cn\fyll)\to(\bwx,\dcut)$ are continuous, the subsequence  would
converge to $W$ in both $(\cW,\gs)$ and $(\bwx,\dcut)$ too; hence both
$W=\frac12$ \aex{} and $W\equ W_1$, so again $W_1\equ\frac12$, a
contradiction. 

Furthermore, 
with $W=\frac12$, so $W_n\to W$ weakly,
$t(K_3,W_n)=0$, while $t(K_3,W)=\frac18>0$; 
hence, $t(K_3,W)$ is not weakly continuous.
\end{example}

\begin{lemma}
  \label{Lweakcont}
The map $W\mapsto t(F,W)$ 
is weakly continuous (for $\sss=\oi$, say)
if and only if $F$ is a disjoint union of isolated
vertices and edges.
\end{lemma}

\begin{proof}
  Let $F$ have $m$ vertices and $e$ edges.
If every component of $F$ is a vertex or an edge, then
$t(F,W)=\bigpar{\int_\sssq W}^e$, which is weakly continuous.

Conversely, suppose that
$F$ is a graph such that $W\mapsto t(F,W)$ is weakly continuous.
Let $\ga\in(0,1/2)$ be rational and let 
$W_n\=\wgv{G_n}$, where
$G_n$ is the
complete bipartite graph
$K_{{\ga n},{n-\ga n}}$ (for $n$ such that $\ga n$ is an integer).
Taking the vertices of $G_n$ in suitable (\eg{} random) order, we have
$W_n\to W$ weakly, where $W=2\ga(1-\ga)$ is a constant graphon.
Thus, by assumption, $t(F,W_n)\to t(F,W)$.

If $F$ is not bipartite, then $t(F,W_n)=t(F,G_n)=0$, while $t(F,W)>0$, a
contradiction. 

If $F$ is bipartite, suppose first that $F$ is connected, so
$e\ge m-1$ edges. Then $F$ has a bipartition where the smallest
part has $k\le m/2$ vertices, and thus 
\begin{equation}
t(F,W_n)=t(F,G_n)\ge \ga^k(1-\ga)^{m-k}
\ge 2^{-m} \ga^{m/2},
\end{equation}
while 
\begin{equation}
  t(F,W)=\bigpar{2\ga(1-\ga)}^e \le 2^e\ga^{m-1}.
\end{equation}
If $m\ge3$, then $m/2<m-1$, and thus we can choose $\ga$ so small that
$t(F,W_n)>2t(F,W)$ for all $n$, a contradiction. Hence $m\le2$.

If $F$ is bipartite and disconnected, we use the same argument for every
component of $F$, noting that $t(F,W_n)=t(F,W)$ if $F$ has at most two
vertices. It follows that no component of $F$ can have more than two vertices.
\end{proof}

See \citet{ChatterjeeV} for a recent application of the weak topology on $\cW$.

\section{Separability in Lebesgue spaces}\label{Asep}

In many cases, the Banach space
$\lsfmu$ is separable.
For example, this is the case if $\sss=\oi$ with any Borel measure $\mu$.
(One example of a countable dense set is the set of polynomials with
rational coefficients; this is dense \eg{} by the monotone class theorem
\cite[Theorem  A.1]{SJII}.)
Hence, by \refT{Tborel}, $\lsfmu$ is separable for every Borel probability space
$(\sss,\cF,\mu)$.  This includes almost all examples used in graph limit
theory. 

However,
there are cases when $\lsfmu$ is
  non-separable.
For example, this is the case when $(\sss,\mu)$ is an uncountable product
  $(\oi,\nu)^{\bbR}$ or  $(\setoi,\nu)^{\bbR}$, with $\nu$ the
  uniform distribution, say. (Any uncountable product of non-trivial spaces
  will do.) 
 In this case there are some technical difficulties and we sometimes
have to be more careful. 

Recall that the elements $f$ of $\lsfmu$ formally are equivalence classes of
functions, so to define pointwise values $f(x)$ we have to make a choice of
representative of $f$. This is usually harmless, but it may be a serious problem
if we want to define $f(x)$ for many $f$ simultaneously, in particular if we
want to define a measurable evaluation map $(f,x)\mapsto f(x)$ on
$\lsmu\times\sss\to\bbR$.  

The following lemma shows that this is possible when $\lsmu$ is separable,
and more generally on $A\times \sss$ when $A\subseteq\lsmu$ is a separable
subspace. Note, however, that there is \emph{no} such measurable evaluation
map in general, without separability assumption, see \refE{Enonsep} below.
This justifies stating and proving the lemma carefully, although it may look
obvious. 

\begin{lemma}\label{Leval}
  If $A$ is a closed separable subspace of $\lsfmu$, then there is a
  measurable function $\Phi:A\times\sss\to\bbR$ such that for every $f\in
  A$, $\Phi(f,x)=f(x)$ for \aex{} $x\in\sss$.
\end{lemma}

\begin{proof}
There exists a countable dense set $D\subset A$. 
Each element of $D$ is an element of $\lsfmu$, \ie, an equivalence class
of measurable functions on $\sss$; we fix one representative for each
element of $D$ and regard the elements of $D$ as these fixed functions.
Write $D= \set{d_1,d_2,\dots}$ with some arbitrary ordering of the elements.

Since $D$ is dense in $A$, we may recursively define maps $H_i:A\to D$
such that 
\begin{equation}\label{kof}
  \Bignorm{f-\sum_{i=1}^k H_i(f)}_{L^1} \le 2^{-k},
\qquad k\ge1,
\end{equation}
by defining $H_k(f)$ as the first element of
$D$ that satisfies \eqref{kof}. Then each $H_i:A\to D$ is measurable.
Further, \eqref{kof} implies $\normll{H_i(f)}\le 3\cdot 2^{-i}$ for $i\ge2$, so
$\ints\sumi|H_i(f)|\dd\mu=\sumi\normll{H_i(f)}<\infty$ for every $f\in A$,
which implies that $\sumi H_i(f)(x)$ converges absolutely \aex.
Moreover, \eqref{kof} implies by dominated convergence 
$\normll{f-\sumi H_i(f)(x)}=0$, so
$\sumi H_i(f)(x)=f$ a.e.
We now define
\begin{equation*}
  \Phi(f,x)\=
  \begin{cases}
	\sumi H_i(f)(x), & \text{if the sum converges;}\\
0 & \text{otherwise}.
  \end{cases}
\end{equation*}
Each map $(f,x)\mapsto H_i(f)(x)$ is measurable, and thus $\Phi$ is measurable.
\end{proof}

\begin{example}
  \label{Enonsep}
Let $\sss_0$ be the two-point set \setoi, with uniform measure
$\mu_0\set0=\mu_0\set1=1/2$, and let $(\sss,\mu)$ be the uncountable product
$(\sss_0,\mu_0)^{\bbR}$. Any measurable function 
$\Phi:\lsmu\times\sss\to\bbR$ depends only on countably many coordinates
in $\lsmu\times\sss=\lsmu\times\sss_0^{\bbR}$, \ie, there is a countable set
$C\subset\bbR$ such that if $x=(x_r)_{r\in\bbR}$ and $y=(y_r)_{r\in\bbR}$ are
elements of $\sss=\sss_0^\bbR$ with $x_r=y_r$ for $r\notin C$, then 
\begin{equation}
  \label{sjwf}
\Phi(f,x)=\Phi(f,y) \qquad\text{for all $f\in\lsmu$}. 
\end{equation}
Fix $s\notin C$ and define
$\gs:\sss\to\sss$ by $\gs:(x_r)_r\mapsto(x'_r)_r$ with $x'_r=x_r$ for $r\neq
s$ and $x'_s=1-x_s$; note that $\gs$ is \mpp. 
By \eqref{sjwf},  $\Phi(f,\gs(x))=\Phi(f,x)$ for every $f$ and
$x\in\sss$. If $\Phi(f,x)=f(x)$ for \aex{} $x$, then thus
$f(x)=f(\gs(x))$ for \aex{} $x$, which obviously is incorrect for the
coordinate function $f(x)=x_s$.

Consequently, there exists no measurable evaluation map
$\Phi:\lsmu\times\sss\to\bbR$ such that $\Phi(f,x)=f(x)$ for 
every $f$ and
\aex{} $x$.
\end{example}

In fact, it can be shown (again using the monotone class theorem) that if
$A$ is any measurable space and $\Phi:A\times\sss\to\bbR$ is measurable and
such that $x\mapsto \Phi(\ga,x)\in\lsfmu$ for every $\ga\in A$, then these
function all lie in some separable subspace of $\lsfmu$. This shows that the
condition in \refL{Leval} that $A$ be separable is both necessary and
sufficient for the conclusion of the lemma.

\newcommand\AAP{\emph{Adv. Appl. Probab.} }
\newcommand\JAP{\emph{J. Appl. Probab.} }
\newcommand\JAMS{\emph{J. \AMS} }
\newcommand\MAMS{\emph{Memoirs \AMS} }
\newcommand\PAMS{\emph{Proc. \AMS} }
\newcommand\TAMS{\emph{Trans. \AMS} }
\newcommand\AnnMS{\emph{Ann. Math. Statist.} }
\newcommand\AnnPr{\emph{Ann. Probab.} }
\newcommand\CPC{\emph{Combin. Probab. Comput.} }
\newcommand\JMAA{\emph{J. Math. Anal. Appl.} }
\newcommand\RSA{\emph{Random Struct. Alg.} }
\newcommand\ZW{\emph{Z. Wahrsch. Verw. Gebiete} }
\newcommand\DMTCS{\jour{Discr. Math. Theor. Comput. Sci.} }

\newcommand\AMS{Amer. Math. Soc.}
\newcommand\Springer{Springer-Verlag}
\newcommand\Wiley{Wiley}

\newcommand\vol{\textbf}
\newcommand\jour{\emph}
\newcommand\book{\emph}
\newcommand\inbook{\emph}
\def\no#1#2,{\unskip#2, no. #1,} 
\newcommand\toappear{\unskip, to appear}

\newcommand\arxiv[1]{\url{arXiv:#1.}}
\newcommand\arXiv{\arxiv}

\def\nobibitem#1\par{}

\end{document}